\documentclass[onefignum,onetabnum]{siamart220329}

\usepackage{lipsum}
\usepackage{amsfonts}
\usepackage{tikz-cd}  
\usepackage{amsmath,booktabs,ctable,threeparttable,boxedminipage}
\usepackage{amssymb,amsfonts,boxedminipage}
\usepackage{graphicx,epstopdf} 
\usepackage{changes}
\usepackage{amsopn}

\usepackage[caption=false]{subfig} 
\usepackage{algorithm}
\usepackage{algpseudocode}
\usepackage{qbordermatrix}
\usepackage{blkarray}
\usepackage{lineno}
\usepackage{xcolor}
\usepackage{cite}

\ifpdf
  \DeclareGraphicsExtensions{.eps,.pdf,.png,.jpg}
\else
  \DeclareGraphicsExtensions{.eps}
\fi

\newsiamremark{remark}{Remark}
\newsiamremark{hypothesis}{Hypothesis}
\crefname{hypothesis}{Hypothesis}{Hypotheses}
\newsiamthm{claim}{Claim}

\newtheorem{remark1}{Remark}
\newtheorem{assumption}{Assumption}
\newtheorem{routine}{Routine}

\headers{Projected Newton method}{H. Li}

\title{Projected Newton method for large-scale Bayesian linear inverse problems 
	\thanks{Submitted to the editors DATE.
  }
}

\author{Haibo Li
\thanks{School of Mathematics and Statistics, The University of Melbourne, Parkville, VIC 3010, Australia. 
  (\email{haibo.li@unimelb.edu.au}).}
}

\algrenewcommand\algorithmicrequire{\textbf{Input:}}
\algrenewcommand\algorithmicensure{\textbf{Output:}}

\newcommand\argmin{\mathop{\mathrm{argmin}}}
\newcommand\bA{\boldsymbol{A}}
\newcommand\bb{\boldsymbol{b}}
\newcommand\bx{\boldsymbol{x}}
\newcommand\bI{\boldsymbol{I}}

\newcommand\bC{\boldsymbol{C}}
\newcommand\bL{\boldsymbol{L}}
\newcommand\bU{\boldsymbol{U}}
\newcommand\bZ{\boldsymbol{Z}}
\newcommand\bV{\boldsymbol{V}}

\newcommand\bM{\boldsymbol{M}}
\newcommand\bN{\boldsymbol{N}}
\newcommand\bB{\boldsymbol{B}}

\newcommand\bD{\boldsymbol{D}}

\newcommand\bp{\boldsymbol{p}}

\newcommand\br{\boldsymbol{r}}
\newcommand\bs{\boldsymbol{s}}
\newcommand\bz{\boldsymbol{z}}
\newcommand\bu{\boldsymbol{u}}
\newcommand\bv{\boldsymbol{v}}
\newcommand\bw{\boldsymbol{w}}
\newcommand\by{\boldsymbol{y}}
\newcommand\be{\boldsymbol{e}}
\newcommand\bd{\boldsymbol{d}}

\newcommand\bSigma{\boldsymbol{\Sigma}}

\newcommand\bepsilon{\boldsymbol{\epsilon}}

\newcommand\calL{\mathcal{L}}

\numberwithin{equation}{section}
\numberwithin{figure}{section}
\numberwithin{table}{section}

\newenvironment{proofof}[1]{\begin{trivlist}
  \item[\hskip\labelsep{\bf Proof of {#1}.}]}{$\hfill\Box$\end{trivlist}}

\ifpdf
\hypersetup{
  pdftitle={Projected Newton method for large-scale Bayesian linear inverse problems},
  pdfauthor={H. Li}
}
\fi

\begin{document}
\maketitle

\begin{abstract}
	Computing the regularized solution of Bayesian linear inverse problems as well as the corresponding regularization parameter is highly desirable in many applications. This paper proposes a novel iterative method, termed the \textit{Projected Newton method} (\textsf{PNT}), that can simultaneously update the regularization parameter and solution step by step without requiring any expensive matrix inversions or decompositions. By reformulating the Tikhonov regularization as a constrained minimization problem and leveraging its Lagrangian function, a Newton-type method coupled with a Krylov subspace method is designed for the unconstrained Lagrangian function. The resulting \textsf{PNT} algorithm only needs solving a small-scale linear system to get a descent direction of a merit function at each iteration, thus significantly reducing computational overhead. Rigorous convergence results are proved, showing that \textsf{PNT} always converges to the unique regularized solution and the corresponding Lagrangian multiplier.  Experimental results on both small and large-scale Bayesian inverse problems demonstrate its excellent convergence property, robustness and efficiency. Given that the most demanding computational tasks in \textsf{PNT} are primarily matrix-vector products, it is particularly well-suited for large-scale problems.
\end{abstract}



\section{Introduction}
Inverse problems arise in various scientific and engineering fields, where the aim is to recover unknown parameters or functions from noisy observed data. Applications include image reconstruction, computed tomography, medical imaging, geoscience, data assimilation and so on \cite{Kaip2006,Hansen2006,Buzug2008,Law2015,Richter2016}. A linear inverse problem of the discrete form can be written as
\begin{equation}\label{inverse1}
	\bb = \bA\bx+\bepsilon, 
\end{equation}
where $\bx\in\mathbb{R}^{n}$ is the underlying quantity to reconstruct, $\bA\in\mathbb{R}^{m\times n}$ is the discretized forward model matrix, $\bb\in\mathbb{R}^{m}$ is the vector of observation with noise $\bepsilon$. We assume that the distribution of $\bepsilon$ is known, which follows a zero mean Gaussian distribution with positive definite covariance matrix $\bM$, i.e., $\bepsilon\sim\mathcal{N}(\boldsymbol{0}, \bM)$. A big challenge for reconstructing a good solution is the ill-posedness of inverse problems, which means that there may be multiple solutions that fit the observation equally well, or the solution is very sensitive with respect to observation perturbation.

To overcome the ill-posedness, regularization is a commonly used technique. From a Bayesian perspective \cite{Kaip2006,Stuart2010}, this corresponds to adding a prior distribution of the desired solution to constrain the set of possible solutions to improve stability and uniqueness. By treating $\bx$ and $\bb$ as random variables, the observation vector $\bb$ has a conditional probability density function (pdf) of the form
$
	p(\bb|\bx) \propto \exp\left(-\frac{1}{2}\|\bA\bx-\bb\|_{\bM^{-1}}^{2}\right).
$
To get a regularized solution, this paper considers a Gaussian prior about the desired solution with the form $\bx\sim\mathcal{N}(\boldsymbol{0}, \mu^{-1}\bN)$, where $\bN$ is a positive definite covariance matrix. Then the Bayes' formula leads to
\begin{equation*}
	p(\bx|\bb,\lambda) 
	\propto p(\bx|\lambda)p(\bb|\bx)
	\propto \exp\left(-\frac{1}{2}\|\bA\bx-\bb\|_{\bM^{-1}}^{2}-\frac{\mu}{2}\|\bx\|_{\bN^{-1}}^2\right),
\end{equation*}
where $\|\bx\|_{\bB}:=(\bx^{\top}\bB\bx)^{1/2}$ is the $\bB$-norm of $\bx$ for a positive definite matrix $\bB$. Maximize the posterior pdf $p(\bx|\bb,\lambda)$ leads to the Tikhonov regularization problem
\begin{equation}\label{Bayes1}
	\min_{\bx \in \mathbb{R}^{n}}\{\|\bA\bx-\bb\|_{\bM^{-1}}^{2} + \mu\|\bx\|_{\bN^{-1}}^2\},
\end{equation}
where the regularization term $\mu\|\bx\|_{\bN^{-1}}^2$ enforces extra structure on the solution that comes from the prior distribution of $\bx$. 

The parameter $\mu$ in the Gaussian prior $\mathcal{N}(\boldsymbol{0}, \mu^{-1}\bN)$ is crucial for obtaining a good regularized solution, which controls the trade-off between the data-fit term and regularization term. There is tremendous effort in determining a proper value of $\mu$. For the standard 2-norm problem, i.e. $\bM=\bI$ and $\bN=\bI$, the classical parameter-selection methods include the L-curve criterion\cite{Hansen1992}, generalized cross-validation \cite{Golub1979}, unbiased predictive risk estimation \cite{renaut2019unbiased} and discrepancy principle \cite{Morozov1966}. There are also some iterative methods based on solving a nonlinear equation of $\mu$; see e.g. \cite{mead2008newton,bazan2008fixed,reichel2008new,gazzola2020krylov}. However, the aforementioned methods can not be directly applied to \cref{Bayes1}. A common procedure needs to first transform \cref{Bayes1} into standard 2-norm form
\begin{equation}\label{gen_regu}
	\min_{\bx \in \mathbb{R}^{n}}\{\|\bL_{M}(\bA\bx-\bb)\|_{2}^{2} + \mu\|\bL_{N}\bx\|_{2}^2\},
\end{equation}
where $\bM^{-1}=\bL_{M}^{\top}\bL_{M}$ and $\bN^{-1}=\bL_{N}^{\top}\bL_{N}$ are the Cholesky factorizations, and then apply the parameter-selection methods. This procedure needs the matrix inversions of $\bM$ and $\bN$ as well as the Cholesky factorizations of $\bM^{-1}$ and $\bN^{-1}$. For large-scale matrices, these two types of computations are almost impossible or extremely expensive. 

For large-scale problems, there exist some iterative regularization methods that can avoid choosing $\mu$ in advance. A class of commonly used iterative methods is based on Krylov subspace \cite{liesen2013krylov}, where the original linear system is projected onto lower-dimensional subspaces to become a series of small-scale problems \cite{Oleary1981,Gazzola2015,jia2023joint,li2024joint}. For dealing with the general-form Tikhonov regularization term $\|\bL_{N}\bx\|_{2}^2$, some recent Krylov iterative methods include \cite{morigi2007orthogonal,Kilmer2007Projection,Reichel2012tikhonov,huang2019choice,Li2023} and so on. When the Cholesky factor $\bL_{N}$ is not accessible, a key difficulty is dealing with the prior covariance $\bN$, which means that the subspaces should be constructed elaborately such that the prior information of $\bx$ can be effectively incorporated into these subspaces \cite{Calvetti2018,li2023subspace}. Such methods have been proposed in \cite{calvetti2005priorconditioners,calvetti2017priorconditioned,Calvetti2018}, where a statistically inspired priorconditioning technique is used to whiten the noise and the desired solution. However, these methods still require large-scale matrix inversions and Cholesky factorizations, which prohibits their applications to large-scale problems. 

Recently, there are several Krylov methods for directly solving \cref{inverse1} without choosing $\mu$ in advance and can avoid the matrix inversions and Cholesky factorizations \cite{Chung2017,li2023subspace}. These methods use the generalized Golub-Kahan bidiagonalization (\textsf{gen-GKB}), which can iteratively reduce the original large-scale problem to small-scale ones and generate Krylov subspaces that effectively incorporate the prior information of $\bx$ encoded by $\bN$. In \cite{li2023subspace}, the regularization effect of the proposed method comes from early stopping the iteration, where the iteration number plays the role of the regularization parameter, while in \cite{Chung2017}, the authors proposed a hybrid regularization method that simultaneously computes the regularized parameter and solution step by step. Although these two methods are very efficient for large-scale problems, there may be some issues in certain situations. The method in \cite{li2023subspace} only computes a good regularized solution but not a good $\mu$. However, in some applications, we need an accurate estimate of $\mu$ to get the posterior distribution of $\bx$ for sampling and uncertainty quantification \cite{Stuart2010,saibaba2020efficient,garbuno2020interacting}. For the hybrid method in \cite{Chung2017}, the convergence property does not have a solid theoretical foundation, and it has been numerically found that the method sometimes does not converge to a good solution, which is a common potential flaw for hybrid methods \cite{Chungnagy2008,Renaut2017}. 

Many optimization methods have been proposed for inverse problems, particularly those stemming from image processing that leads to total variation regularization or $\ell_p$ regularization. These methods include the Bregman iteration \cite{osher2005iterative,yin2008bregman,goldstein2009split}, iterative shrinkage thresholding \cite{daubechies2004iterative,beck2009fast}, and many others \cite{afonso2010augmented,tian2016linearized,luiken2021relaxed}. However, these methods either need a good parameter $\mu$ in advance or can not well deal with $\bM^{-1}$ and $\bN^{-1}$. In \cite{landi2008lagrange} the author proposed a modification of the Newton method that can iteratively compute a good $\mu$ and regularized solution simultaneously. However, this method needs to solve a large-scale linear system at each iteration, which is very costly for large-scale problems. This method was improved in \cite{cornelis2020projected1,cornelis2020projected2}, where the Newton method is successfully combined with a Krylov subspace method to get a so-called projected Newton method. Compared with the original method, the projected Newton method only needs to solve a small-scale linear system at each iteration, thereby very efficient for large-scale Tikhonov regularization \cref{gen_regu}. However, for solving \cref{Bayes1}, this method needs to compute $\nabla(\frac{1}{2}\|\bx\|_{\bN^{-1}}^2)=\bN^{-1}\bx$ to construct subspaces, which is also very costly. Besides, their methods lack rigorous proof of convergence.

In this paper, we develop a new efficient iterative method for \cref{Bayes1} that simultaneously updates the regularization parameter and solution step by step, and it does not require any expensive matrix inversions or Cholesky factorizations. This method follows the Newton-type approach for noise constrained Tikhonov regularization proposed in \cite{cornelis2020projected1}, where the \textsf{gen-GKB} process is integrated to compute a projected Newton direction by solving a small-scale linear system at each iteration, thereby it is also named the \textit{projected Newton method} (\textsf{PNT}). The main contributions of this paper are listed as follows:
\begin{itemize}
    \item We reformulate the regularization of the original Bayesian linear inverse problem as a noise constrained minimization problem and prove the existence, uniqueness and positivity of its Lagrangian multiplier $\lambda$ under a very reasonable assumption. The correspondence between the constrained minimization problem and Tikhonov regularization \cref{Bayes1} is connected by $\mu=1/\lambda$.
    \item We propose a \textsf{gen-GKB} based Newton-type method to compute the regularized solution by optimizing its Lagrangian function and obtaining the corresponding Lagrangian multiplier. A series of Krylov subspaces is generated by \textsf{gen-GKB}, avoiding the need for costly matrix inversions or Cholesky factorizations. Using the subspace projection technique, we only need to solve a small-scale linear system to compute the descent direction at each iteration.
    \item A rigorous proof of convergence for the proposed method is provided. With a very practical initialization $(\bx_0,\lambda_0)$, we prove that \textsf{PNT} always converges to the unique solution of the constrained minimization problem and the corresponding Lagrangian multiplier.
\end{itemize}

We use both small-scale and large-scale inverse problems to test the proposed method and compare it with other state-of-the-art methods. The experimental results demonstrate excellent convergence properties of \textsf{PNT}, and it is very robust and efficient for regularizing Bayesian linear inverse problems. Since the most computationally intensive operations in \textsf{PNT} primarily involve matrix-vector products, it is especially appropriate for large-scale problems.

This paper is organized as follows. In \Cref{sec2}, we formulate the noise constrained minimization problem for regularizing \cref{inverse1} and study its properties. In \Cref{sec3}, we propose the \textsf{PNT} method. In \Cref{sec4}, we prove the convergence  of  \textsf{PNT}. Numerical results are presented in \Cref{sec5} and conclusions are provided in \Cref{sec6}.


\section{Noise constrained minimization for Bayesian inverse problems} \label{sec2}
In order to get a good estimate of $\mu$ in \cref{Bayes1}, the discrepancy principle (DP) criterion is commonly used, which depends on the variance of the noise. Based on DP, we can rewrite \cref{Bayes1} as an equivalent form of noise constrained minimization problem.

\subsection{Noise constrained minimization}
If $\bepsilon\sim\mathcal{N}(\mathbf{0},\sigma^2\bI)$ is a white Gaussian noise, the DP criterion states that the $2$-norm discrepancy between the data and predicted output $\|\bA\bx({\mu})-\bb\|_2$ should be of the order of $\|\bepsilon\|_{2}\approx \sqrt{m}\sigma$, where $\bx(\mu)$ is the solution to \cref{Bayes1}; see \cite[\S 5.6]{Kaip2006}. If $\bepsilon$ is a general Gaussian noise, notice that \cref{inverse1} leads to
$
	\bL_{M}\bb = \bL_{M}\bA\bx + \bL_{M}\bepsilon,
$
and $\bL_{M}\bepsilon\sim\mathcal{N}(\mathbf{0}, \bI)$, thereby this transformation whitens the noise. Since $\bar{\bepsilon}:=\bL_{M}\bepsilon$ is a white Gaussian noise with zero mean and covariance $\bI$, it follows that
\[
	\mathbb{E}\left[\|\bar{\bepsilon}\|_{2}^{2}\right] = \mathbb{E}\left[\mathrm{trace}\left(\bar{\bepsilon}^{\top}\bar{\bepsilon}\right)\right]
	= \mathbb{E}\left[\mathrm{trace}\left(\bar{\bepsilon}\bar{\bepsilon}^{\top}\right)\right]
	= \mathrm{trace}\left(\mathbb{E}\left[\bar{\bepsilon}\bar{\bepsilon}^{\top}\right]\right)
	= \mathrm{trace}\left(\bI\right) = m.
\]
Therefore, the DP for \cref{inverse1} can be written as 
\begin{equation}\label{DP0}
	\|\bA\bx({\mu})-\bb\|_{\bM^{-1}}^2=\|\bL_{M}\bA\bx({\mu})-\bL_{M}\bb\|_{2}^2 = \tau m ,
\end{equation}
where $\tau$ is chosen to be marginally greater than $1$, such as $\tau=1.01$.

Using this expression of DP, we rewrite the regularization of \cref{inverse1} as the noise constrained minimization problem 
\begin{equation}\label{discrepancy}
	\min_{\bx\in\mathbb{R}^{n}}\frac{1}{2}\|\bx\|_{\bN^{-1}}^{2} \ \ \ \mathrm{s.t.} \ \ \
	\frac{1}{2}\|\bA\bx-\bb\|_{\bM^{-1}}^{2} \leq \frac{\tau m}{2} ,
\end{equation}
where its Lagrangian is
\begin{equation}\label{lagr}
	\calL(\bx,\lambda)=\frac{1}{2}\|\bx\|_{\bN^{-1}}^{2}+\frac{\lambda}{2}\left(\|\bA\bx-\bb\|_{\bM^{-1}}^{2}-\tau m\right)
\end{equation}
with $\lambda\geq 0$ the Lagrangian multiplier. To further investigate \cref{discrepancy,lagr}, we first state the following basic assumption, which is used throughout the paper.

\begin{assumption}\label{assump1}
	For all $\bx\in \{\bx\in\mathbb{R}^{n}:\|\bA\bx-\bb\|_{\bM^{-1}}=\min\}$, it holds 
	\begin{equation}\label{assum_ineq}
		\|\bA\bx-\bb\|_{\bM^{-1}}^{2} < \tau m < \|\bb\|_{\bM^{-1}}^{2}.
	\end{equation}
\end{assumption}
The first inequality means that the naive solutions to \eqref{inverse1} fit the observation very well, and it ensures the feasible set of \cref{discrepancy} is nonempty. The second inequality comes from the condition $\|L_{\bM}\bepsilon\|_{2}<\|L_{\bM}\bb\|_{2}$, meaning that the noise does not dominate the observation, which ensures the effectiveness of the regularization. Under this assumption, the following result describes the solution to \cref{discrepancy}.

\begin{theorem}\label{thm:lagr}
	The noise constrained minimization \cref{discrepancy} has a unique solution $\bx^{*}$ satisfying $\|\bA\bx^{*}-\bb\|_{\bM^{-1}}^{2}=\tau m$. Furthermore, there is a unique $\lambda^{*}>0$, which is the Lagrangian multiplier corresponding to $\bx^*$ in \cref{lagr}.
\end{theorem}
\begin{proof}
	Let $\varphi(\bx):=\frac{1}{2}(\|\bA\bx-\bb\|_{\bM^{-1}}^{2}-\tau m)$, which is a convex function. In \cref{discrepancy} we seek solutions to $\min\frac{1}{2}\|\bx\|_{\bN^{-1}}^{2}$ in the feasible set $S:=\{\bx\in\mathbb{R}^{n}: \varphi(\bx)\leq 0\}$, which is the $0$-lower level set of $\varphi(\bx)$. Note that $S$ is a compact and convex set and $\frac{1}{2}\|\bx\|_{\bN^{-1}}^{2}$ is continuous and strictly convex. Thus, there is a unique solution $\bx^{*}$ to \cref{discrepancy}.
	Suppose $\lambda^{*}$ is a Lagrangian multiplier corresponding to $\bx^{*}$. By the Karush-Kuhn-Tucker (KKT) condition \cite[\S 12.3]{nocedal2006numerical}, the solution $(\bx^{*}, \lambda^{*})$ satisfies
	\begin{equation*}
		\begin{cases}
		\bN^{-1}\bx^{*}+\lambda^{*}\nabla\varphi(\bx^{*})=\mathbf{0},  \\
		\lambda^{*}\varphi(\bx^{*})=0, \\
		\lambda^{*}\geq 0.
		\end{cases}
	\end{equation*}
	If $\lambda^{*}=0$, then $\bN^{-1}\bx^{*}$, leading to $\bx^{*}=\mathbf{0}$. This means $\mathbf{0}\in S$, i.e. $\|\bb\|_{\bM^{-1}}^{2}\leq\tau m$, a contradiction. Consequently, it must hold $\lambda^{*}>0$. From the relation $\lambda^{*}\varphi(\bx^{*})=0$ we have $\varphi(\bx^{*})=0$, i.e. $\|\bA\bx^{*}-\bb\|_{\bM^{-1}}^{2}=\tau m$.

	For the uniqueness of $\lambda^{*}$, here we give two proofs. In the first proof, we note that $\nabla\varphi(\bx^{*})=\bA^{\top}\bM^{-1}(\bA\bx^*-\bb)\neq\mathbf{0}$ since $\bx^*\notin\{\bx\in\mathbb{R}^{n}:\|\bA\bx-\bb\|_{\bM^{-1}}=\min\}$ by \cref{assump1}. Therefore, the linear independence constraint qualification (LICQ) holds at $\bx^{*}$, which leads to the uniqueness of $\lambda^{*}$; see \cite[\S 12.3]{nocedal2006numerical}.
	
	In the second proof, we note that for any $\lambda\geq 0$, there is a unique $\bx_{\lambda}$ that solves the first equality of the KKT condition:
	\begin{equation}\label{KKT1}
		\bN^{-1}\bx+\lambda\nabla\varphi(\bx)=\mathbf{0}\ \ \Leftrightarrow \ \ 
		(\bN^{-1}+\lambda\bA^{\top}\bM^{-1}\bA)x=\lambda\bA^{\top}\bM^{-1}\bb,
	\end{equation}
	since $\bN^{-1}+\lambda\bA^{\top}\bM^{-1}\bA$ is positive definite. Here we prove a stronger property: \textit{there exist a unique $\lambda\geq 0$ such that $\|\bA\bx_{\lambda}-\bb\|_{\bM^{-1}}^{2}=\tau m$}. The existence of such a $\lambda$ has been proved, since $\bx^{*}=\bx_{\lambda^{*}}$. For the uniqueness, define two functions
	\[
		K(\lambda):= \frac{1}{2}\|\bx_{\lambda}\|_{\bN^{-1}}^{2}, \ \ \ 
		H(\lambda):= \frac{1}{2}\left(\|\bA\bx_{\lambda}-\bb\|_{\bM^{-1}}^{2}-\tau m\right) .
	\]
	Note that $\calL(\bx,\lambda)$ is strictly convex for a fixed $\lambda>0$, which has the unique minimizer $\bx_{\lambda}$. Thus, for any two positive $\lambda_{1}\neq\lambda_{2}$, we have 
	$
		\calL(\bx_{\lambda_1},\lambda_1) < \calL(\bx_{\lambda_2},\lambda_1)   \Leftrightarrow 
		K(\lambda_1)+\lambda_1H(\lambda_1) < K(\lambda_2)+\lambda_1H(\lambda_2) ,
	$
	since $\bx_{\lambda_1}\neq\bx_{\lambda_2}$; see the following \Cref{lem:tikh_sol}. Similarly, we have
	$
		K(\lambda_2)+\lambda_2H(\lambda_2) < K(\lambda_1)+\lambda_2H(\lambda_1) .
	$
	Adding the above two inequalities leads to
	$(\lambda_1-\lambda_2)(H(\lambda_1)-H(\lambda_2)) < 0$,
	meaning that $H(\lambda)$ is a strictly monotonic decreasing function. Therefore, there is a unique $\lambda$ such that $H(\lambda)=0$.
\end{proof}

We emphasize that \Cref{assump1} is essential for ensuring the validity of \Cref{thm:lagr} and plays a key role in the regularization of \cref{inverse1}. If the left inequality of \Cref{assump1} is violated, then either the feasible set of \cref{discrepancy} is empty when $\|\bA\bx-\bb\|_{\bM^{-1}}^{2} > \tau m$ or it becomes the equivalent least squares problem
\[ \min\frac{1}{2}\|\bx\|_{\bN^{-1}}^{2} \ \ \ \mathrm{s.t.} \ \ \
\|\bA\bx-\bb\|_{\bM^{-1}}^{2} =\min  \]
when $\|\bA\bx-\bb\|_{\bM^{-1}}^{2} = \tau m$.
The latter case means that no regularization is required, which makes the problem much easier to handle. The right inequality of \Cref{assump1} ensures the positivity of the Lagrangian multiplier $\lambda^{*}$, which is a necessary condition. To see it, let us assume there exist a $\lambda^{*}>0$ and follow the second proof for the uniqueness of $\lambda^*$. From the KKT condition it must hold that $\varphi(\bx_{\lambda^{*}})=H(\lambda^{*})=0$. Now $H(0)=\frac{1}{2}(\|\bb\|_{\bM^{-1}}^2-\tau m)\leq 0$ and $H(+\infty)=\frac{1}{2}(\|\bA\bx-\bb\|_{M^{-1}}^2-\tau m)<0$ for any $\bx\in\argmin_{\bx\in\mathbb{R}^{n}}\|\bA\bx-\bb\|_{\bM^{-1}}$. Since $H(\lambda)$ is strictly monotonically decreasing, the only possible zero root of $H(\lambda)$ is $\lambda=0$, which leads to a contradiction. In this case, it implies that the noise in $b$ is too large, resulting in $\lambda^{*}=0$ and a very poor regularized solution $\bx^{*}=\mathbf{0}$.

\begin{lemma}\label{lem:tikh_sol}
	For each $\lambda\geq0$, the regularization problem 
	\begin{equation}\label{tikh2}
		\min_{\bx\in\mathbb{R}^{n}}\{\lambda\|\bA\bx-\bb\|_{\bM^{-1}}^{2}+\|\bx\|_{\bN^{-1}}^2\}
	\end{equation}
	has the unique solution $\bx_{\lambda}$. If $\lambda_{1}\neq\lambda_2$, then $\bx_{\lambda_1}\neq\bx_{\lambda_2}$.
\end{lemma}
\begin{proof}
	Note that the normal equation of \cref{tikh2} is equivalent to \cref{KKT1}. Thus, $\bx_{\lambda}$ is the unique solution to \cref{tikh2}. Using the Cholesky factors of $\bM^{-1}$ and $\bN^{-1}$, and noticing that $\bL_{N}$ is invertible, we can write the generalized singular value decomposition (GSVD) \cite{Van1976} of $\{\bL_{M}\bA, \bL_{N}\}$ as
	$
		\bL_{M}\bA = \bU_A\bSigma_A\bZ^{-1}, \ 
		\bL_{N} = \bU_N\bSigma_N\bZ^{-1}
	$
	with
	\begin{equation*}
		\bSigma_A = \bordermatrix*[()]{%
		\bD_{A}  &  &  r \cr
		&  \boldsymbol{0} & m-r \cr
		r & n-r 
	} , \ \ \
	\bSigma_N =
	\bordermatrix*[()]{%
		\bD_{N}  &  & r \cr
		&  \bI  & n-r \cr
		r & n-r
	} ,
	\end{equation*}
	where $\bU_{A}\in\mathbb{R}^{m\times m}$ and $\bU_{N}\in\mathbb{R}^{n\times n}$ are orthogonal, $\bZ=(\bz_1,\dots,\bz_n)$ in nonsingular, $r=\mathrm{rank}(\bA)$, and $\bD_{A}=\mathrm{diag}(\sigma_1,\dots,\sigma_r)$ with $1>\sigma_1\geq\cdots\geq\sigma_r>0$ and $\bD_{N}=\mathrm{diag}(\rho_1,\dots,\rho_r)$ with $0<\rho_1\leq\cdots\leq\rho_r<1$, such that $\sigma_{i}^{2}+\rho_{i}^{2}=1$. Then $\bx_{\lambda}$ can be expressed as $\bx_{\lambda} =\sum_{i=1}^{r}\frac{\lambda\sigma_{i}}{\lambda\sigma_{i}^{2}+\rho_{i}^{2}}(\bu_{A,i}^{\top}\bL_{M}\bb)\bz_i$
	where $\bu_{A,i}$ is the $i$-th column of $\bU_A$. Since $\{\bz_{i}\}_{i=1}^{r}$ are linear independent, if $\bx_{\lambda_1}=\bx_{\lambda_2}$, then it must hold 
	\[
		\frac{\lambda_1\sigma_{i}}{\lambda_1\sigma_{i}^{2}+\rho_{i}^{2}} 
		= \frac{\lambda_2\sigma_{i}}{\lambda_2\sigma_{i}^{2}+\rho_{i}^{2}} \ \  \Leftrightarrow \ \ 
		(\lambda_1-\lambda_2)\sigma_{i}\rho_{i}^2 = 0, \ \ i=1,\dots,r .
	\]
	Since $\sigma_i\rho_{i}>0$ for $i=1,\dots,r$, we obtain $\lambda_1=\lambda_2$.
\end{proof}

\begin{remark1}
	From the proof of \Cref{thm:lagr}, we find that $\lambda$ plays the role of $\mu^{-1}$ in \cref{Bayes1}, meaning that $\bx({\mu})=\bx_{\lambda}$ if $\lambda=\mu^{-1}$. In fact, there is a one-to-one correspondence between \cref{Bayes1} and \cref{discrepancy}. Note that $\bx^{*}= \bx_{\lambda^{*}}$. Comparing \cref{tikh2} with \cref{Bayes1}, we can use $(\lambda^{*})^{-1}$ as a good estimate of the optimal regularization parameter. 
\end{remark1}

\begin{corollary}\label{coro:F}
	Let $\mathbb{R}^{+}=[0,\infty)$. Write the gradient of $\calL(\bx,\lambda)$ as
	\begin{equation}
		F(\bx,\lambda)=
		\begin{pmatrix}
			\lambda\bA^{\top}\bM^{-1}(\bA\bx-\bb)+\bN^{-1}\bx \\
			\frac{1}{2}\|\bA\bx-\bb\|_{\bM^{-1}}^{2}-\frac{\tau m}{2}
		\end{pmatrix} .
	\end{equation}
	Then $F(\bx,\lambda)=\mathbf{0}$ has a unique solution $(\bx^{*},\lambda^{*})$ in $\mathbb{R}^{n}\times\mathbb{R}^{+}$, which is the unique minimizer and corresponding Lagrangian multiplier of \cref{discrepancy}.
\end{corollary}

\subsection{Newton method}
A modification of the Newton method was proposed in \cite{landi2008lagrange} to solve the nonlinear equation $F(\bx,\lambda)=\mathbf{0}$, which is referred to as the Lagrange method since it is based on the Lagrangian of \cref{discrepancy}. In this method, the Jacobian matrix of $F(\bx,\lambda)$ is first computed as
\begin{equation}\label{Jacob}
	J(\bx,\lambda)=
	\begin{pmatrix}
		\lambda\bA^{\top}\bM^{-1}\bA+\bN^{-1} & \bA^{\top}\bM^{-1}(\bA\bx-\bb) \\
		(\bA\bx-\bb)^{\top}\bM^{-1}\bA & 0
	\end{pmatrix}
\end{equation}
at the current iterate $(\bx,\lambda)$, and then it computes the Newton direction $(\Delta\bx^{\top},\Delta\lambda)^{\top}$ by solving inexactly the linear system
\begin{equation}\label{Newton_direc}
	J(\bx,\lambda)\begin{pmatrix}
		\Delta\bx \\ \Delta\lambda
	\end{pmatrix}=-F(\bx,\lambda) 
\end{equation}
using the MINRES solver \cite{paige1975solution}. We remark that this method is essentially a Newton-Krylov method \cite{brown1994convergence} for optimizing the nonlinear and nonconvex Lagrangian function \cref{lagr}. It was shown that the computed $(\Delta\bx,\Delta\lambda)$ is a descent direction for the merit function
\[
	h_{w}(\bx,\lambda)=\frac{1}{2}\left(\|\nabla_{\bx}\calL(\bx,\lambda)\|_{2}^2+w|\nabla_{\lambda}\calL(\bx,\lambda)|^2\right) 
\]
with a $w>0$, which means that $(\Delta\bx^{\top}, \Delta\lambda)\nabla h_{w}(\bx,\lambda)\leq 0$.
By a backtracking line search strategy to determine a step length $\gamma>0$, the iterate is updated as $(\bx,\lambda)\leftarrow (\bx,\lambda)+\gamma(\Delta\bx,\Delta\lambda)$. 

An advantage of this method is that it can compute a good regularized solution and its regularization parameter simultaneously. However, for large-scale problems, we need to compute $\bM^{-1}$ and $\bN^{-1}$ to form $F(\bx,\lambda)$ and $J(\bx,\lambda)$, which is almost impossible. Moreover,  at each iteration, an $(n+1)\times(n+1)$ linear system \eqref{Newton_direc} needs to be solved, which is very computationally expensive even if we only compute a less accurate solution by an iterative algorithm.

In \cite{cornelis2020projected1}, the authors proposed a projected Newton method, where at each iteration, the large-scale linear system \eqref{Newton_direc} is projected to be a small-scale linear system that can be solved cheaply. However, this method can only deal with the standard $\ell_2-\ell_2$ regularization, which means we can only apply this method to \eqref{gen_regu} by the substitution $\bar{\bx}=\bL_{N}\bx$, requiring the expensive Cholesky factorization of $\bN^{-1}$. A generalization of this method \cite{cornelis2020projected2} can deal with a general-form regularization term. However, for \eqref{discrepancy}, it needs to compute $\nabla(\frac{1}{2}\|x\|_{\bN^{-1}}^2)=\bN^{-1}x$ to construct subspace for projecting \eqref{Newton_direc}, also very costly.

\section{Projected Newton method based on generalized Golub-Kahan bidiagonalization} \label{sec3}
To reduce expensive computations of the Newton method for large-scale problems, we design a new projected Newton method to solve \cref{discrepancy}. This method uses the generalized Golub-Kahan bidiagonalization (\textsf{gen-GKB}) to construct Krylov subspaces to compute projected Newton directions by only solving small-scale problems, and it does not need any expensive matrix inversions or decompositions. This method is composed by the following three main steps:
\begin{enumerate}
\item[Step 1:] \emph{Construct Krylov subspaces.} We adopt \textsf{gen-GKB} to iteratively construct a series of low-dimensional Krylov subspaces; see \Cref{alg:GKB}. 
\item[Step 2:] \emph{Compute the projected Newton direction.} At each iteration, we compute the projected Newton direction by solving a small-scale problem; see \cref{proj_direction}.
\item[Step 3:] \emph{Determine the step-length to update solution.} We use the Armijo backtracking line search to determine a step-length and update the solution; see \Cref{rout:backtracking}.
\end{enumerate}

In the next subsection, we present detailed derivations of the whole algorithm. All the proofs can be found in \Cref{subsec3.3}.


\subsection{Derivation of projected Newton method}\label{subsec3.2}
This subsection presents detailed derivations for the above three steps.
\medskip

\noindent\textbf{Step 1: Construct Krylov subspaces by gen-GKB.} 
The \textsf{gen-GKB} process has been proposed for solving Bayesian linear inverse problems in \cite{Chung2017,li2023subspace}. The basic idea is to treat $\bA$ as the compact linear operator
\[
	\bA: (\mathbb{R}^{n}, \langle\cdot,\cdot\rangle_{\bN^{-1}}) \to (\mathbb{R}^{m}, \langle\cdot,\cdot\rangle_{\bM^{-1}}), \  
	\bx\mapsto \bA\bx ,
\]
where $\bx$ and $\bA\bx$ are vectors under the canonical bases of $\mathbb{R}^{n}$ and $\mathbb{R}^{m}$, respectively. Here the two inner products are defined as
$
	\langle\bx,\bx'\rangle_{\bN^{-1}}:=\bx^{\top}\bN^{-1}\bx'$ and
$
	\langle\by,\by'\rangle_{\bM^{-1}}:=\by^{\top}\bM^{-1}\by' .
$
Therefore, we can define 
\[
	\bA^{*}: (\mathbb{R}^{m}, \langle\cdot,\cdot\rangle_{\bM^{-1}}) \to (\mathbb{R}^{n}, \langle\cdot,\cdot\rangle_{\bN^{-1}}), \  
	\by \mapsto \bA^{*}\by,
\]
which is the adjoint operator of $\bA$, by the relation $\langle \bA\bx, \by\rangle_{\bM^{-1}} = \langle \bx, \bA^{*}\by\rangle_{\bN^{-1}}$. Note that $(\bA\bx)^{\top}\bM^{-1}\by=\bx^{\top}\bN^{-1}\bA^{*}\by$ for any $\bx\in\mathbb{R}^{n}$ and $\by\in\mathbb{R}^{m}$. Thus, the matrix-form expression of $\bA^{*}$ is
$
	\bA^{*} = \bN\bA^{\top}\bM^{-1}.
$

Applying the standard Golub-Kahan bidiagonalization (GKB) to the compact operator $\bA$ with starting vector $\bb$ between the two Hilbert spaces $(\mathbb{R}^{n}, \langle\cdot,\cdot\rangle_{\bN^{-1}})$ and $(\mathbb{R}^{m}, \langle\cdot,\cdot\rangle_{\bM^{-1}})$, we can obtain the \textsf{gen-GKB} process; see \cite{caruso2019convergence} for GKB for compact operators. The basic recursive relations of  \textsf{gen-GKB} are as follows:
\begin{subequations}
	\begin{align}
		& \beta_1\bu_1 = \bb, \label{GKB1} \\
		& \alpha_{i}\bv_i = \bA^{*}\bu_i -\beta_i\bv_{i-1}, \label{GKB2} \\
		& \beta_{i+1}\bu_{i+1} = \bA\bv_{i} - \alpha_i\bu_i, \label{GKB3}
	\end{align}
\end{subequations}
where $\alpha_i$ and $\beta_i$ are computed such that $\|\bu_i\|_{\bM^{-1}}=\|\bv_i\|_{\bN^{-1}}=1$, and $\bv_0:=\boldsymbol{0}$. The whole iterative process is summarized in \Cref{alg:GKB}. For more details of the derivation, please see \cite{li2023subspace}. 

We remark that computing with $\bM^{-1}$ can not be avoided, but for the most commonly encountered cases that $\bepsilon$ is a Gaussian noise with uncorrelated components, $\bM$ is diagonal and thereby $\bM^{-1}$ can be directly obtained. For applications that $\bepsilon$ is a colored Gaussian noise such that $\bM$ is not diagonal, computing $\bM$ is the most expensive operation. In these cases, the proposed \textsf{PNT} method based on \textsf{gen-GKB} may not be the optimal choice.

\begin{algorithm}[htb]
	\caption{Generalized Golub-Kahan bidiagonalization (\textsf{gen-GKB})}\label{alg:GKB}
	\begin{algorithmic}[1]
		\Require $\bA\in\mathbb{R}^{m\times n}$, $\bb\in\mathbb{R}^{m}$, $\bM\in\mathbb{R}^{m\times m}$, $\bN\in\mathbb{R}^{n\times n}$
		\State $\bar{\bs}=\bM^{-1}\bb$, \ 
			$\beta_1=\bar{\bs}^{\top}\bb$, \ $\bu_1=\bb/\beta_1$, \ $\bar{\bu}_1=\bar{\bs}/\beta_1$
		\State $\bar{\br}=\bA^{\top}\bar{\bu}_1$, \ $\br=\bN\bar{\br}$
		\State $\alpha_1 = (\br^{\top}\bar{\br})^{1/2}$, \ $\bar{\bv}_1=\bar{\br}/\alpha_1$, \ $\bv_1=\br/\alpha_1$
		\For {$i=1,2,\dots,k$}
		\State $\bs=\bA\bv_{i} - \alpha_i\bu_{i}$, \
		       $\bar{\bs}=\bM^{-1}\bs$ 
		\State $\beta_{i+1}=(\bs^{\top}\bar{\bs})^{1/2}$, \ $\bu_{i+1}=\bs/\beta_{i+1}$, \ $\bar{\bu}_{i+1}=\bar{\bs}/\beta_{i+1}$  
		\State $\bar{\br}=\bA^{\top}\bar{\bu}_{i+1} - \beta_{i+1}\bar{\bv}_{i}$, \
			$\br=\bN\bar{\br}$ 
		\State $\alpha_{i+1}=(\br^{\top}\bar{\br})^{1/2}$, \ $\bar{\bv}_{i+1}=\bar{\br}/\alpha_{i+1}$, \ $\bv_{i+1}=\br/\alpha_{i+1}$  
		\EndFor
		\Ensure $\{\alpha_i, \beta_i\}_{i=1}^{k+1}$, \ $\{\bu_i, \bv_i\}_{i=1}^{k+1}$
	\end{algorithmic}
\end{algorithm}

The following result gives the basic property of \textsf{gen-GKB}; see \cite{li2023subspace} for the proof.

\begin{proposition}\label{prop:Krylov}
	The group of vectors $\{\bu_i\}_{i=1}^{k}$ is an $\bM^{-1}$-orthonormal basis of the Krylov subspace 
	\begin{equation}
		\mathcal{K}_k(\bA\bN\bA^{\top}\bM^{-1}, \bb) = \mathrm{span}\{(\bA\bN\bA^{\top}\bM^{-1})^i\bb\}_{i=0}^{k-1},
	\end{equation}
	and $\{\bv_i\}_{i=1}^{k}$ is an $\bN^{-1}$-orthonormal basis of the Krylov subspace 
	\begin{equation}
		\mathcal{K}_k(\bN\bA^{\top}\bM^{-1}\bA, \bN\bA^{\top}\bM^{-1}\bb) = \mathrm{span}\{(\bN\bA^{\top}\bM^{-1}\bA)^i\bN\bA^{\top}\bM^{-1}\bb\}_{i=0}^{k-1}.
	\end{equation}
\end{proposition}

Define $\bU_{k+1}=(\bu_1,\dots,\bu_{k+1})$ and $\bV_{k+1}=(\bv_1,\dots,\bv_{k+1})$. Then \Cref{prop:Krylov} indicates that $\bU_{k+1}^{\top}\bM^{-1}\bU_{k+1}=\bI$ and $\bV_{k+1}^{\top}\bN^{-1}\bV_{k+1}=\bI$. We remark that \textsf{gen-GKB} will eventually terminate in at most $\min\{m,n\}$ steps, since the column rank of $\bU_{k}$ or $\bV_{k}$ can not exceed $\min\{m,n\}$. If we define the \textit{termination step} as
\begin{equation}\label{break_down}
	k_{t}: = \max\{k: \alpha_{k}\beta_{k}>0\},
\end{equation}
then $\bV_{k}$ will eventually expand to be $\bV_{k_t}$ with $k_{t}\leq \min\{m,n\}$.

By \cref{GKB1}--\cref{GKB3}, we can write the $k$-step ($k\leq k_t$) \textsf{gen-GKB} in the  matrix-form
\begin{subequations}\label{gGKB_matrix}
	\begin{align}
		& \beta_1\bU_{k+1}e_{1} = \bb, \label{GKB13} \\
		& \bA\bV_k = \bU_{k+1}\bB_k, \label{GKB23} \\
		& \bN\bA^{\top}\bM^{-1}\bU_{k+1} = \bV_k\bB_{k}^{\top}+\alpha_{k+1}\bv_{k+1}e_{k+1}^{\top}, \label{GKB33}
	\end{align}
\end{subequations}
where $\be_1$ and $\be_{k+1}$ are the first and $(k+1)$-th columns of the identity matrix of order $k+1$, respectively, and 
\begin{equation}
	\bB_{k}=\begin{pmatrix}
		\alpha_{1} & & & \\
		\beta_{2} &\alpha_{2} & & \\
		&\beta_{3} &\ddots & \\
		& &\ddots &\alpha_{k} \\
		& & &\beta_{k+1}
		\end{pmatrix}\in  \mathbb{R}^{(k+1)\times k} .
\end{equation}
Note that $\bB_k$ has full column rank if $k\leq k_t$. At the $k_t$-th iteration, it is possible that either $\beta_{k_t+1}=0$ occurs first or $\alpha_{k_t+1}=0$ occurs first. For the former case, the relations \cref{gGKB_matrix} are replaced by
\begin{subequations}\label{gGKB_matrix1}
	\begin{align}
		& \beta_1\bU_{k_t}e_{1} = \bb, \label{GKB41} \\
		& \bA\bV_{k_t} = \bU_{k_t}\underline{\bB}_{k_t}, \label{GKB42} \\
		& \bN\bA^{\top}\bM^{-1}\bU_{k_t} = \bV_{k_t}\underline{\bB}_{k_t}^{\top}, \label{GKB43}
	\end{align}
\end{subequations}
where $\underline{\bB}_{k_t}$ is the first $k\times k$ part of $\bB_{k_t}$ by discarding $\beta_{k_t+1}$. 
\medskip

\noindent\textbf{Step 2: Compute the projected Newton direction.} 
At the $k$-th iteration, we update $\bx_k\in\mathrm{span}\{\bV_{k}\}$ and $\lambda_k$ from the previous ones. For any $\bx\in\mathrm{span}\{\bV_{k}\}$ of the form $\bx=\bV_{k}\by$ with $\by\in\mathbb{R}^{k}$, define the projected gradient of $\calL(\bx,\lambda)$ as
\begin{equation}\label{proj_F}
	F^{(k)}(\by,\lambda) = \begin{pmatrix}
		\bV_{k}^{\top} &  \\
		 & 1
	\end{pmatrix}F(\bx,\lambda)
\end{equation}
and the projected Jacobian of $F(\bx,\lambda)$ as
\begin{equation}\label{proj_J}
	J^{(k)}(\by,\lambda) = 
	\begin{pmatrix}
		\bV_{k}^{\top} &  \\
		 & 1
	\end{pmatrix}J(\bx,\lambda) \begin{pmatrix}
		\bV_{k} & \\
		 & 1
	\end{pmatrix} .
\end{equation}

\begin{remark1}\label{rem:not}
	Since \textsf{gen-GKB} must terminate at the $k_t$-th iteration and $\bV_k$ eventually expands to be $\bV_{k_t}$, we need to discuss the two different cases that $k\leq k_t$ and $k>k_t$. For notational simplicity, in the rest part of the paper, we use $\bV_{k}$ and $\bB_{k}$ by default unless stated otherwise to denote 
	\begin{equation*}\label{notation1}
		\bV_{k} = 
		\begin{cases}
			\bV_{k}, \ \ k\leq k_t \\
			\bV_{k_t}, \ \ k>k_t
		\end{cases},  \ \ 
		\bB_{k} = 
		\begin{cases}
			\bB_{k}, \ \ k\leq k_t \\
			\bB_{k_t}, \ \ k>k_t
		\end{cases} ,  \ \
		\bx_{k} = 
		\begin{cases}
			\bV_{k}\by_k, \ \ k\leq k_t \\
			\bV_{k_t}\by_k, \ \ k>k_t
		\end{cases} ,
	\end{equation*}
	where $\by\in\mathbb{R}^{k}$ for $k\leq k_t$ and $\by\in\mathbb{R}^{k_t}$ for $k> k_t$. Moreover, for the case $\beta_{k_t+1}=0$, the relations \cref{gGKB_matrix} are replaced by \cref{gGKB_matrix1} and $\bB_{k_t}$ is replaced by $\underline{\bB}_{k_t}$. In the subsequent discussions, we employ the unified notations as presented in \cref{gGKB_matrix}, but the readers can readily differentiate between the two cases.
\end{remark1}

Notice that $\by$ is uniquely determined by $\bx=\bV_k\by$ since $\bV_k$ has full-column rank. Thus, $F^{(k)}(\by,\lambda)$ and $J^{(k)}(\by,\lambda)$ are well-defined. The next result shows how we can obtain $F^{(k)}(\by,\lambda)$ and $J^{(k)}(\by,\lambda)$ from $\bB_{k}$ without any additional computations.

\begin{lemma}\label{lem:proj}
	For any $\bx\in\mathrm{span}\{\bV_{k}\}$ with the form $\bx=\bV_{k}\by$, the projected gradient of $\calL(\bx,\lambda)$ has the expression
	\begin{equation}\label{proj_F_com}
		F^{(k)}(\by,\lambda) =
		\begin{pmatrix}
			\lambda\bB_{k}^{\top}(\bB_{k}\by-\beta_{1}e_{1})+\by \\
			\frac{1}{2}\|\bB_{k}\by-\beta_{1}e_{1}\|_{2}^2-\frac{\tau m}{2}
		\end{pmatrix},
	\end{equation}
	and the projected Jacobian of $F(\bx,\lambda)$ has the expression
	\begin{equation}\label{proj_J_com}
		J^{(k)}(\by,\lambda) = 
		\begin{pmatrix}
			\lambda\bB_{k}^{\top}\bB_k+\bI & \bB_{k}^{\top}(\bB_{k}\by-\beta_{1}e_{1}) \\
			(\bB_{k}\by-\beta_{1}e_{1})^{\top}\bB_{k} & 0
		\end{pmatrix} .
	\end{equation}
\end{lemma}

Now we can compute the projected Newton direction for updating the solution. Starting from an initial solution $(\bx_0,\lambda_0)$, consider the following two cases. \\
\begin{subequations}\label{proj_direction}
\noindent\textbf{Case 1: Update $(\bx_{k},\lambda_{k})$ from $(\bx_{k-1},\lambda_{k-1})$ for $k\leq k_{t}$.}
Suppose at the $(k-1)$-th iteration, we have $\bx_{k-1}=\bV_{k-1}\by_{k-1}$, where $\bx_{0}:=\mathbf{0}$ and $\by_{0}:=()$ is an empty vector. Let $\bar{\by}_{k-1} = (\by_{k-1}^{\top},0)^{\top}\in\mathbb{R}^{k}$. If $J^{(k)}(\bar{\by}_{k-1},\lambda_{k-1})$ is nonsingular, we compute the Newton direction for the projected function $F^{(k)}(\by,\lambda)$ at $(\bar{\by}_{k-1},\lambda_{k-1})$: 
	\begin{equation} \label{proj_Newton1}
	\begin{pmatrix} \Delta \by_{k}\\ \Delta \lambda_{k} \end{pmatrix} = 
	- J^{(k)}(\bar{\by}_{k-1},\lambda_{k-1}) ^{-1} F^{(k)}(\bar{\by}_{k-1},\lambda_{k-1}) .
	\end{equation}
	Then we update $(\bar{\by}_{k},\lambda_{k})$ by 
	\begin{equation}\label{update1}
		\by_{k} = \bar{\by}_{k-1} + \gamma_k  \Delta \by_{k}, \ \ \ 
		\lambda_{k} = \lambda_{k - 1} + \gamma_k \Delta \lambda_{k}
	\end{equation}
	with a suitably chosen step-length $\gamma_k>0$, and let $\bx_k=\bV_{k}\by_{k}$. \\
\noindent\textbf{Case 2: Update $(\bx_{k},\lambda_{k})$ from $(\bx_{k-1},\lambda_{k-1})$ for $k> k_{t}$.}
At each iteration, we seek a solution of the form $\bx_{k}=\bV_{k_t}\by_{k}$ with $\by_{k}\in\mathbb{R}^{k_t}$. We compute the Newton direction 
	\begin{equation} \label{proj_Newton2}
		\begin{pmatrix} \Delta \by_{k}\\ \Delta \lambda_{k} \end{pmatrix} = 
		- J^{(k_t)}(\by_{k-1},\lambda_{k-1}) ^{-1} F^{(k_t)}(\by_{k},\lambda_{k-1}) ,
	\end{equation}
	and then compute
	\begin{equation}\label{update2}
	\by_{k} = \by_{k-1} + \gamma_k  \Delta \by_{k}, \ \ \
	\lambda_{k} = \lambda_{k - 1} + \gamma_k \Delta \lambda_{k}
	\end{equation}
	to get $\bx_k=\bV_{k_t}\by_{k}$.
\end{subequations}

For both the two cases, we call $(\Delta\by_{k},\Delta\lambda_{k})$ the \textit{projected Newton direction}, since it is the Newton direction of a projected problem. The corresponding update formula for $\bx_k$ is
\[
	\bx_k = \bx_{k-1}+\gamma_{k}\Delta\bx_{k}, \ \ \  \Delta\bx_{k}:= \bV_{k}\Delta\by_{k},
\]
which is easy to be verified.
For notational simplicity, in the subsequent part we always use the unified notation
\begin{equation}\label{notation2}
	\bar{\by}_{k-1}=
	\begin{cases}
		(\by_{k-1}^{\top}, \ 0)^{\top}, \ \ k\leq k_{t} \\
		\by_{k-1}, \ \ \ \ \ \ \ \ \ \ k> k_{t}
	\end{cases}
\end{equation}
for $\bar{\by}_{k-1}$. Following the notations stated in \cref{notation1,notation2}, we can use \cref{proj_Newton1} and \cref{update1} to describe the update procedure for both the two cases. 

It is vital to make sure that the projected Jacobian matrix $J^{(k)}(\bar{\by}_{k-1},\lambda_{k-1})$ is always nonsingular. This desired property is given in the following result. The proof appears as a part of the proof of \Cref{lem:inverse_bnd}.

\begin{proposition}\label{prop:nonsingular}
	If we choose $\bx_{0}=\mathbf{0}$ and $\bar{\by}_0=0$, then at each iteration $J^{(k)}(\bar{\by}_{k-1},\lambda_{k-1})$ is nonsingular as long as $\lambda_{k-1}\geq 0$. 
\end{proposition}

In order to investigate the convergence behavior of the method, define the following merit function:
\begin{equation}
	h(\bx,\lambda) = \frac{1}{2}[\|\lambda\bA^{\top}\bM^{-1}(\bA\bx-\bb)+\bN^{-1}\bx\|_{\bN}^{2} + (\frac{1}{2}\|\bA\bx-\bb\|_{\bM^{-1}}^{2}-\frac{\tau m}{2})^2] .
\end{equation}
Notice from \Cref{coro:F} that $(\bx^{*},\lambda^{*})$ is the unique minimizer of $h(\bx,\lambda)$ and that $h(\bx^{*},\lambda^{*})=0$. The following result shows that $(\Delta \bx_{k}^{\top},\Delta \lambda_k)^{\top}$ is indeed a descent direction for $h(\bx,\lambda)$.

\begin{theorem}\label{thm:descent}
	Let $\Delta\bx_k=\bV_{k}\Delta\by_k$. Then it holds 
	\begin{equation}
		\nabla h(\bx_{k-1},\lambda_{k-1})^{\top} \begin{pmatrix}
			\Delta\bx_k \\ \Delta\lambda_k
		\end{pmatrix} = -2h(\bx_{k-1},\lambda_{k-1}) \leq 0. 
	\end{equation}
\end{theorem}

\Cref{thm:descent} is a desired property for a gradient descent type algorithm. At the $(k-1)$-th iteration, if $h(\bx_{k-1},\lambda_{k-1})=0$, then we have $(\bx_{k-1},\lambda_{k-1})=(\bx^{*},\lambda^{*})$, meaning we have obtained the unique solution to \cref{discrepancy}. Otherwise, $(\Delta \bx_{k}^{\top},\Delta \lambda_k)^{\top}$ is a descent direction of $h(\bx,\lambda)$ at $(\bx_{k-1},\lambda_{k-1})$, thereby we can continue updating the solution by a backtracking line search strategy. 
\medskip

\noindent\textbf{Step 3: Determine step-length by backtracking line search.} 
For the case that $h(\bx_{k-1},\lambda_{k-1})\neq 0$, we need to determine a step-length $\gamma_k$ such that $h(\bx_k,\lambda_k)$ decreases strictly. To this end, we use the backtracking line search procedure to ensure that the \textit{Armijo condition} \cite[\S 3.1]{nocedal2006numerical} is satisfied: 
\begin{equation} \label{eq:sufficient_decrease}
h(\bx_{k},\lambda_{k}) \leq  h(\bx_{k-1},\lambda_{k-1}) + c\gamma_k (\Delta\bx_{k-1}^{\top}, \Delta\lambda_{k})\nabla h(\bx_{k-1},\lambda_{k-1}) ,
\end{equation}
where $(\bx_{k},\lambda_{k})=(\bx_{k-1},\lambda_{k-1})+\gamma_{k}(\Delta\bx_{k},\Delta\lambda_{k})$, and $c\in(0,1)$ is a fixed constant. At each iteration, we can quickly compute $h(\bx_{k},\lambda_{k})$ based on the following result.

\begin{lemma}\label{lem:merit_comp}
	Let
	\begin{equation*}
		\bar{F}^{(k)}(\by,\lambda) =
		\begin{pmatrix}
			\lambda\bar{\bB}_{k}^{\top}(\bar{\bB}_{k}\by-\beta_{1}e_{1})+y \\
			\frac{1}{2}\|\bar{\bB}_{k}\by-\beta_{1}e_{1}\|_{2}^{2}-\frac{\tau m}{2}
		\end{pmatrix}, \ \ \ 
		\bar{\bB}_{k} = 
		\begin{pmatrix}
			\alpha_{1} & & & \\
			\beta_{2}  & \alpha_{2} & & \\
			           & \ddots & \ddots & \\
					   &        & \beta_{k+1} & \alpha_{k+1}
		\end{pmatrix} .
	\end{equation*}
	Then we have 
	\begin{align}\label{merit}
		h(\bx_{k-1},\lambda_{k-1}) = \frac{1}{2}\|F^{(k)}(\bar{\by}_{k-1},\lambda_{k-1})\|_{2}^{2}, \ \ \
		h(\bx_{k},\lambda_{k}) = \frac{1}{2}\|\bar{F}^{(k)}(\bar{\by}_{k},\lambda_{k})\|_{2}^{2} .  
	\end{align}
\end{lemma}

We remark that in the above expression we have $\bar{\bB}_{k}=\bB_{k_t}$ for $k\geq k_t$, and specifically, we have $\bar{\bB}_{k}=\underline{\bB}_{k_t}^{\top}$ if $\beta_{k_t+1}=0$. The following theorem shows the existence of a suitable step-length; see e.g. \cite[pp. 121, Theorem 2.1]{blowey2003frontiers} for details.
\begin{theorem}\label{thm:Armijo}
	For any continuously differentiable function $f(\bs):\mathbb{R}^{l}\rightarrow \mathbb{R}$, suppose $\nabla f$ is Lipschitz continuous with constant $\zeta(\bs)$ at $\bs$. If $\bp$ is a descent direction at $\bs$, i.e $\nabla f(\bs)^{\top}\bp < 0$, then for a fixed $c\in(0,1)$ the Armijo condition 
	$
		f(\bs + \gamma \bp) \leq f(\bs) + c\gamma \nabla f(\bs)^{\top} \bp
	$
	is satisfied for all $\gamma \in [0,\gamma_{\mathrm{max}}]$ with
	$
		\gamma_{\mathrm{max}} = \frac{2(c-1)\nabla f(\bs)^{\top} \bp}{\zeta(\bs)\|\bp\|^2}.
	$
\end{theorem} 

With the aid of \Cref{lem:merit_comp}, a suitable step-length $\gamma_k$ can be determined using the following backtracking line search strategy. 
\begin{routine}\label{rout:backtracking}
	Armijo backtracking line search:
	\begin{enumerate}\label{rout:armijo}
		\item Given $\gamma_{\text{init}}>0$, let $\gamma^{(0)}=\gamma_{\text{init}}$ and $l=0$.
		\item Until $\frac{1}{2}\|\bar{F}^{(k)}(\bar{\by}_{k},\lambda_{k})\|_{2}^{2}\leq\left(\frac{1}{2}-c\gamma^{(l)}\right)\|F^{(k)}(\bar{\by}_{k-1},\lambda_{k-1})\|_{2}^{2}$, \\
		\ \ (i) set $\gamma^{(l+1)}=\eta\gamma^{(l)}$, where $\eta\in(0,1)$ is a fixed constant; \\
		\ \ (ii) $l\leftarrow l+1$.
		\item Set $\gamma_{k}=\gamma^{(l)}$.
	\end{enumerate}
\end{routine}
We set $c=10^{-4}$, $\gamma_{\text{init}} = 1.0$ and $\eta=0.9$ by default. Note that at each iteration we need to ensure $\lambda_{k}>0$. Suppose at the $(k-1)$-th iteration we already have $\lambda_{k-1}>0$. Then at the $k$-th iteration, if $\Delta\lambda_{k}<0$, we only need to enforce $\gamma_{\text{init}}<-\lambda_{k-1}/\Delta\lambda_{k}$.

Overall, the whole procedure of \textsf{PNT} is presented in \Cref{alg:PNTM}. In the \textsf{PNT} algorithm, at each $k$-th iteration, computing the projected Newton direction requires solving only the $(k+1)$-order linear system \cref{proj_Newton1}, which can be done very quickly when $k\ll n$. Starting from the termination step $k_t$, at each subsequent iteration, a $(k_t+1)$-order linear system \cref{proj_Newton2} needs to be solved. We numerically find that the algorithm almost always obtains a satisfied solution before \textsf{gen-GKB} terminates. The \textsf{PNT} method is a natural generalization of the projected Newton method in \cite{cornelis2020projected1}. Specifically, when $\bM=\bI$ and $\bN=\bI$, it can be confirmed that both methods are identical.

We remark that for very large-scale problems, it may take too many iterations for \textsf{PNT} to converge. In this case, we can update the solution starting from the $k_0$-th step of \textsf{gen-GKB} to save some computation for solving \cref{proj_Newton1}. This means that we first run $k_0-1$ steps \textsf{gen-GKB} to construct a $(k_0-1)$-dimensional subspace and then start to update the solution from the $k_0$-th iteration. From the derivation of \textsf{PNT}, it can be easily verified that if we set $\bar{\by}_{k_0-1}=\mathbf{0}\in\mathbb{R}^{k_0}$, then $(\Delta \bx_{k}^{\top},\Delta \lambda_k)^{\top}$ is a descent direction of $h(\bx,\lambda)$ at each iteration $k\geq k_0$. We  refer to this modified PNT method as \textsf{PNT-md}.

\begin{algorithm}[htbp]
\caption{Projected Newton method (\textsf{PNT}) for \cref{discrepancy,lagr}}\label{alg:PNTM}
\begin{algorithmic}[1]
\Require $\bA\in\mathbb{R}^{m\times n}$, $\bb\in\mathbb{R}^{m}$, $\bM\in\mathbb{R}^{m\times m}$, $\bN\in\mathbb{R}^{n\times n}$, $\tau\gtrsim 1$ 
\State{Initialization: $\lambda_0>0$, $\bar{\by}_0 = \mathbf{0}$; $c = 10^{-4}$, $\eta = 0.9$; tol $>$ 0}
\State{Compute $\beta_1$, $\alpha_1$, $\bu_1$, $\bv_1$ by \Cref{alg:GKB}} 
\For{$k = 1,2,\ldots$} 
\State Compute $\beta_{k+1}$, $\alpha_{k+1}$, $\bu_{k+1}$, $\bv_{k+1}$ by \Cref{alg:GKB}; Form $\bB_{k+1}$ and $\bV_{k}$
\State (Terminate \textsf{gen-GKB} if $\beta_{k+1}$ or $\alpha_{k+1}$ is extremely small)
\State{Compute $F^{(k)}(\bar{\by}_{k-1},\lambda_{k-1})$ and $J^{(k)}(\bar{\by}_{k-1},\lambda_{k-1})$ by \cref{proj_F_com,proj_J_com}}
\State{Compute $(\Delta \by_{k}, \Delta \lambda_{k})$ by \cref{proj_Newton1}} 
\If{$\Delta \lambda_{k}> 0$} 
\State{$\gamma_{\text{init}} = 1$}
\Else
\State{$\gamma_{\text{init}} = \min\{1, -\eta\lambda_{k-1} / \Delta \lambda_{k}\}$}   \Comment{Ensure the positivity of $\lambda_{k}$}
\EndIf
\State Determine the step-length $\gamma_k$ by \Cref{rout:armijo}
\State Update $(\by_{k}, \lambda_k)$ by \cref{update1}
\If{$\frac{1}{2}\|{\bar{F}}^{(k)}(\bar{\by}_{k},\lambda_{k})\|_2 \leq \text{tol} $} 
\State{Compute $\bx_k = \bV_k\by_k$; \ Stop iteration}
\EndIf
\EndFor
\Ensure Final solution $(\bx_{k},\lambda_k)$
\end{algorithmic}
\end{algorithm}

From the derivations, we find that the success of PNT is attributed to the fact that $\bM$ and $\bN$ can induce inner products. Therefore, the PNT method can not be directly used to handle the total variation (TV) or $\ell_1$ regularization terms. One possible approach for handling TV or $\ell_1$ norms is to approximate them with weighted $\ell_2$ norms at each iterated point \cite{rodriguez2008efficient1,rodriguez2008efficient2}. Furthermore, for the nonlinear inverse problem $\bb=G(\bx)+\bepsilon$ with differentiable $G$, at each iterated point we can approximate $\|G(\bx)-\bb\|_{\bM^{-1}}^2$ by a quadratic convex function using the first-order Taylor expansion of $G$. The above approaches follow a similar idea to the sequential quadratically constrained quadratic programming
(SQCQP) method \cite{fukushima2003sequential,messerer2021survey}. This allows us to obtain a sequence of optimization problems similar to \cref{discrepancy}, which can be solved efficiently by PNT. Theoretical and computational aspects of this approach will be further studied in the future.

\subsection{Proofs}\label{subsec3.3}
We give the proofs of all the results in \Cref{subsec3.2}. Remember again that we always follow the notations as stated in \Cref{rem:not} and \cref{notation2}.
\medskip

\begin{proofof}{\Cref{lem:proj}}
	By \cref{proj_F} and \cref{proj_J} we have
	\begin{equation*}
		F^{(k)}(\by,\lambda) = \begin{pmatrix}
			\lambda(\bA\bV_{k})^{\top}\bM^{-1}(\bA\bV_{k}\by-\bb)+\bV_{k}^{\top}\bN^{-1}\bV_{k}\by \\
			\frac{1}{2}\|\bA\bV_{k}\by-\bb\|_{\bM^{-1}}^2-\frac{\tau m}{2}
		\end{pmatrix},
	\end{equation*}
	and 
	\begin{equation*}
		J^{(k)}(\by,\lambda) = \begin{pmatrix}
			\lambda(\bA\bV_{k})^{\top}\bM^{-1}(\bA\bV_{k})+ \bV_{k}^{\top}\bN^{-1}\bV_{k} & (\bA\bV_{k})^{\top}\bM^{-1}(\bA\bV_{k}\by-\bb) \\
			(\bA\bV_{k}\by-\bb)^{\top}\bM^{-1}(\bA\bV_{k}) & 0
		\end{pmatrix} .
	\end{equation*}
	Using relations \cref{gGKB_matrix} and \Cref{prop:Krylov}, we have $\bA\bV_{k}\by-\bb=\bU_{k+1}(\bB_{k}-\beta_{1}e_{1})$, leading to
	\[(\bA\bV_{k})^{\top}\bM^{-1}(\bA\bV_{k}\by-\bb)=(\bU_{k+1}\bB_{k})^{\top}\bM^{-1}\bU_{k+1}(\bB_{k}-\beta_{1}e_{1})
	=\bB_{k}^{\top}(\bB_{k}-\beta_{1}e_{1}),\]
	and
	\[\|\bA\bV_{k}\by-\bb\|_{\bM^{-1}}^2= \|\bU_{k+1}(\bB_{k}-\beta_{1}e_{1})\|_{\bM^{-1}}^2=\|\bB_{k}-\beta_{1}e_{1}\|_{2}^2 .\]
	If $\beta_{k_t+1}=0$, then for $k\geq k_t$, the relation $\bA\bV_{k}\by-\bb=\bU_{k+1}(\bB_{k}-\beta_{1}e_{1})$ is replaced by $\bA\bV_{k_t}\by-\bb=\bU_{k_t}(\underline{\bB}_{k_t}-\beta_{1}e_{1})$. Therefore, the above identity is also applied to the case $k \geq k_t$. Now we have proved \cref{proj_F_com}. The expression \cref{proj_J_com} can be proved similarly.
\end{proofof}

In order to prove \Cref{lem:merit_comp}, we first give the following result.
\begin{lemma}\label{lem:proj_F}
	Let $\widehat{\bN}=\begin{pmatrix}
		\bN &  \\
		 & 1
	\end{pmatrix}$. Then we have the following identity:
	\begin{equation}
		\|F(\bx_{k-1},\lambda_{k-1})\|_{\widehat{\bN}}=\|F^{(k)}(\bar{\by}_{k-1},\lambda_{k-1})\|_{2} .
	\end{equation}
\end{lemma}
\begin{proof}
	First notice that 
	\begin{equation*}
		\|F(\bx_{k-1},\lambda_{k-1})\|_{\widehat{\bN}}^2 = 
		\|\lambda_{k-1}\bA^{\top}\bM^{-1}(\bA\bx_{k-1}-\bb)+\bN^{-1}\bx\|_{\bN}^{2} + (\frac{1}{2}\|\bA\bx_{k-1}-\bb\|_{\bM^{-1}}^{2}-\frac{\tau m}{2})^2.
	\end{equation*}
	For the first term of the above summation, we have
	\begin{align*}
		& \ \ \ \ \|\lambda_{k-1}\bA^{\top}\bM^{-1}(\bA\bx_{k-1}-\bb)+\bN^{-1}\bx_{k-1}\|_{\bN}^{2} \\
		&= (\lambda_{k-1}\bA^{\top}\bM^{-1}(\bA\bx_{k-1}-\bb)+\bN^{-1}\bx_{k-1})^{\top}\bN(\lambda_{k-1}\bA^{\top}\bM^{-1}(\bA\bx_{k-1}-\bb)+\bN^{-1}\bx_{k-1}),
	\end{align*}
	and
	\begin{align*}
		& \ \ \ \ \bN(\lambda_{k-1}\bA^{\top}\bM^{-1}(\bA\bx_{k-1}-\bb)+\bN^{-1}\bx_{k-1}) \\ 
		&= \lambda_{k-1}\bN\bA^{\top}\bM^{-1}\bU_{k+1}(\bB_{k}\bar{\by}_{k-1}-\beta_{1}e_1)+\bV_{k}\bar{\by}_{k-1} \\
		&= \lambda_{k-1}(\bV_{k}\bB_{k}^{\top}+\alpha_{k+1}\bv_{k+1}e_{k+1}^{\top})(\bB_{k}\bar{\by}_{k-1}-\beta_{1}e_1) + \bV_{k}\bar{\by}_{k-1} \\
		&= \bV_{k}(\lambda_{k-1}\bB_{k}^{\top}(\bB_{k}\bar{\by}_{k-1}-\beta_{1}e_1)+\bar{\by}_{k-1}),
	\end{align*}
	where we have used
	\begin{align*}
		 \alpha_{k+1}\bv_{k+1}e_{k+1}^{\top}(\bB_{k}\bar{\by}_{k-1}-\beta_{1}e_1) 
		= \alpha_{k+1}\beta_{k+1}\bv_{k+1}e_{k}^{\top}\bar{\by}_{k-1} 
		= 0,
	\end{align*}
	because $e_{k}^{\top}\bar{\by}_{k-1} = 0$ for $k\leq k_t$ and $\alpha_{k+1}\beta_{k+1} = 0$ for $k>k_t$.
	Similarly, we have 
	\begin{equation*}
		\lambda_{k-1}\bA^{\top}\bM^{-1}(\bA\bx_{k-1}-\bb)+\bN^{-1}\bx_{k-1}
		= \bN^{-1}\bV_{k}(\lambda_{k-1}\bB_{k}^{\top}(\bB_{k}\bar{\by}_{k-1}-\beta_{1}e_1)+\bar{\by}_{k-1}) .
	\end{equation*}
	We also have
	\[
		\|\lambda_{k-1}\bA^{\top}\bM^{-1}(\bA\bx_{k-1}-\bb)+\bN^{-1}\bx_{k-1}\|_{\bN}
		= \|\lambda_{k-1}\bB_{k}^{\top}(\bB_{k}\bar{\by}_{k-1}-\beta_{1}e_1)+\bar{\by}_{k-1}\|_{2}
	\] 
	and 
	\[\frac{1}{2}\|\bA\bx_{k-1}-\bb\|_{\bM^{-1}}^{2}-\frac{\tau m}{2}=\frac{1}{2}\|\bB_{k}\bar{\by}_{k-1}-\beta_{1}e_1\|_{2}^{2}-\frac{\tau m}{2}. \] 
	The desired result immediately follows by using \cref{proj_F_com}.
\end{proof}

\begin{proofof}{\Cref{lem:merit_comp}}
	First notice that 
	$
		h(\bx,\lambda) = \frac{1}{2}\|F(\bx,\lambda)\|_{\widehat{N}}^2 .
	$
	Combining the above relation with \Cref{lem:proj_F} we obtain the first identity of \cref{merit}. Also, for $k<k_t$ we have 
	$
		h(\bx_{k},\lambda_{k}) = \frac{1}{2}\|F^{(k+1)}(\bar{\by}_{k},\lambda_{k})\|_{2}^{2}
	$
	with
	\begin{equation*}
		F^{(k+1)}(\bar{\by}_{k},\lambda_{k}) = 
		\begin{pmatrix}
			\lambda\bB_{k+1}^{\top}(\bB_{k+1}\bar{\by}_k-\beta_{1}e_{1})+\bar{\by}_k \\
			\frac{1}{2}\|\bB_{k+1}\bar{\by}_k-\beta_{1}e_{1}\|_{2}^2-\frac{\tau m}{2}
		\end{pmatrix} .
	\end{equation*}
	Since the last element of $\beta_{1}e_{1}$ and $\bar{\by}_k$ is zero, it is easy to verify that 
	\begin{equation*}
		\bB_{k+1}\bar{\by}_k-\beta_{1}e_{1} = 
		\begin{pmatrix}
			\bar{\bB}_{k}\bar{\by}_k-\beta_{1}e_{1} \\ 0
		\end{pmatrix}, \ \
		\bB_{k+1}^{\top}(\bB_{k+1}\bar{\by}_k-\beta_{1}e_{1}) = 
		\bar{\bB}_{k}^{\top}(\bar{\bB}_{k}\bar{\by}_k-\beta_{1}e_{1}).
	\end{equation*}
	For $k\geq k_t$, we have $F^{(k)}(\by,\lambda)=F^{(k+1)}(\by,\lambda)=\bar{F}^{(k)}(\by,\lambda)$. Therefore, we prove the second identity of \cref{merit}.
\end{proofof}

In order to prove \Cref{thm:descent}, we need \Cref{lem:proj_F} and the following result.
\begin{lemma}\label{lem:proj_J}
	For any $k\geq 1$, we have the following identity:
	\begin{equation*}
		\begin{pmatrix}
			\bV_{k}^{\top} &  \\
			 & 1
		\end{pmatrix}J(\bx_{k-1},\lambda_{k-1})\widehat{\bN}
		F(\bx_{k-1},\lambda_{k-1})
		= J^{(k)}(\bar{\by}_{k-1},\lambda_{k-1})F^{(k)}(\bar{\by}_{k-1},\lambda_{k-1}).
	\end{equation*}
\end{lemma}
\begin{proof}
	First notice from \cref{Jacob} that
	\begin{align*}
		\begin{pmatrix}
			\bV_{k}^{\top} &  \\
			 & 1
		\end{pmatrix}J(\bx_{k-1},\lambda_{k-1})\widehat{\bN}
		&= \begin{pmatrix}
			\lambda_{k-1}(\bA\bV_{k})^{\top}\bM^{-1}\bA\bN+\bV_{k}^{\top} & (\bA\bV_{k})^{\top}\bM^{-1}(\bA\bx_{k-1}-\bb) \\
			(\bA\bx_{k-1}-\bb)^{\top}\bM^{-1}\bA\bN & 0
		\end{pmatrix} .
	\end{align*}
	Using \cref{GKB33} and similar derivations to the proof of Lemma \ref{lem:proj_F}, we get
	\begin{align}
		(\bA\bx_{k-1}-\bb)^{\top}\bM^{-1}\bA\bN
		&= (\bB_{k}\bar{\by}_{k-1}-\beta_{1}e_{1})^{\top}\bU_{k+1}^{\top}\bM^{-1}\bA\bN  \nonumber  \\
		&= (\bB_{k}\bar{\by}_{k-1}-\beta_{1}e_{1})^{\top}(\bB_{k}\bV_{k}^{\top}+\alpha_{k+1}e_{k+1}\bv_{k+1}^{\top})  \nonumber \\
		&= (\bB_{k}\bar{\by}_{k-1}-\beta_{1}e_{1})^{\top}\bB_{k}\bV_{k}^{\top}.  \label{id1}
	\end{align}
	Also, we can get
	\begin{align*}
		(\bA\bV_{k})^{\top}\bM^{-1}\bA\bN
		& = (\bU_{k+1}\bB_{k})^{\top}\bM^{-1}\bA\bN  
		= \bB_{k}^{\top}(\bB_{k}\bV_{k}^{\top}+\alpha_{k+1}e_{k+1}\bv_{k+1}^{\top})  \\
		&= \bB_{k}^{\top}\bB_{k}\bV_{k}^{\top}+\alpha_{k+1}\beta_{k+1}e_{k}\bv_{k+1}^{\top} , 
	\end{align*}
	and 
	\begin{align*}
		(\bA\bV_{k})^{\top}\bM^{-1}(\bA\bx_{k-1}-\bb)
		&= (\bB_{k}\bU_{k+1})^{\top}\bM^{-1}\bU_{k+1}(\bB_{k}\bar{\by}_{k-1}-\beta_1e_1) \\
		&= \bB_{k}^{\top}(\bB_{k}\bar{\by}_{k-1}-\beta_1e_1) .
	\end{align*}
	Using \cref{proj_J_com}, we get
	\begin{align*}
		& \ \ \ \ 
		\begin{pmatrix}
			\bV_{k}^{\top} &  \\
			 & 1
		\end{pmatrix}J(\bx_{k-1},\lambda_{k-1})\widehat{\bN} \\
		&= \begin{pmatrix}
			(\lambda_{k-1}\bB_{k}^{\top}\bB_{k}+\bI)\bV_{k}^{\top} & \bB_{k}^{\top}(\bB_{k}\bar{\by}_{k-1}-\beta_1e_1) \\
			\left(\bB_{k}\bar{\by}_{k-1}-\beta_{1}e_{1}\right)^{\top}\bB_{k}\bV_{k}^{\top} & 0
		\end{pmatrix}+
		\begin{pmatrix}
			\alpha_{k+1}\beta_{k+1}e_{k+1}\bv_{k+1}^{\top} &  \\
			 & 0
		\end{pmatrix} \\
		&= J^{(k)}(\bar{\by}_{k-1},\lambda_{k-1}) 
		\begin{pmatrix}
			\bV_{k}^{\top} &  \\
			 & 1
		\end{pmatrix} +
		\begin{pmatrix}
			\alpha_{k+1}\beta_{k+1}e_{k}\bv_{k+1}^{\top} &  \\
			 & 0
		\end{pmatrix}.
	\end{align*}
	Using similar derivations to the proof of Lemma \ref{lem:proj_F}, we get
	\begin{align*}
		F(\bx_{k-1},\lambda_{k-1})
		&= \begin{pmatrix}
			\bN^{-1} &  \\
			 & 1
		\end{pmatrix}
		\begin{pmatrix}
			\lambda_{k-1}\bN\bA^{\top}\bM^{-1}(\bA\bx_{k-1}-\bb)+\bx_{k-1} \\
			\frac{1}{2}\|\bA\bx_{k-1}-\bb\|_{\bM^{-1}}^{2}-\frac{\tau m}{2}
		\end{pmatrix} \\
		&=\begin{pmatrix}
			\bN^{-1} &  \\
			 & 1
		\end{pmatrix}
		\begin{pmatrix}
			\bV_{k} &  \\
			 & 1
		\end{pmatrix}
		\begin{pmatrix}
			\lambda_{k-1}\bB_{k}^{\top}(\bB_{k}\bar{\by}_{k-1}-\beta_{1}e_1)+\bar{\by}_{k-1} \\
			\frac{1}{2}\|\bB_{k}\bar{\by}_{k-1}-\beta_{1}e_1\|_{2}^{2}-\frac{\tau m}{2}
		\end{pmatrix} \\
		&= \begin{pmatrix}
			\bN^{-1} &  \\
			 & 1
		\end{pmatrix}
		\begin{pmatrix}
			\bV_{k} &  \\
			 & 1
		\end{pmatrix}
		F^{(k)}(\bar{\by}_{k-1},\lambda_{k-1}) .
	\end{align*}
	Using the relations 
	\begin{equation*}
		\begin{pmatrix}
			\bV_{k}^{\top} &  \\
			 & 1
		\end{pmatrix}
		\begin{pmatrix}
			\bN^{-1} &  \\
			 & 1
		\end{pmatrix}
		\begin{pmatrix}
			\bV_{k} &  \\
			 & 1
		\end{pmatrix} = \bI , \ \
		\begin{pmatrix}
			\alpha_{k+1}\beta_{k+1}e_{k}\bv_{k+1}^{\top} &  \\
			 & 0
		\end{pmatrix}
		\begin{pmatrix}
			\bN^{-1} &  \\
			 & 1
		\end{pmatrix}
		\begin{pmatrix}
			\bV_{k} &  \\
			 & 1
		\end{pmatrix} = \mathbf{0},
	\end{equation*}
	we finally obtain the desired result.
\end{proof}

\begin{proofof}{\Cref{thm:descent}}
	Notice $J(\bx,\lambda)$ is the Jacobian of $F(\bx,\lambda)$. Using $h(\bx,\lambda) = \frac{1}{2}\|F(\bx,\lambda)\|_{\widehat{N}}^2$ we get
	$
		\nabla h(\bx,\lambda) 
		= J(\bx,\lambda)\widehat{\bN}F(\bx,\lambda) ,
	$ leading to
	\begin{align*}
		\nabla h(\bx_{k-1},\lambda_{k-1})^{\top} \begin{pmatrix}
			\Delta\bx_k \\ \Delta\lambda_k
		\end{pmatrix} 
		& = \begin{pmatrix}
			\Delta\by_{k} \\ \Delta\lambda_{k}
		\end{pmatrix}^{\top}
		\begin{pmatrix}
			\bV_{k}^{\top} &  \\
			 & 1
		\end{pmatrix} J(\bx_{k-1},\lambda_{k-1})\widehat{\bN}F(\bx_{k-1},\lambda_{k-1}) \\
		&= \begin{pmatrix}
			\Delta\by_{k} \\ \Delta\lambda_{k}
		\end{pmatrix}^{\top}
		J^{(k)}(\bar{\by}_{k-1},\lambda_{k-1})F^{(k)}(\bar{\by}_{k-1},\lambda_{k-1}) \\
		&= -\|F^{(k)}(\bar{\by}_{k-1},\lambda_{k-1})\|_{2}^{2} = -\|F(\bx_{k-1},\lambda_{k-1})\|_{\widehat{\bN}}^2 \\
		&= -2h(\bx_{k-1},\lambda_{k-1}) \leq 0 ,
	\end{align*}
	where we have used \Cref{lem:proj_F} and \Cref{lem:proj_J}.
\end{proofof}

\section{Convergence analysis}\label{sec4}
The objective of this section is to prove the convergence of \textsf{PNT}, which is stated in the following result and \Cref{coro:converge}.

\begin{theorem}\label{thm:conv}
Suppose the \textsf{PNT} algorithm is initialized with $\bar{\by}_{0}=\mathbf{0}$, $\bx_{0}=\mathbf{0}$ and $\lambda_0>0$. Then we either have
\begin{equation}\label{conv1}
	h(\bx_k,\lambda_k) = 0
\end{equation}
for some $k<\infty$, or have
\begin{equation}\label{conv2}
	\lim_{k\rightarrow \infty} h(\bx_k,\lambda_k) = 0 .
\end{equation}
\end{theorem}

Notice that $(\bx^{*},\lambda^{*})$ is the unique minimizer of $h(\bx,\lambda)$ and $h(\bx^{*},\lambda^{*})=0$. Therefore, \cref{conv1} implies that the algorithm finds the exact solution to \cref{discrepancy,lagr} at the $k$-th iteration. In the following part, we prove \cref{conv2} under the assumption that $h(\bx_k,\lambda_k)>0$ for any $k\geq 1$. We need a series of lemmas, which are \Cref{lem:ls_lim}--\Cref{lem:gamma}. All these lemmas follow the same assumption of \Cref{thm:conv}.

\begin{lemma}\label{lem:ls_lim}
	For any matrix $\bC\in\mathbb{R}^{m\times n}$ with full column rank and $\bd\in\mathbb{R}^{m}$, if the vector sequence $\{\bw_k\}\in\mathbb{R}^{n}$ satisfies 
	$
		\lim_{k\rightarrow \infty}\|\bC^{\top}(\bC\bw_{k}-\bd)\|_{2} = 0,
	$
	then
	$
		\lim_{k\rightarrow \infty}\bw_k = \bw^{*}:= \argmin _{\bw\in\mathbb{R}^{n}}\|\bC\bw-\bd\|_{2} .
	$
\end{lemma}
\begin{proof}
	First notice that $\bw^{*}$ is well-defined, since $\argmin _{\bw\in\mathbb{R}^{n}}\|\bC\bw-\bd\|_{2}$ has a unique solution for the full column rank matrix $\bC$. For any $\bw_{k}$, let $\bw_k=\bw^{*}+\bar{\bw}_{k}$. Then we have
	\[
		\lim_{k\rightarrow \infty}\|\bC^{\top}(\bC\bw_{k}-\bd)\|_{2} = 
		\lim_{k\rightarrow \infty}\|\bC^{\top}(\bC\bw^{*}-\bd)+\bC^{\top}\bC\bar{\bw}_{k}\|_{2} =
		\lim_{k\rightarrow \infty}\|\bC^{\top}\bC\bar{\bw}_{k}\|_{2}=0,
	\]
	since $\bC^{\top}(\bC\bw^{*}-\bd)=\mathbf{0}$. Now we have $\|\bar{\bw}_k\|_{2}\rightarrow 0$ since $\bC^{\top}\bC$ is positive definite and all norms of $\mathbb{R}^{n}$ are equivalent. Therefore, we have $\|\bw_k-\bw^{*}\|_{2}\rightarrow 0$ or the equivalent form $\lim_{k\rightarrow \infty}\bw_k = \bw^{*}$.
\end{proof}

\begin{lemma}\label{lem:solv_ls}
	If the unique solution to $\min_{\by\in\mathbb{R}^{k_t}}\|\bB_{k_t}y-\beta_{1}e_1\|_{2}$ is $\by_{\mathrm{min}}$, then $\bx_{\mathrm{min}}:=\bV_{k_t}\by_{\mathrm{min}}$ is the unique solution to 
	\begin{equation}\label{ls_MN}
		\min_{\bx\in\mathbb{R}^{n}}\|\bx\|_{\bN^{-1}} \ \ \ \mathrm{s.t.} \ \ \
		\|\bA\bx-\bb\|_{\bM^{-1}}=\mathrm{min} .
	\end{equation}
\end{lemma}
\begin{proof}
	It is easy to verify that both $\min_{\by\in\mathbb{R}^{k_t}}\|\bB_{k_t}y-\beta_{1}e_1\|_{2}$ and $\cref{ls_MN}$ have a unique solution. A vector $\bx$ is the unique solution to \cref{ls_MN} if and only if
	\begin{equation*}
		\bA^{\top}\bM^{-1}(\bA\bx-\bb)=\mathbf{0} , \ \ \ 
		\bx\perp_{\bN^{-1}}\mathcal{N}(\bA) ,
	\end{equation*}
	where $\perp_{\bN^{-1}}$ means the orthogonality relation under the $\bN^{-1}$-inner product.
	Now we verify the above two conditions for $\bx_{\mathrm{min}}$. For the first condition, using the relations $\bA\bx_{\mathrm{min}}=\bA\bV_{k_t}\by_{\mathrm{min}}=\bU_{k_t+1}\bB_{k_t}\by_{\mathrm{min}}$ and \cref{GKB33}, we get
	\begin{align*}
		\bA^{\top}\bM^{-1}(\bA\bx_{\mathrm{min}}-\bb)
		&= \bA^{\top}\bM^{-1}\bU_{k_t+1}(\bB_{k_t}\by_{\mathrm{min}}-\beta_{1}e_{1}) \\
		&= \bN^{-1}(\bV_{k_t}\bB_{k_t}^{\top}+\alpha_{k_t+1}\bv_{k_t+1}e_{k_t+1}^{\top})\left(\bB_{k_t}\by_{\mathrm{min}}-\beta_{1}e_{1}\right) \\
		&= \bN^{-1}[\bV_{k_t}\bB_{k_t}^{\top}(\bB_{k_t}\by_{\mathrm{min}}-\beta_{1}e_{1})+\alpha_{k_t+1}\beta_{k_t+1}\bv_{k_t+1}e_{k_t}^{\top}\by_{\mathrm{min}}] \\
		&= \mathbf{0},
	\end{align*}
	since $\bB_{k_t}^{\top}(\bB_{k_t}\by_{\mathrm{min}}-\beta_{1}e_{1})=\mathbf{0}$ and $\alpha_{k_t+1}\beta_{k_t+1}=0$. For the second condition, by \Cref{prop:Krylov} we have 
	\begin{align*}
		\bx_{\mathrm{min}}\in\mathrm{span}\{\bV_{k_t}\} &= 
		\mathrm{span}\{(\bN\bA^{\top}\bM^{-1}\bA)^i\bN\bA^{\top}\bM^{-1}\bb\}_{i=0}^{k_t-1}  \\
		&\subseteq \mathcal{R}(\bN\bA^{\top})=\bN\mathcal{N}(\bA)^{\perp} .
	\end{align*}
	Write $\bx_{\mathrm{min}}=\bN\bar{\bx}_{\mathrm{min}}$ with $\bar{\bx}_{\mathrm{min}}\in\mathcal{N}(\bA)^{\perp}$. For any $\bw\in\mathcal{N}(\bA)$, we have
	\begin{equation*}
		\langle\bx_{\mathrm{min}}, \bw\rangle_{\bN^{-1}} = \langle\bN\bar{\bx}_{\mathrm{min}}, \bw\rangle_{\bN^{-1}}
		= \langle\bar{\bx}_{\mathrm{min}}, \bw\rangle_{2}=0.
	\end{equation*}
	Therefore, it holds that $\bx_{\mathrm{min}}\perp_{\bN^{-1}}\mathcal{N}(\bA)$.
\end{proof}

\begin{lemma}\label{lem:lower_bnd}
	There exist a positive constant $C_{1}$ such that for any $k\geq 1$,
	\begin{equation}\label{res_norm_bound}
		\|\bA^{\top}\bM^{-1}(\bA\bx_{k-1}-\bb)\|_{\bN} = \|\bB_{k}^{\top}(\bB_{k}\bar{\by}_{k-1}-\beta_{1}e_{1})\|_2 \geq C_{1} >0
	\end{equation}
\end{lemma}
\begin{proof}
	First, we get from \cref{id1} the first identity:
	\begin{align*}
		& \ \ \ \ \|\bA^{\top}\bM^{-1}(\bA\bx_{k-1}-\bb)\|_{\bN}^{2}  \\
		&= (\bA\bx_{k-1}-\bb)^{\top}\bM^{-1}\bA\bN\bN^{-1}\left((\bA\bx_{k-1}-\bb)^{\top}\bM^{-1}\bA\bN\right)^{\top} \\
		&= \left(\bB_{k}\bar{\by}_{k-1}-\beta_{1}e_{1}\right)^{\top}\bB_{k}\bV_{k}^{\top}\bN^{-1}\bV_{k}^{\top}\bB_{k}^{\top}\left(\bB_{k}\bar{\by}_{k-1}-\beta_{1}e_{1}\right) \\
		&= \|\bB_{k}^{\top}\left(\bB_{k}\bar{\by}_{k-1}-\beta_{1}e_{1}\right)\|_{2}^{2} .
	\end{align*}
	Then, we prove
	\begin{equation}\label{res_bound}
		\|\bA\bx_{k-1}-\bb\|_{\bM^{-1}} = \|\bB_{k}\bar{\by}_{k-1}-\beta_{1}e_{1}\|_2 \geq \sqrt{\tau m} 
	\end{equation}
	by mathematical induction. For $k=1$, we have $\|\bA\bx_0-\bb\|_{\bM^{-1}}=\|\bb\|_{\bM^{-1}}>\sqrt{\tau m}$ since $\bx_0=\mathbf{0}$. Suppose $\|\bA\bx_{k-1}-\bb\|_{\bM^{-1}}\geq \sqrt{\tau m}$ for $k\geq 1$. We have
	\begin{align*}
		& \ \ \ \ \|\bA\bx_{k}-\bb\|_{\bM^{-1}}^2
		= \|\bA\bV_{k}(\bar{\by}_{k-1}+\gamma_{k}\Delta\by_{k})-\bb\|_{\bM^{-1}}^2 \\
		&= \|\bA\bx_{k-1}-\bb\|_{\bM^{-1}}^2+\gamma_{k}^2\|\bA\bV_{k}\Delta\by_{k}\|_{\bM^{-1}}^2+2\gamma_{k}(\bA\bx_{k-1}-b)^{\top}\bM^{-1}\bA\bV_{k}\Delta\by_{k} \\
		&= \|\bB_{k}\bar{\by}_{k-1}-\beta_{1}e_{1}\|_{2}^{2}+\gamma_{k}^2\|\bA\bV_{k}\Delta\by_{k}\|_{\bM^{-1}}^2+2\gamma_{k}(\bB_{k}\bar{\by}_{k-1}-\beta_{1}e_{1})^{\top}\bB_{k}\Delta\by,
	\end{align*}
	since 
	\begin{align*}
		(\bA\bx_{k-1}-b)^{\top}\bM^{-1}\bA\bV_{k} 
		&= (\bB_{k}\bar{\by}_{k-1}-\beta_{1}e_{1})^{\top}\bU_{k+1}^{\top}\bM^{-1}\bU_{k+1}\bB_{k} \\
		&= (\bB_{k}\bar{\by}_{k-1}-\beta_{1}e_{1})^{\top}\bB_{k}.
	\end{align*}
	Writing $J^{(k)}(\bar{\by}_{k-1},\lambda_{k-1})\begin{pmatrix}
		\Delta\by_{k} \\ \Delta\lambda_k
	\end{pmatrix}=-F^{(k)}(\bar{\by}_{k-1},\lambda_{k-1})$ in the matrix form
	and using $\|\bB_{k}\bar{\by}_{k-1}-\beta_{1}e_{1}\|_2 \geq \sqrt{\tau m}$, we get from the second equality of the above equation that 
	$
		(\bB_{k}\bar{\by}_{k-1}-\beta_{1}e_{1})^{\top}\bB_{k}\Delta\by = 
		-\frac{1}{2}\left(\|\bB_{k}\bar{\by}_{k-1}-\beta_{1}e_{1}\|_{2}^2-\tau m \right) \leq 0.
	$
	Since $\gamma_{k}\leq 1$, we get
	\begin{align*}
		& \ \ \ \ \|\bA\bx_{k}-\bb\|_{\bM^{-1}}^2 \\
		&\geq \|\bB_{k}\bar{\by}_{k-1}-\beta_{1}e_{1}\|_{2}^{2}+\gamma_{k}^2\|\bA\bV_{k}\Delta\by_{k}\|_{\bM^{-1}}^2-\left(\|\bB_{k}\bar{\by}_{k-1}-\beta_{1}e_{1}\|_{2}^2-\tau m \right) \\
		&= \tau m + \gamma_{k}^2\|\bA\bV_{k}\Delta\by_{k}\|_{\bM^{-1}}^2 \geq \tau m . 
	\end{align*}
	Therefore, we prove \cref{res_bound}. 
	
	To obtain the lower bound in \cref{res_norm_bound}, we investigate two cases: $k< k_t$ and $k\geq k_t$. 

	\noindent \textbf{Case 1: $k< k_t$.} For this case, we have
	\begin{align*}
		\|\bB_{k}^{\top}\left(\bB_{k}\bar{\by}_{k-1}-\beta_{1}e_{1}\right)\|_{2}
		\geq \sigma_{\min}(\bB_{k})\|\bB_{k}\bar{\by}_{k-1}-\beta_{1}e_{1}\|_{2}
		\geq \sigma_{\min}(\underline{\bB}_{k_t})\sqrt{\tau m} > 0,
	\end{align*}
	where $\sigma_{\min}(\cdot)$ is the smallest singular value of a matrix, and $\underline{\bB}_{k_t}$ is the first $k_t\times k_t$ part of $\bB_{k_t}$.

	\noindent \textbf{Case 2: $k\geq k_t$.} For this case, we can write $\bx_{k-1}$ as $\bx_{k-1}=\bV_{k_t}\bar{\by}_{k-1}$. Remember that $\bar{\by}_{k-1}=\by_{k-1}$ if $k> k_t$. We first prove $\|\bB_{k_t}^{\top}(\bB_{k_t}\bar{\by}_{k-1}-\beta_{1}e_{1})\|_2\neq 0$. If it is not true, then $\bar{\by}_{k-1}=\argmin_{\by}\|\bB_{k_t}\by-\beta_{1}e_{1}\|_2$. By \Cref{lem:solv_ls}, $\bx_{k-1}$ is the solution to \cref{ls_MN}. Thus, it must hold that $\|\bA\bx_{k-1}-\bb\|_{\bM^{-1}}<\sqrt{\tau m}$ by \Cref{assump1}, which contradicts \cref{res_bound}. Now suppose the lower bound in \cref{res_norm_bound} is not true. Then there exists a subsequence $\{\bar{\by}_{k_j-1}\}$ with $k_j\geq k_t$ such that 
	$
		\lim_{j\rightarrow \infty}\|\bB_{k_t}^{\top}(\bB_{k_t}\bar{\by}_{k_j-1}-\beta_{1}e_{1})\|_2 = 0.
	$
	By \Cref{lem:ls_lim}, we have 
	$
		\lim_{j\rightarrow \infty}\bar{\by}_{k_j-1} = \by_{\mathrm{min}} := \argmin_{\by}\|\bB_{k_t}\by-\beta_{1}e_{1}\|_2,
	$
	leading to 
	$
		\lim_{j\rightarrow \infty}\bx_{k_j-1} = \lim_{k_j\rightarrow \infty}\bV_{k_t}\bar{\by}_{k_j-1}
		= \bV_{k_t}\by_{\mathrm{min}} .
	$
	It follows from \Cref{lem:solv_ls} that $\bV_{k_t}\by_{\mathrm{min}}=\bx_{\mathrm{min}}$, which is the solution to \cref{ls_MN}. Therefore, it must hold 
	$
		\lim_{j\rightarrow \infty}\|\bA\bx_{k_j-1}-\bb\|_{\bM^{-1}} = \|\bA\bx_{\mathrm{min}}-\bb\|_{\bM^{-1}} < \sqrt{\tau m}
	$
	by \Cref{assump1}, which contradicts \cref{res_bound}. Summarizing both the two cases, the desired result is proved.
\end{proof}

\begin{lemma}\label{lem:point_bnd}
	The points $\{(\bx_{k},\lambda_k)\}_{i=0}^{\infty}$ generated by the \textsf{PNT} algorithm lie in a bounded set of $\mathbb{R}^{n}\times\mathbb{R}^{+}$.
\end{lemma}
\begin{proof}
	First notice that $h(\bx_{0},\lambda_0)\geq h(\bx_{1},\lambda_1)\geq \cdots$. We only need to prove $\{(\bx_{k},\lambda_k)\}_{k\geq k_t}$ is bounded above. In this case, notice that $\bx_{k}=\bV_{k_t}\by_k$ and
	\[
		h(\bx_k,\lambda_k) 
		= \frac{1}{2}[\|\lambda_k\bA^{\top}\bM^{-1}(\bA\bx_k-\bb)+\bN^{-1}\bx_k\|_{\bN}^{2} + (\frac{1}{2}\|\bB_{k_t}\by_k-\beta_{1}e_1\|_{2}^{2}-\frac{\tau m}{2})^2].
	\]
	 If the points do not lie in a bounded set, there exists a subsequence $\{(\bx_{k_j},\lambda_{k_j})\}$ with $k_j\geq k_t$ such that $(\bx_{k_j},\lambda_{k_j})\rightarrow \infty$. If $\|\bx_{k_j}\|_2\rightarrow \infty$, then $\|\by_{k_j}\|_2\rightarrow \infty$, since $\bx_{k_j}=\bV_{k_t}\by_{k_j}$ and $\bV_{k_t}$ has full column rank. This leads to $\|\bB_{k_t}\by_{k_j}\|_{2}\rightarrow \infty$ since $\bB_{k_t}$ has full column rank. It follows that the second term of $h(\bx_{k_j},\lambda_{k_j})$ tends to infinity and $h(\bx_{k_j},\lambda_{k_j})\rightarrow \infty$, a contradiction. Therefore, it must hold that $\|\bx_{k_j}\|_2$ is bounded above and $\lambda_{k_j}\rightarrow \infty$. By \Cref{lem:lower_bnd}, we have $\|\lambda_{k_j}\bA^{\top}\bM^{-1}(\bA\bx_{k_j}-\bb)\|_{\bN}\geq \lambda_{k_j}C_1$. Notice that $\{\bN^{-1}\bx_{k_j}\}$ lie in a bounded set. It follows that $\|\lambda_{k_j}\bA^{\top}\bM^{-1}(\bA\bx_{k_j}-\bb)+\bN^{-1}\bx_{k_j}\|_{\bN}\rightarrow \infty$ and $h(\bx_{k_j},\lambda_{k_j})\rightarrow \infty$, also a contradiction. 
\end{proof}

\begin{lemma}\label{lem:inverse_bnd}
	There exist a positive constant $C_{2}<+\infty$ such that for any $k\geq 1$,
	\begin{equation}
		\|J^{(k)}(\bar{\by}_{k-1},\lambda_{k-1})^{-1}\|_2 \leq C_{2} .
	\end{equation}
\end{lemma}
\begin{proof}
	First, we prove that $J^{(k)}(\bar{\by}_{k-1},\lambda_{k-1})$ is always nonsingular. Write it as 
	\begin{equation*}
		J^{(k)}(\bar{\by}_{k-1},\lambda_{k-1}) 
		= \begin{pmatrix}
			\lambda_{k-1}\bB_{k_t}^{\top}\bB_{k}+\bI & \bB_{k}^{\top}(\bB_{k_t}\bar{\by}_{k-1}-\beta_{1}e_{1}) \\
			(\bB_{k}\bar{\by}_{k-1}-\beta_{1}e_{1})^{\top}\bB_{k} & 0
		\end{pmatrix}
		=: \begin{pmatrix}
			\bC_{k} & \bd_{k} \\
			\bd_{k}^{\top} & 0
		\end{pmatrix} 
	\end{equation*}
	and notice that 
	\begin{equation*}
		\begin{pmatrix}
			\bC_{k} & \bd_{k} \\
			\bd_{k}^{\top} & 0
		\end{pmatrix} 
		= \begin{pmatrix}
			\bI & \mathbf{0} \\
			-\bd_{k}^{\top}\bC_{k}^{-1} & 1
		\end{pmatrix}^{-1}
		\begin{pmatrix}
			\bC_{k} & \mathbf{0} \\
			\mathbf{0} & -\bd_{k}^{\top}\bC_{k}^{-1}\bd_{k}
		\end{pmatrix}
		\begin{pmatrix}
			\bI & -\bC_{k}^{-1}\bd_{k} \\
			\mathbf{0} & 1
		\end{pmatrix}^{-1} .
	\end{equation*}
	It follows that 
	\begin{equation}\label{matrix_inv}
		\begin{pmatrix}
			\bC_{k} & \bd_{k} \\
			\bd_{k}^{\top} & 0
		\end{pmatrix}^{-1}
		= \begin{pmatrix}
			\bI & -\bC_{k}^{-1}\bd_{k} \\
			\mathbf{0} & 1
		\end{pmatrix}
		\begin{pmatrix}
			\bC_{k}^{-1} & \mathbf{0} \\
			\mathbf{0} & -(\bd_{k}^{\top}\bC_{k}^{-1}\bd_{k})^{-1}
		\end{pmatrix}
		\begin{pmatrix}
			\bI & \mathbf{0} \\
			-\bd_{k}^{\top}\bC_{k}^{-1} & 1
		\end{pmatrix} ,
	\end{equation}
	Since $\bC_{k}$ is positive definite and $\|\bd_{k}\|_2\geq C_1>0$.

	To give an upper bound on $\|J^{(k)}(\bar{\by}_{k-1},\lambda_{k-1})^{-1}\|_2$, we only need to consider $k\geq k_t$, where $\bB_{k}=\bB_{k_t}$ in $J^{(k)}(\bar{\by}_{k-1},\lambda_{k-1})$.	Since $\sigma_{\min}(\bC_{k})=\sigma_{\min}(\lambda_{k-1}\bB_{k_t}^{\top}\bB_{k_t}+\bI)\geq 1$, we have $\|\bC_{k}^{-1}\|_{2}\leq 1$. 
	By \Cref{lem:point_bnd}, there exist a positive constant $C_3<+\infty$ such that $\lambda_{k}\leq C_3$, thereby
	\[
		\sigma_{\max}(\bC_{k}) \leq \sigma_{\max}(\lambda_{k-1}\bB_{k_t}^{\top}\bB_{k_t}) + \sigma_{\max}(\bI)
		\leq C_3\sigma_{\max}(\bB_{k_t}^{\top}\bB_{k_t}) + 1 =: \bar{C}_3 .
	\]
	By \Cref{lem:lower_bnd} we have $\|\bd_{k}\|_{2}=\|\bB_{k_t}^{\top}(\bB_{k_t}\bar{\by}_{k-1}-\beta_{1}e_{1})\|_{2}\geq C_1$. 
	On the other hand, by \Cref{lem:point_bnd} we know that $\|\bx_{k-1}\|_2=\|\bV_{k_t}\bar{\by}_{k-1}\|_2$ is bounded above, thereby $\|\bar{\by}_{k-1}\|_{2}$ is bounded above since $\bV_{k_t}$ has full column rank. Thus, there exists a positive constant $\bar{C}_1$ such that $\|\bd_{k}\|_2\leq\bar{C}_1$, leading to
	\[
		\|\bC_{k}^{-1}\bd_{k}\|_{2} \leq 
		\|\bC_{k}^{-1}\|_2\|\bd_{k}\|_{2}
		\leq \bar{C}_{1}\sigma_{\min}(\bC_{k})^{-1}
		\leq \bar{C}_{1} ,
	\]
	and
	\[
		\bd_{k}^{\top}\bC_{k}^{-1}\bd_{k} \geq 
		\sigma_{\min}(\bC_{k}^{-1})\|\bd_{k}\|_{2}^2 =
		\sigma_{\max}(\bC_{k})^{-1}\|\bd_{k}\|_{2}^2
		\geq C_{1}^2/\bar{C}_3 > 0 .
	\]
	Therefore, we have
	\begin{align*}
		\left\|\begin{pmatrix}
			\mathbf{0} & -\bC_{k}^{-1}\bd_{k} \\
			\mathbf{0} & 0
		\end{pmatrix}\right\|_{2}^2 
		&=
		\left\|\begin{pmatrix}
			\mathbf{0} & -\bC_{k}^{-1}\bd_{k} \\
			\mathbf{0} & 0
		\end{pmatrix}^{\top}\begin{pmatrix}
			\mathbf{0} & -\bC_{k}^{-1}\bd_{k} \\
			\mathbf{0} & 0
		\end{pmatrix}\right\|_{2}\\
		&= 
		\left\|\begin{pmatrix}
			\mathbf{0} & \\
			 & \|\bC_{k}^{-1}\bd_{k}\|_{2}^2
		\end{pmatrix}\right\|_{2}\leq \bar{C}_{1}^2
	\end{align*}
	and 
	\begin{equation*}
		\left\|\begin{pmatrix}
			\bI & -\bC_{k}^{-1}\bd_{k} \\
			\mathbf{0} & 1
		\end{pmatrix}\right\|_{2} \leq 
		\left\|\begin{pmatrix}
			\bI &  \\
			 & 1
		\end{pmatrix}\right\|_{2} +
		\left\|\begin{pmatrix}
			\mathbf{0} & -\bC_{k}^{-1}\bd_{k} \\
			\mathbf{0} & 0
		\end{pmatrix}\right\|_{2} \leq 1+\bar{C}_1 .
	\end{equation*}
	Similarly, we have
	\begin{equation*}
		\left\|\begin{pmatrix}
		\bC_{k}^{-1} & \mathbf{0} \\
		\mathbf{0} & -(\bd_{k}^{\top}\bC_{k}^{-1}\bd_{k})^{-1}
	\end{pmatrix}\right\|_2 \leq \max\{1, \bar{C}_{3}/C_{1}^{2}\} .
	\end{equation*}
	Using the expression of $J^{(k)}(\bar{\by}_{k-1},\lambda_{k-1})$ in \cref{matrix_inv}, we finally obtain the desired result.
\end{proof}

\begin{lemma}\label{lem:gamma}
	There exists a positive constant $C_4$ such that for any $k\geq 1$,
	\begin{equation}
		\gamma_{k} \geq C_{4} > 0 .
	\end{equation}
\end{lemma}
\begin{proof}
	By \Cref{thm:Armijo} and \Cref{thm:descent}, at each iteration the Armijo backtracking line search must terminate in finite steps with a $\gamma_k$ satisfying
	\begin{equation*} \label{eq:lower}
		\gamma_k \geq \min\left\{1, \frac{4(1-c)\eta h(\bx_{k-1},\lambda_{k-1})}{\zeta(\bx_{k-1},\lambda_{k-1})\|(\Delta\bx_{k}^{\top},\Delta\lambda_{k})\|_{2}^{2}} \right\} ,
	\end{equation*}
	where $\zeta(\bx_{k-1},\lambda_{k-1})$ is the Lipschitz constant of $\nabla h$ at $(\bx_{k-1},\lambda_{k-1})$; see also \cite[pp. 122, Corollary 2.1]{blowey2003frontiers}. Now we prove $\zeta(\bx_{k-1},\lambda_{k-1})$ are bounded above. Notice that $\nabla h(\bx,\lambda) = J(\bx,\lambda)\widehat{N}F(\bx,\lambda)$. Thus, all the elements in the Jacobian of $\nabla h(\bx,\lambda)$ are polynomials of $(\bx,\lambda)$ with degrees not bigger than 4. Since $\{(\bx_{k-1},\lambda_{k-1})\}$ lie in a bounded set, the norms of the Jacobians of $\nabla h(\bx,\lambda)$ at the points $\{(\bx_{k-1},\lambda_{k-1})\}$ are bounded above. Therefore, the Lipschitz constants $\zeta(\bx_{k-1},\lambda_{k-1})$ are bounded above. 
	
	Let $\zeta(\bx_{k-1},\lambda_{k-1})\leq \zeta_0$ with $0<\zeta_0<+\infty$ for any $k\geq 1$. Then by \Cref{lem:inverse_bnd} and \Cref{lem:proj_F}, we have
	\begin{align*}
		\left\| \begin{pmatrix}
			\Delta\bx_{k} \\ \Delta\lambda_k
		\end{pmatrix} \right\|_2 
		&\leq 
		\left\|\begin{pmatrix}
			\bV_{k} &  \\
			 & 1
		\end{pmatrix}\right\|_2
		\left\|\begin{pmatrix}
			\Delta\by_{k} \\ \Delta\lambda_k
		\end{pmatrix}\right\|_2 \\
		&\leq (\|\bV_{k_t}\|_{2}+1)\|J^{(k)}(\bar{\by}_{k-1},\lambda_{k-1})^{-1}\|_2\|F^{(k)}(\bar{\by}_{k-1},\lambda_{k-1})\|_{2} \\
		&\leq C_{2}(\|\bV_{k_t}\|_{2}+1)(2h(\bx_{k-1},\lambda_{k-1}))^{1/2} .
	\end{align*}
	Then we obtain
	\begin{align*}
		\gamma_k \geq \min\left\{1, \frac{4(1-c)\eta h(\bx_{k-1},\lambda_{k-1})}{\zeta_0\|(\Delta\bx_{k}^{\top},\Delta\lambda_{k})^{\top}\|_{2}^{2}} \right\} 
		\geq \min\left\{1, \frac{2(1-c)\eta }{\zeta_0C_{2}^{2}(\|\bV_{k_t}\|_{2}+1)^2}\right\} =: C_{4} .
	\end{align*}
	The desired result is obtained.
\end{proof}

\begin{proofof}{\Cref{thm:conv}}
	By \Cref{lem:point_bnd}, the sequence $\{(\bx_{k},\lambda_k)\}_{k=1}^{\infty}$ is contained in a bounded set, thereby there exists a convergent subsequence $\{(\bx_{k_j},\lambda_{k_j})\}_{j=1}^{\infty}$. Suppose $(\bx_{k_j},\lambda_{k_j})\rightarrow (\widehat{\bx},\widehat{\lambda})$. It follows that $h(\bx_{k_j},\lambda_{k_j})\rightarrow h(\widehat{\bx},\widehat{\lambda})$ since $h(\bx,\lambda)$ is continuous. Note that $h(\bx_{k_j},\lambda_{k_j})$ is nonincreasing, thereby $h(\widehat{\bx},\widehat{\lambda})\leq h(\bx_{k_j},\lambda_{k_j})$ for any $k_j$. Thus, for any $\varepsilon>0$, there exist a $k_{\star}\in\mathbb{N}$ such that 
	$
		h(\bx_{k_j},\lambda_{k_j}) < h(\widehat{\bx},\widehat{\lambda})+\varepsilon, \ k_{j}>k_{\star}.
	$
	Select one $k_j$ that satisfies $k_{j}>k_{\star}$. For any $k\geq k_j$, we have
	$
		h(\bx_{k},\lambda_{k}) \leq h(\bx_{k_j},\lambda_{k_j}) < h(\widehat{\bx},\widehat{\lambda})+\varepsilon ,
	$
	which means that 
	$
		\lim_{k\rightarrow\infty}h(\bx_{k},\lambda_{k}) = h(\widehat{\bx},\widehat{\lambda}) .
	$
	The Armijo condition and \Cref{thm:descent} lead to
	$
		h(\bx_{k+1},\lambda_{k+1}) - h(\bx_{k},\lambda_{k}) \leq c\gamma_{k}\left(\Delta\bx_{k}^{\top},\Delta\lambda_{k}\right)\nabla h(\bx_{k},\lambda_{k}) \leq 0 .
	$
	Taking the limit on both sides leads to 
	$
		\lim_{k\rightarrow\infty}c\gamma_{k}\left(\Delta\bx_{k}^{\top},\Delta\lambda_{k}\right)\nabla h(\bx_{k},\lambda_{k}) = 0.
	$
	By \Cref{lem:gamma} we get $\lim_{k\rightarrow\infty}\left(\Delta\bx_{k}^{\top},\Delta\lambda_{k}\right)\nabla h(\bx_{k},\lambda_{k}) = 0$. Noticing by \Cref{thm:descent} that $-2h(\bx_{k},\lambda_{k})=\left(\Delta\bx_{k}^{\top},\Delta\lambda_{k}\right)\nabla h(\bx_{k},\lambda_{k})$, we get
	$
		h(\widehat{\bx},\widehat{\lambda}) = \lim_{k\rightarrow\infty}h(\bx_{k},\lambda_{k}) = 0 .
	$
	This proves the desired result.
\end{proofof}

Now we can give the convergence result of $(\bx_{k},\lambda_k)$.
\begin{corollary}\label{coro:converge}
	The sequence $\{(\bx_{k},\lambda_k)\}_{k=0}^{\infty}$ generated by the \textsf{PNT} algorithm eventually converges to $(\bx^{*},\lambda^{*})$, i.e. the solution of \cref{discrepancy} and the corresponding Lagrange multiplier.
\end{corollary}
\begin{proof}
	Using the same notations as the proof of \Cref{thm:conv}, we obtain that $(\widehat{\bx},\widehat{\lambda})=(\bx^{*},\lambda^{*})$, since $h(\bx,\lambda)$ has the unique zero point $(\bx^{*},\lambda^{*})$. Therefore, the subsequence $\{(\bx_{k_j},\lambda_{k_j})\}_{j=1}^{\infty}$ defined in the proof of \Cref{thm:conv} converges to $(\bx^{*},\lambda^{*})$. Now we need to prove the whole sequence $\{(\bx_{k},\lambda_{k})\}_{k=1}^{\infty}$ converges to $(\bx^{*},\lambda^{*})$. Assume that there is a subsequence $\{(\bx_{l_j},\lambda_{l_j})\}_{j=1}^{\infty}$ that does not converge to $(\bx^{*},\lambda^{*})$. We can select a subsequence from $\{(\bx_{l_j},\lambda_{l_j})\}_{j=1}^{\infty}$ that converges to a point $(\bar{\bx},\bar{\lambda})\neq (\bx^{*},\lambda^{*})$. Since $h(\bx_{l_j},\lambda_{l_j})$ is nonincreasing with respect to $j$, using the same procedure as the proof of \Cref{thm:conv}, we can obtain again that $h(\bar{\bx},\bar{\lambda})=0$, leading to $(\bar{\bx},\bar{\lambda})= (\bx^{*},\lambda^{*})$, a contradiction. Therefore, any subsequence of $\{(\bx_{k},\lambda_k)\}_{k=0}^{\infty}$ converges to $(\bx^{*},\lambda^{*})$, thereby $\{(\bx_{k},\lambda_k)\}_{k=0}^{\infty}$ converges to $(\bx^{*},\lambda^{*})$.
\end{proof}

\section{Experimental results}\label{sec5}
We test the \textsf{PNT} method and compare it with the standard \textsf{Newton} method, which refers to the method in \cite{landi2008lagrange} but \cref{Newton_direc} is solved using direct matrix inversions. These two methods use the same initialization and backtracking line search strategy. The setting of hyperparameters follows \Cref{alg:PNTM}, and we set $\tau=1.001$ and $\lambda_0=0.1$ in all the experiments. We also implement the generalized hybrid iterative method proposed in \cite{Chung2017} (denoted by \textsf{genHyb}), which is also based on \textsf{gen-GKB}. The \textsf{genHyb} iteratively computes approximations to $\mu_{opt}$ and $\bx_{opt}=\bx(\mu_{opt})$, where $\mu_{opt}$ is the optimal Tikhonov regularization parameter, that is $\mu_{opt}=\min_{\mu>0}\|\bx({\mu})-\bx_{\text{true}}\|_{2}$; the $k$-th approximate Lagrangian multiplier is $\lambda_{k}=1/\mu_{k}$. All the experiments are performed on MATLAB R2023b. The codes are available at \url{https://github.com/Machealb/InverProb_IterSolver}. 

All the inverse problems in the experiments are ill-posed and satisfy \Cref{assump1}. We use three types of ill-posed inverse problems to test the proposed method. The characteristics of these problems are summarized in \Cref{tab1}.

\begin{table}[htp]
	\centering
	\caption{Properties of the inverse problems in the experiments.}
	\scalebox{1.1}{
		\begin{tabular}{*{5}{c}}
			\toprule[0.6pt]
			Problem 	& $m\times n$     & Ill-posedness  & Description \\
			\midrule
			{\sf heat}  &$2000\times 2000$& moderate   & inverse heat equation  \\
			{\sf shaw} &$3000\times 3000$   & severe   & 1D image restoration  \\
			{\sf PRblurshake} & $128^2\times 128^2$  & mild  & 2D image deblurring  \\
			{\sf PRblurspeckle}  & $128^2\times 128^2$  & mild  & 2D image deblurring  \\
			{\sf PRspherical}  & $23168\times 128^2$  & mild  & computed tomography  \\
			\bottomrule[0.6pt]
	\end{tabular}}
	\label{tab1}
\end{table}

\subsection{Small-scale problems}
We choose two small-scale 1D inverse problems from \cite{Hansen2007}. The first problem is {\sf heat}, an inverse heat equation described by the Volterra integral equation of the first kind on $[0,1]$.
The second problem is {\sf shaw}, a one-dimensional image restoration model described by the Fredholm integral equation of the first kind on $[-\pi/2,\pi/2]$.
We use the code in \cite{Hansen2007} to discretize the two problems to generate $\bA$, $\bx_{\text{true}}$ and $\bb_{\text{true}}=\bA\bx_{\text{true}}$, where $m=n=2000$ and $m=n=3000$ for {\sf heat} and {\sf shaw}, respectively. We set the noisy observation vector $\bb$ as $\bb=\bb_{\text{true}}+\bepsilon$, where $\bepsilon$ is a Gaussian noise. For {\sf heat}, we set $\bepsilon$ as a white noise (i.e. $\bM$ is a scalar matrix) with noise level $\varepsilon:=\|\bepsilon\|_{2}/\|\bb_{\text{true}}\|_{2}=5\times 10^{-2}$; for {\sf shaw}, we set $\bepsilon$ as a uncorrelated non-white noise (i.e. $\bM$ is a diagonal matrix) with noise level $\varepsilon=10^{-2}$. The true solutions and noisy observed data for these two problems are shown in \Cref{fig1}.

\begin{figure}[htbp]
	\centering
	\subfloat 
	{\label{fig:1a}\includegraphics[width=0.26\textwidth]{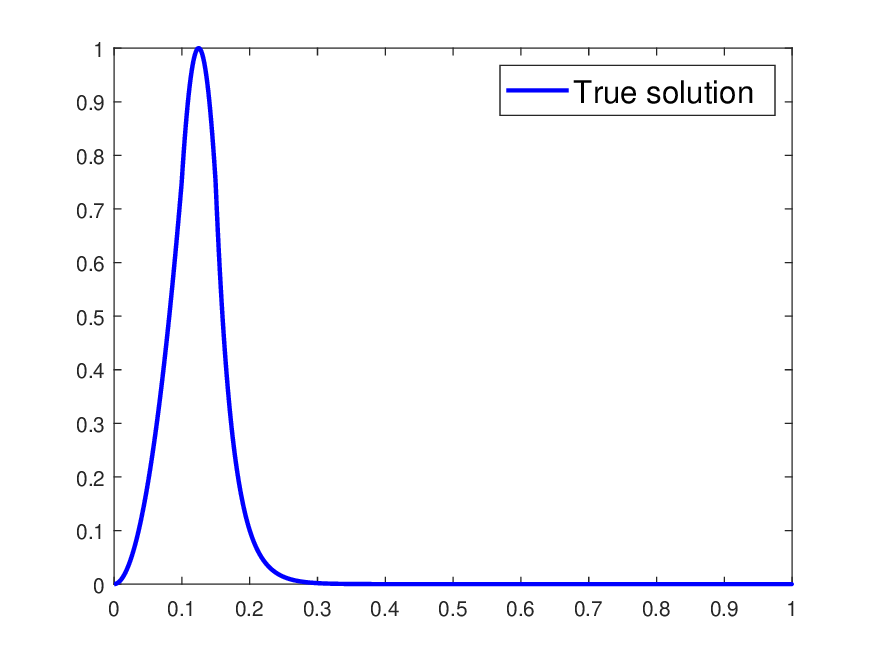}}\hspace{-3mm}
	\subfloat{\label{fig:1b}\includegraphics[width=0.26\textwidth]{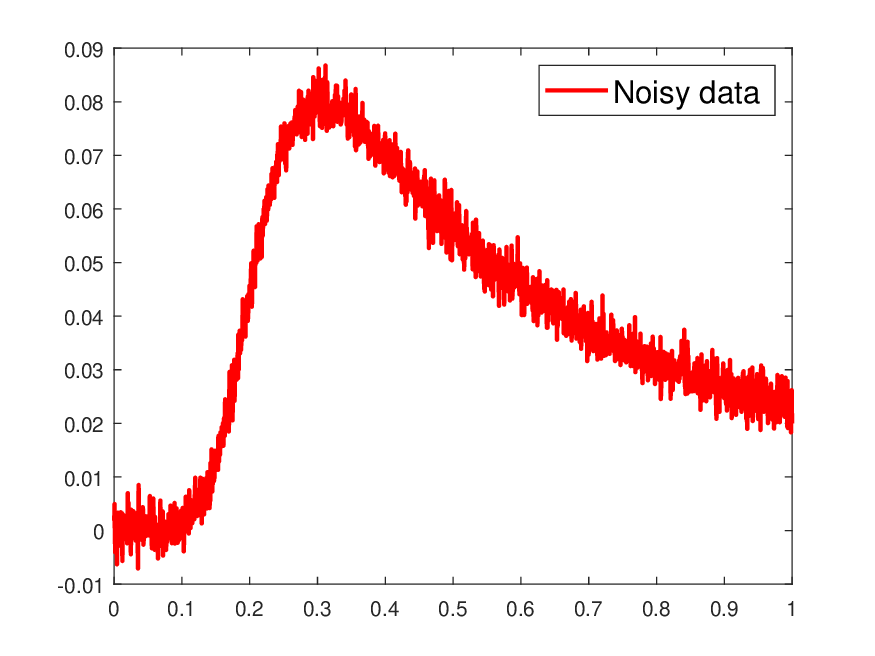}}\hspace{-3mm}
	\subfloat 
	{\label{fig:1c}\includegraphics[width=0.26\textwidth]{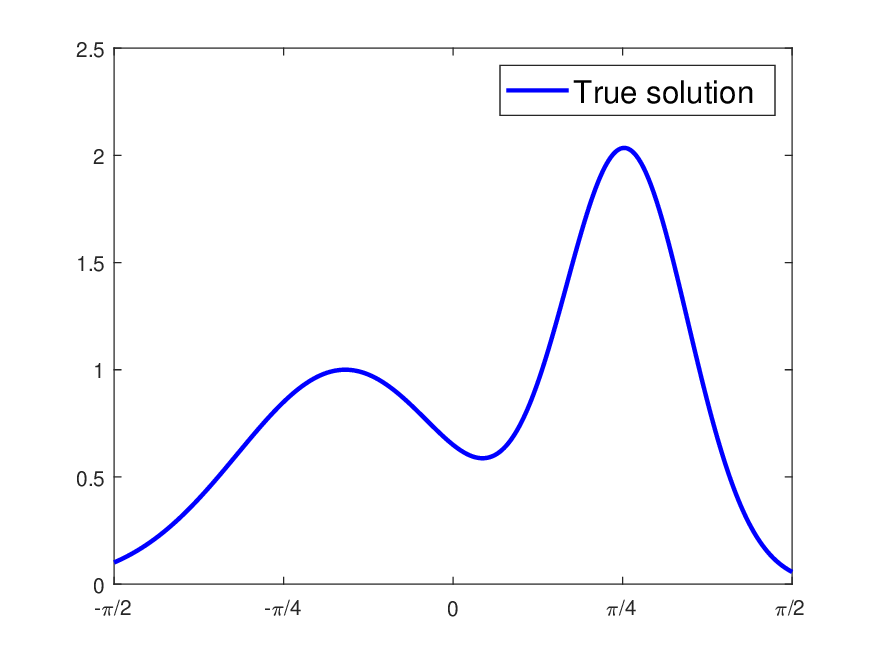}}\hspace{-3mm}
	\subfloat{\label{fig:1d}\includegraphics[width=0.26\textwidth]{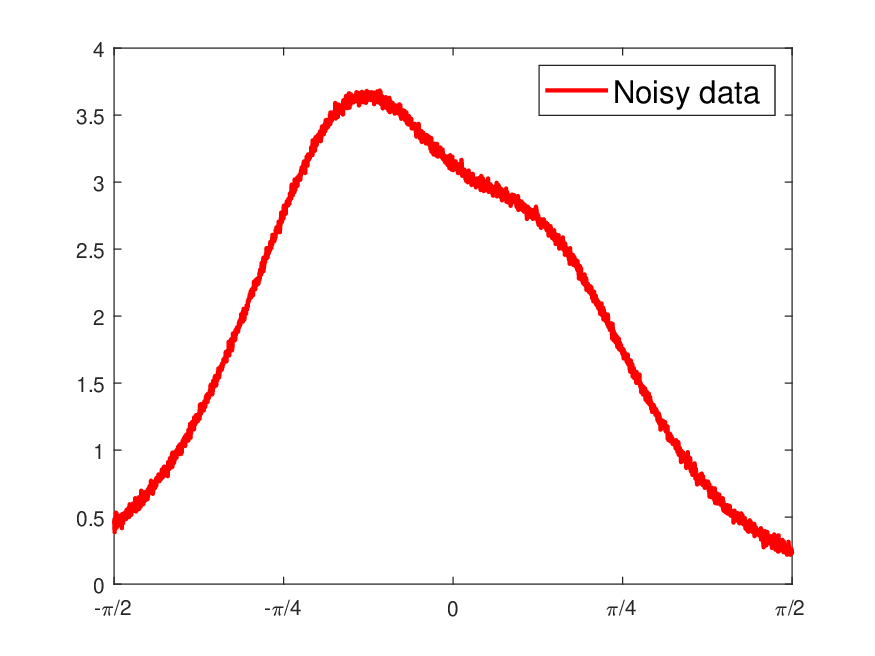}}
	\vspace{-8mm}
	\caption{True solution and noisy observed data. The first two: {\sf heat}. The last two: {\sf shaw}.}
	\label{fig1}
\end{figure}

For {\sf heat}, we assume a Gaussian prior $\bx\sim\mathcal{N}(\boldsymbol{0}, \mu^{-1}\bN)$ with $\bN$ coming from the Gaussian kernel $\kappa_{G}$, i.e. the $ij$ element of $\bN$ is
$
	[\bN]_{ij} = K_{G}(r_{ij}), \  K_{G}(r):=\exp\left(-r^2/(2l^2)\right) , 
$
where $r_{ij}=\|\bp_i-\bp_j\|_{2}$ and $\{\bp_{i}\}_{i=1}^{n}$ are discretized points of the domain of the true solution; the parameter $l$ is set as $l=0.1$. For {\sf shaw}, we construct $\bN$ using the exponential kernel
$
	K_{exp}(r):= \exp\left(-(r/l)^\nu\right),
$
where the parameters $l$ and $\nu$ are set as $l=0.1$ and $\nu=1$. We set $\tau=1.001$ for both the two problems. We factorize $\bM^{-1}$ and $\bN^{-1}$ to form \eqref{gen_regu} and solve it directly to find $\mu_{opt}$ and $\bx_{opt}$; the corresponding Lagrangian multiplier is $\lambda_{opt}=1/\mu_{opt}$. We also compute the $\mu$ of \cref{DP0} and the corresponding regularized solution, which is denoted by $\mu_{DP}$ and $\bx_{DP}$, respectively. Therefore, the solution to \cref{discrepancy,lagr} is $(\bx^{*},\lambda^{*})=(\bx_{DP},1/\mu_{DP})$. We use the optimal Tikhonov solution and the DP solution as the baseline for the subsequent tests.

For these two small-scale problems, we also implement the projected Newton method in \cite{cornelis2020projected1} based on the transformation \cref{gen_regu} as a comparison. This means that we solve
\[ \min_{\bar{\bx} \in \mathbb{R}^{n}}\{\|(\bL_{M}\bA\bL_{N})\bar{\bx}-\bL_{M}\bb\|_{2}^{2} + \mu\|\bar{\bx}\|_{2}^2\},\]
using the method in \cite{cornelis2020projected1} and then compute the regularized solution $\bx_{k}=\bL_{N}^{-1}\bar{\bx}_{k}$. This Cholesky factorization based method is abbreviated as \textsf{Ch-PNT}.

\begin{figure}[htbp]
	\centering
	\subfloat 
	{\label{fig:2a}\includegraphics[width=0.33\textwidth]{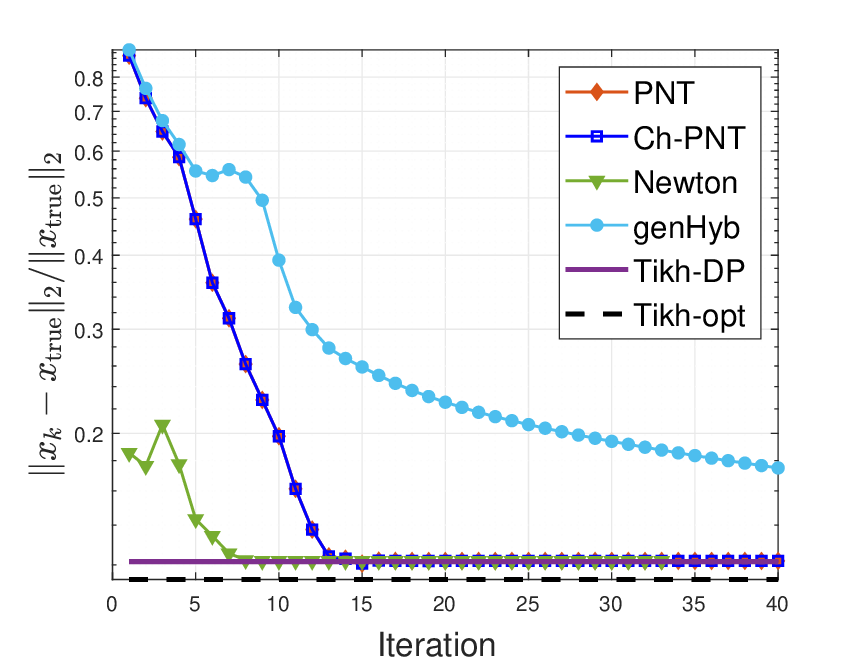}}\hspace{-4.mm}
	\subfloat
	{\label{fig:2b}\includegraphics[width=0.35\textwidth]{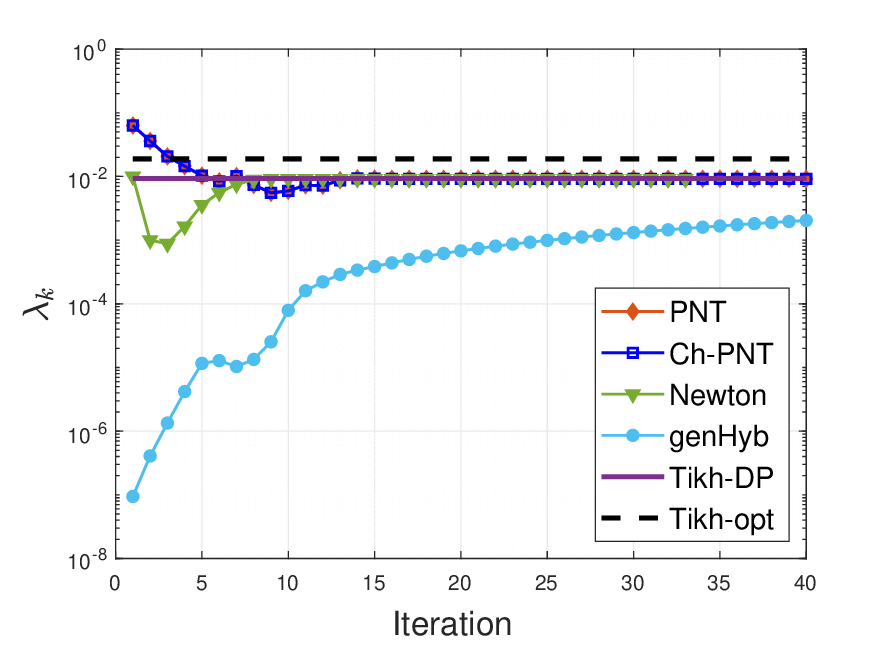}}\hspace{-4.mm}
	\subfloat
	{\label{fig:2c}\includegraphics[width=0.35\textwidth]{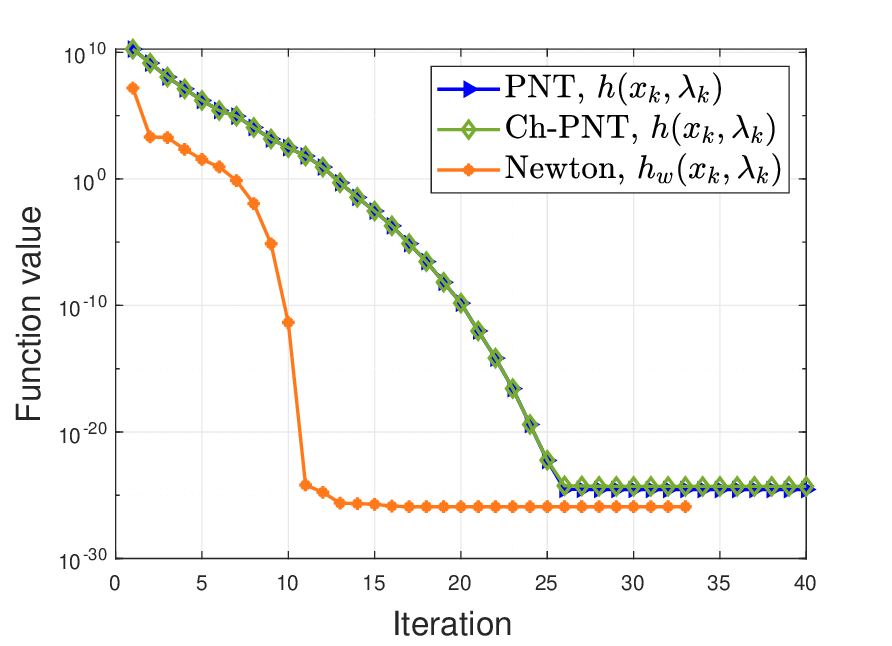}}
	\vspace{-4mm}
	\subfloat 
	{\label{fig:2d}\includegraphics[width=0.34\textwidth]{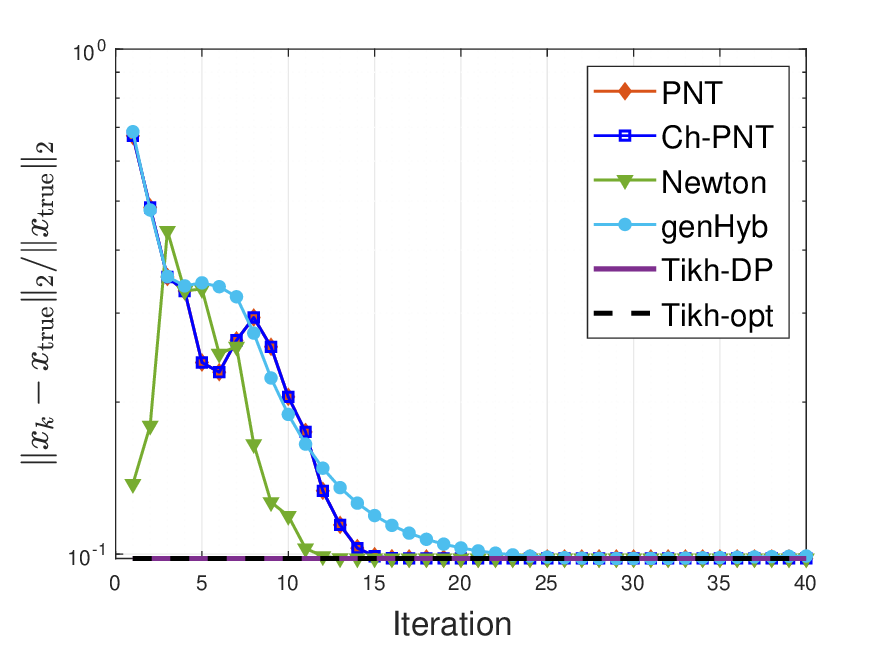}}\hspace{-4.mm}
	\subfloat
	{\label{fig:2e}\includegraphics[width=0.34\textwidth]{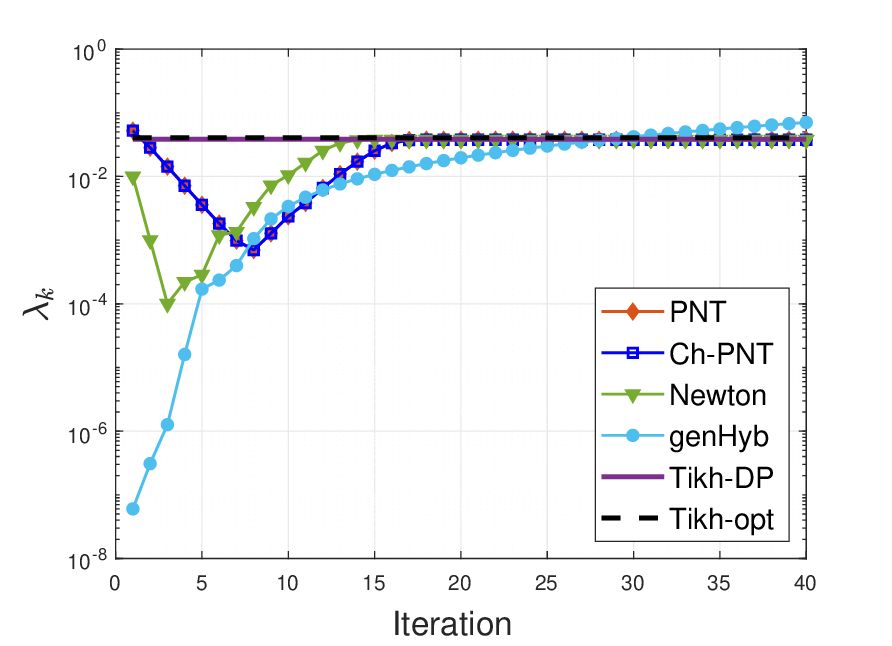}}\hspace{-4.mm}
	\subfloat
	{\label{fig:2f}\includegraphics[width=0.35\textwidth]{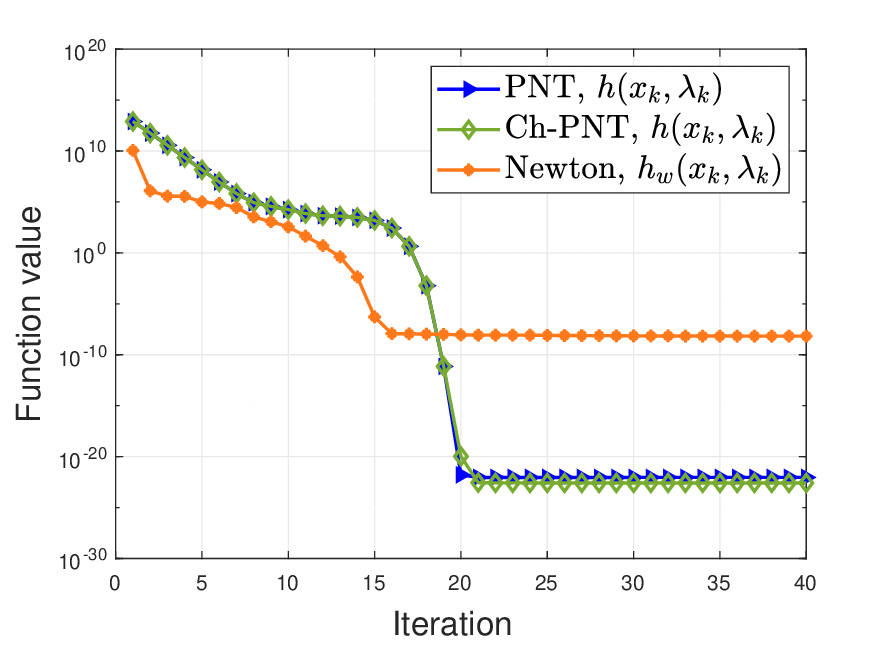}}
	\vspace{-3mm}
	\caption{Relative errors of iterative solutions, convergence of $\lambda_k$, and convergence of merit functions. Top: {\sf heat}. Bottom: {\sf shaw}.}
	\label{fig2}
\end{figure}

We compare the convergence behaviors of \textsf{PNT}, \textsf{Ch-PNT}, \textsf{Newton} and \textsf{genHyb} methods by plotting the relative error curve of $\bx_k$ with respect to $\bx_{\text{true}}$ and the convergence curves of $\lambda_k$ and merit functions. The solutions $(\bx_{DP},1/\mu_{DP})$ and $(\bx_{opt},1/\mu_{opt})$ are used as baselines. From \Cref{fig2} we find that both \textsf{PNT} and \textsf{Newton} methods converge very fast to $\bx_{DP}$ and $\lambda_{DP}:=1/\mu_{DP}$ with very few iterations, and \textsf{PNT} converges only slightly slower than \textsf{Newton}. 
We also find that the convergence behaviors of PNT and Ch-PNT are almost identical. This is not surprising, as both methods utilize the same subspaces for projecting the large-scale system and employ the same hyperparameters and update procedures. For {\sf heat}, the error of the DP solution is slightly higher than the optimal Tikhonov solution, because DP slightly under-estimates $\lambda$. The merit functions of both \textsf{PNT} and \textsf{Newton} decrease monotonically, and $h({\bx_k,\lambda_k})$ of \textsf{PNT} eventually decreases to an extremely small value for the two problems. We remark that we set $w=1$ for $h_{w}(\bx,\lambda)$ in all the tests. For \textsf{Newton} method for {\sf heat}, we stop the iterate at $k=34$ because the step-length $\gamma_k$ is too small. In comparison, \textsf{genHyb} converges much slower than the previous two methods, especially for {\sf heat}.

\begin{figure}[htbp]
	\centering
	\subfloat 
	{\label{fig:3a}\includegraphics[width=0.325\textwidth]{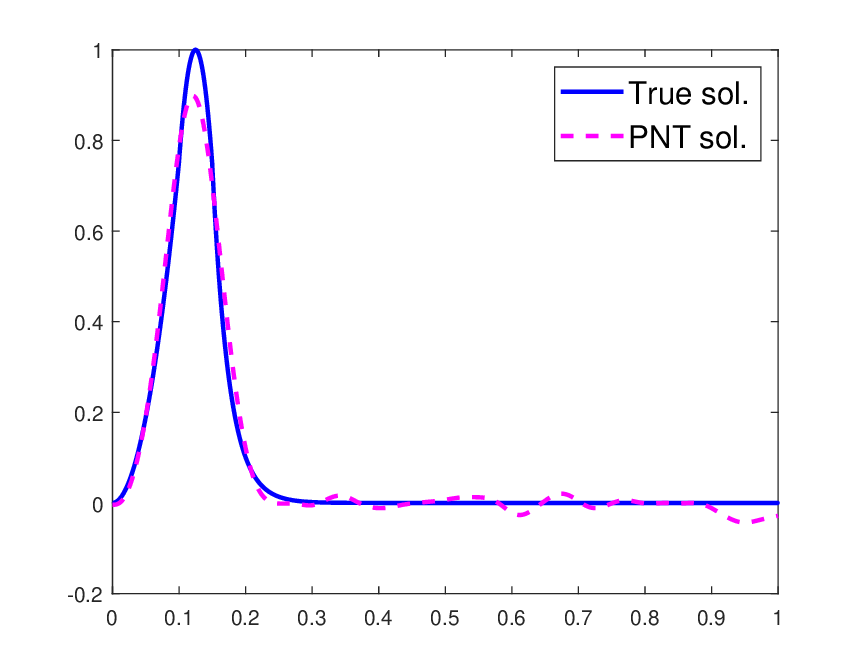}}\hspace{-3.5mm}
	\subfloat
	{\label{fig:3b}\includegraphics[width=0.34\textwidth]{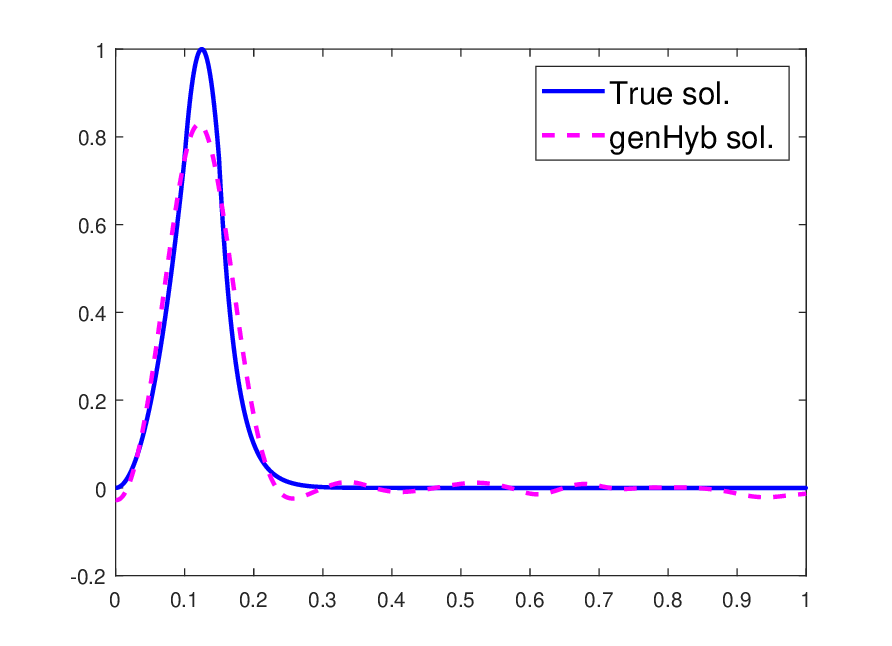}}\hspace{-3.5mm}
	\subfloat
	{\label{fig:3c}\includegraphics[width=0.34\textwidth]{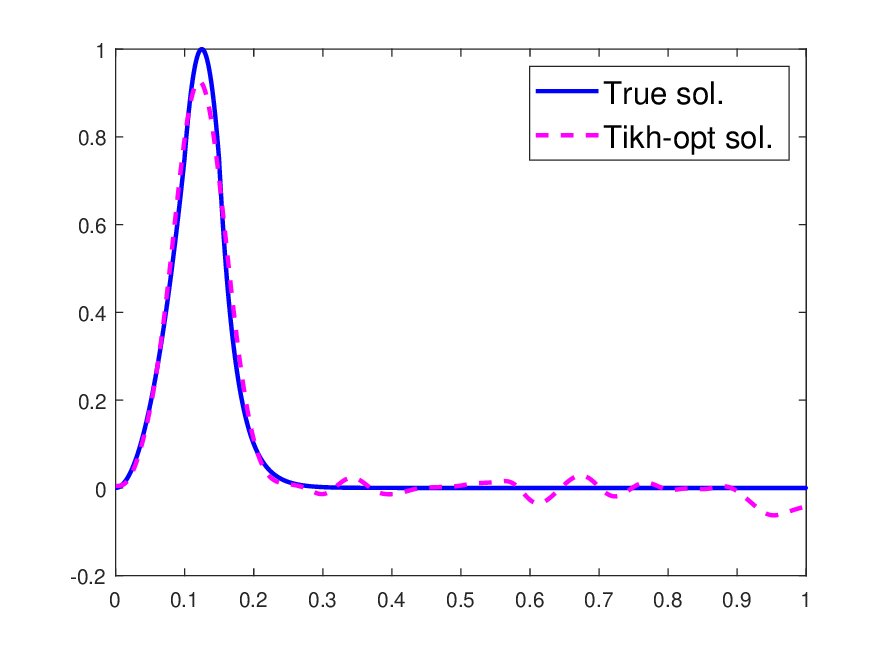}}
	\vspace{-5mm}
	\subfloat 
	{\label{fig:3d}\includegraphics[width=0.33\textwidth]{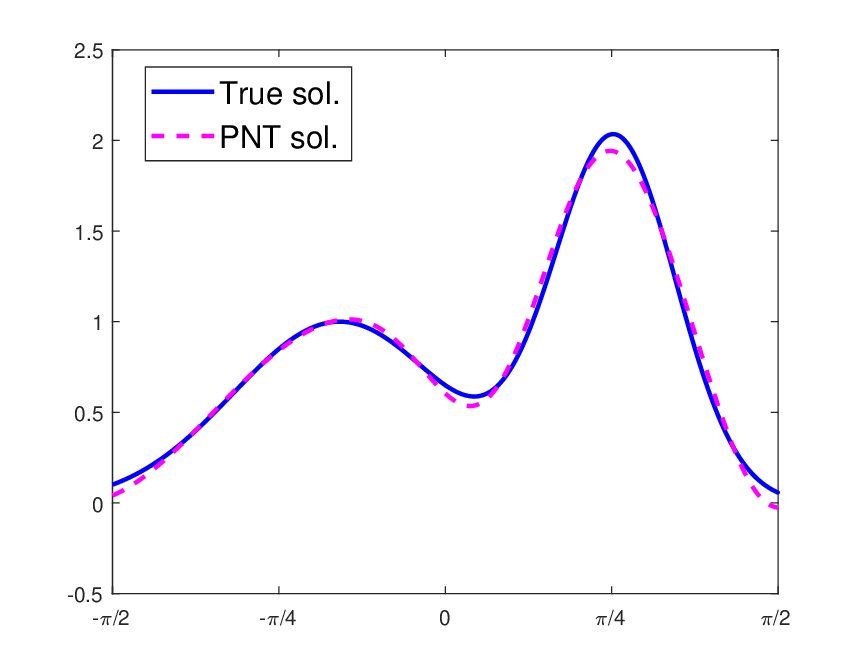}}\hspace{-3.5mm}
	\subfloat
	{\label{fig:3e}\includegraphics[width=0.33\textwidth]{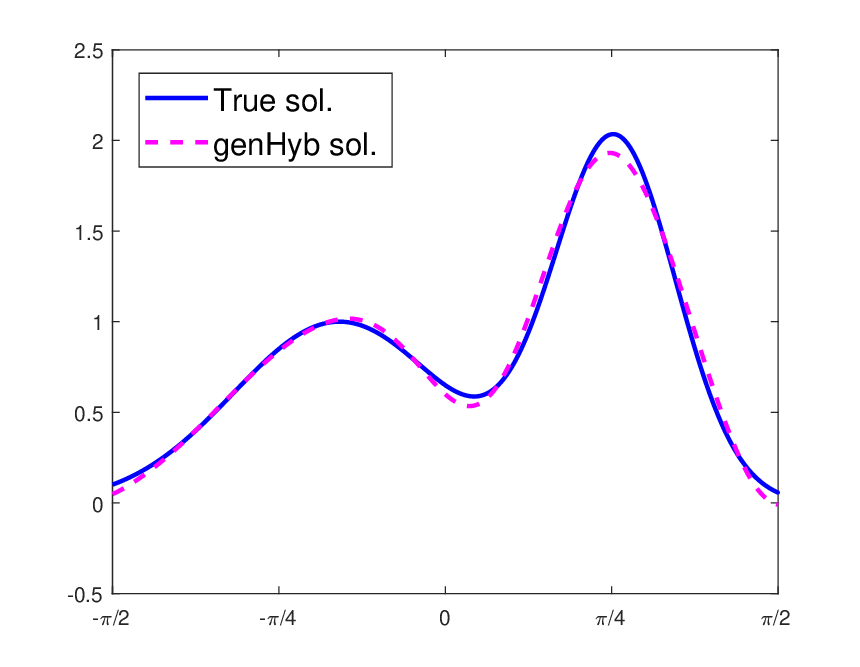}}\hspace{-3.5mm}
	\subfloat
	{\label{fig:3f}\includegraphics[width=0.34\textwidth]{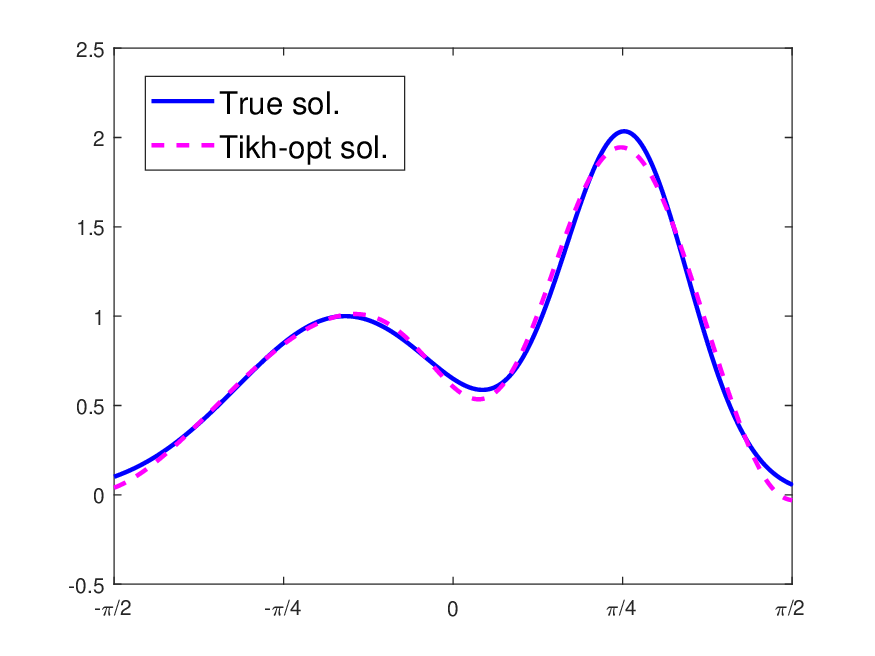}}
	\vspace{-3mm}
	\caption{Comparison of reconstructed solutions at the final iterations with the optimal Tikhonov regularized solution. Top: {\sf heat}. Bottom: {\sf shaw}.}
	\label{fig3}
\end{figure}

\Cref{fig3} plots the recovered solutions computed by \textsf{PNT} and \textsf{genHyb} methods at the final iterations; the solution by \textsf{Newton} is almost the same as that by \textsf{PNT}, thereby we omit it. We also plot the optimal Tikhonov regularized solution as a comparison, where the DP solution is very similar and omitted. Both \textsf{PNT} and \textsf{genHyb} can recover good regularized solutions, and \textsf{PNT} is slightly better for {\sf heat}.

\begin{figure}[htbp]
	\centering
	\subfloat 
	{\label{fig:35a}\includegraphics[width=0.4\textwidth]{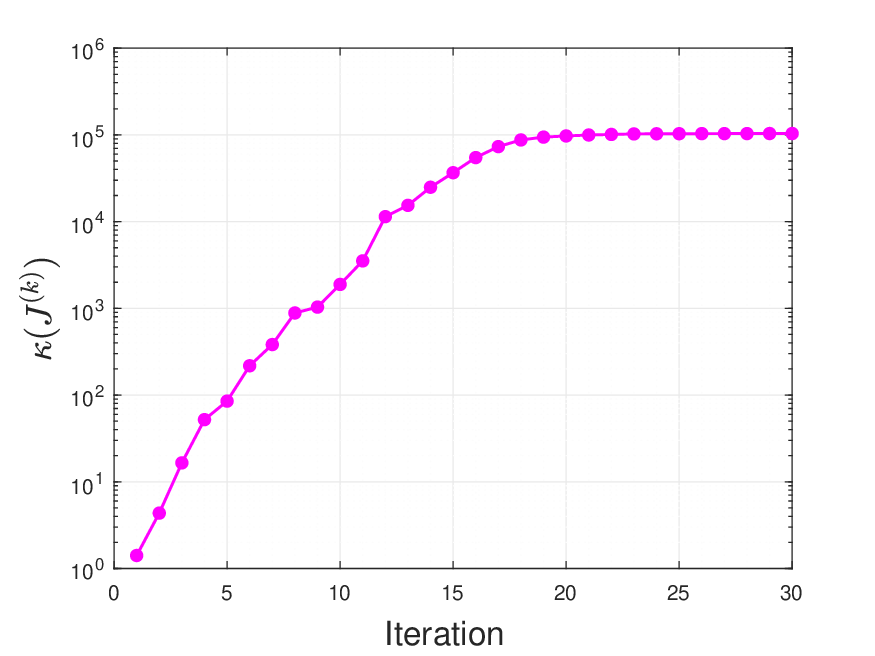}}\hspace{-1mm}
	\subfloat
	{\label{fig:35b}\includegraphics[width=0.42\textwidth]{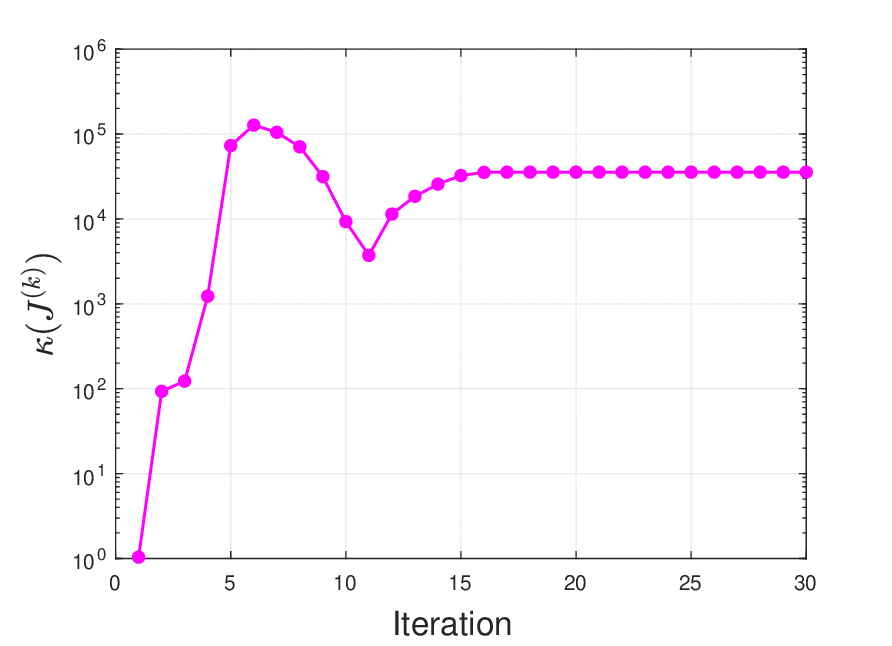}}\hspace{-1mm}
	\vspace{-3mm}
	\caption{Variation of the condition number of $J^{(k)}(\bar{\by}_{k-1},\lambda_{k-1})$ during the iteration of \textsf{PNT}. Left: {\sf heat}. Right: {\sf shaw}.}
	\label{fig35}
\end{figure}

To further demonstrate the performance of \textsf{PNT}, we present the variation of the condition number of $J^{(k)}(\bar{\by}_{k-1},\lambda_{k-1})$ during the iteration of \textsf{PNT} in \Cref{fig35}. This condition number is denoted by $\kappa(J^{(k)})$ in the two pictures. We observe that the condition number  does not increase significantly during the iteration. This ensures that the small-scale linear system \cref{proj_Newton1} can be solved directly via matrix inversions without any issues.

\begin{table}[htp]
	\centering
	\caption{Running time (measured in seconds) of \textsf{PNT}, \textsf{Ch-PNT} and \textsf{Newton} methods as the scale of the problems increasing from $n=1000$ to $n=5000$. Both the two methods stop at the first $k$ (in parentheses) such that $\left|\|\bA\bx_k-\bb\|_{\bM^{-1}}^2-\tau m\right|\leq 10^{-8}$. The ratio of the running time between \textsf{PNT} and \textsf{Ch-PNT} is denoted as \textbf{ratio-1}, while the ratio of the running time between \textsf{PNT} and \textsf{Newton} is denoted as \textbf{ratio-2}.}
	\scalebox{1.0}{
		\begin{tabular}{*{6}{c}}
			\toprule[0.6pt]
			$n$ 	& 1000   & 2000   & 3000  & 4000 & 5000 \\
			\midrule
			&  \multicolumn{4}{c}{{\sf heat}} &    \\
			\midrule
			\textsf{PNT}  &$0.021$ ($18$) & $0.114$ ($21$) & $0.164$ ($19$) & $0.347$ ($19$) & $0.492$ ($19$) \\
			\textsf{Ch-PNT} &$0.032$ ($18$) & $0.280$ ($21$)  & $0.375$ ($19$)  & $0.764$ ($19$) & $1.314$ ($19$) \\
			\textsf{Newton} &$0.249$ ($10$) & $2.568$ ($11$)  & $6.062$ ($10$)  & $13.914$ ($10$) & $26.127$ ($11$) \\
			\textbf{ratio-1}  & \textbf{1.5} & \textbf{2.5} & \textbf{2.3} & \textbf{2.2} & \textbf{2.7}  \\
			\textbf{ratio-2}  & \textbf{11.9} & \textbf{22.5} & \textbf{37.0} & \textbf{40.1} & \textbf{53.1}  \\
			\midrule
			&  \multicolumn{4}{c}{{\sf shaw}} &   \\
			\midrule
			\textsf{PNT}  & $0.014$ ($17$)   & $0.051$ ($16$)     & $0.158$ ($17$) & $0.293$ ($18$) & $0.479$ ($19$)  \\
			\textsf{Ch-PNT} & $0.051$ ($17$)   & $0.170$ ($16$)     & $0.438$ ($19$) & $0.873$ ($18$) & $1.819$ ($19$)  \\
			\textsf{Newton} &$0.455$ ($16$) & $3.675$ ($15$) & $8.904$ ($14$) & $26.835$ ($16$) & $52.768$ ($16$) \\
			\textbf{ratio-1}  & \textbf{3.6} & \textbf{3.3} & \textbf{2.8} & \textbf{3.0} & \textbf{3.8}  \\
			\textbf{ratio-2}  & \textbf{32.5} & \textbf{72.1} & \textbf{56.4} & \textbf{91.6} & \textbf{110.2}  \\
			\bottomrule[0.6pt]
	\end{tabular}}
	\label{tab5.2}
\end{table}

\begin{figure}[htbp]
	\centering
	\subfloat 
	{\label{fig:3.5a}\includegraphics[width=0.47\textwidth]{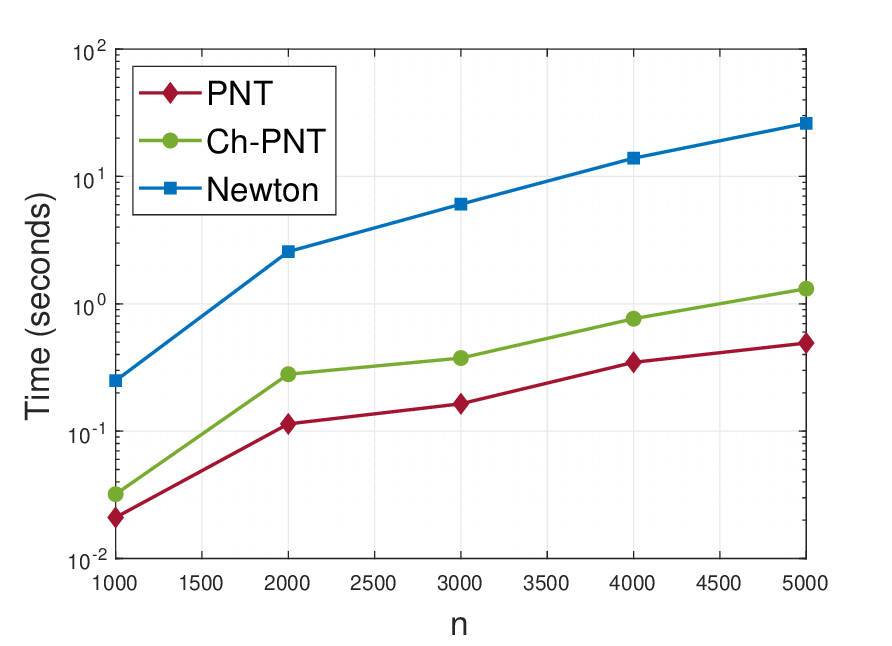}}\hspace{-1mm}
	\subfloat
	{\label{fig:3.5b}\includegraphics[width=0.45\textwidth]{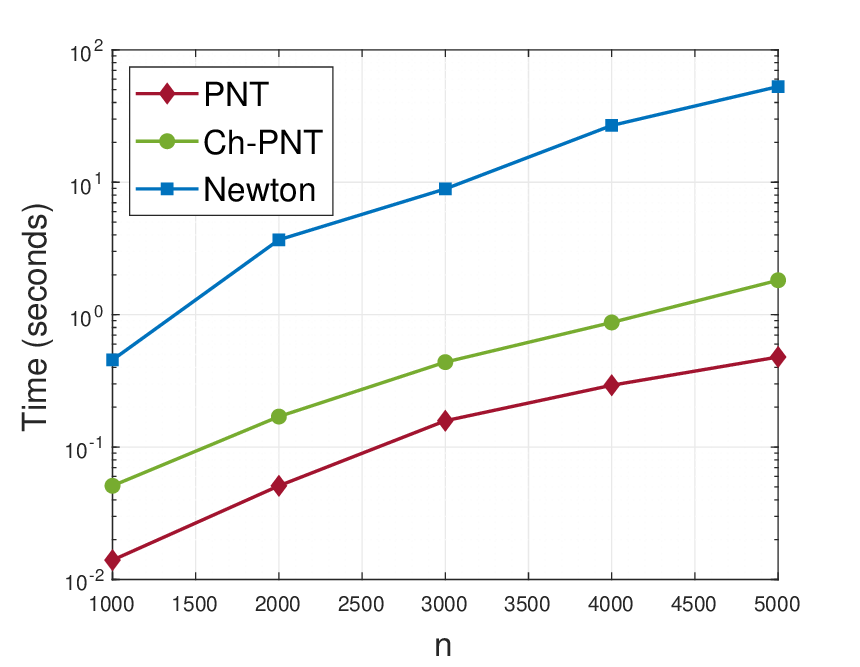}}\hspace{-1mm}	
	\vspace{-3mm}
	\caption{Comparison of scalability of \textsf{PNT}, \textsf{Ch-PNT} and \textsf{Newton} methods as the scale of the problems increasing from $n=1000$ to $n=5000$. Left: {\sf heat}. Right: {\sf shaw}.}
	\label{fig3.5}
\end{figure}

To show the advantage of the computational efficiency of \textsf{PNT} over \textsf{Ch-PNT} and \textsf{Newton}, we gradually increase the scale of the test problems and measure the running time of the three methods, where all of them stop at the first iteration such that $\left|\|\bA\bx_k-\bb\|_{\bM^{-1}}^2-\tau m\right|\leq 10^{-8}$. The time data are listed in \Cref{tab5.2}. We also compute the ratio of the running time, i.e. the value of \textsf{Ch-PNT}-time/\textsf{PNT}-time and \textsf{Newton}-time/\textsf{PNT}-time. For {\sf shaw}, we find that all three methods stop with similar iteration numbers, and the computational speed of \textsf{PNT} is much faster than \textsf{Newton}, with the speedup ratio varying from 41 to 157. For {\sf heat}, we find that \textsf{Newton} stops with only about half iteration numbers of \textsf{PNT}'s. However, the total running time of \textsf{PNT} is still much smaller than \textsf{Newton}'s, with the speedup ratio varying from 8 to 43. To compare the scalability of \textsf{PNT}, \textsf{Ch-PNT} and \textsf{Newton} more clearly, we use the data in \Cref{tab5.2} to plot the curve of time growth with respect to $n$. Clearly, \textsf{PNT} saves much more time compared to \textsf{Newton} while obtaining solutions with the same accuracy. Although the advantage of \textsf{PNT} over \textsf{Ch-PNT} is not significant for small-scale problems, \textsf{Ch-PNT} is not feasible for large-scale problems due to the prohibitive cost of Cholesky factorization.


\subsection{Large-scale problems}
We choose three 2D image deblurring and computed tomography inverse problems from \cite{Gazzola2019}. The first problem is {\sf PRblurshake}, which simulates a spatially invariant motion blur caused by the shaking of a camera. The second problem is {\sf PRblurspeckle}, which simulates a spatially invariant blur caused by atmospheric turbulence. The third problem is {\sf PRspherical} that models spherical means tomography. The true images and noisy observed data are shown in \Cref{fig4}, where all the images have $128\times 128$ pixels, and  $\bepsilon$ are uncorrelated non-white Gaussian noises with $\varepsilon=10^{-3}, \ 5\times 10^{-3}$ and $10^{-2}$, respectively. We have $m=n=128^2$ for the first two problems, and $m=23168,n=128^2$ for the third problem.

\begin{figure}[htbp]
	\centering
	\subfloat
	{\label{fig:4a}\includegraphics[width=0.3\textwidth]{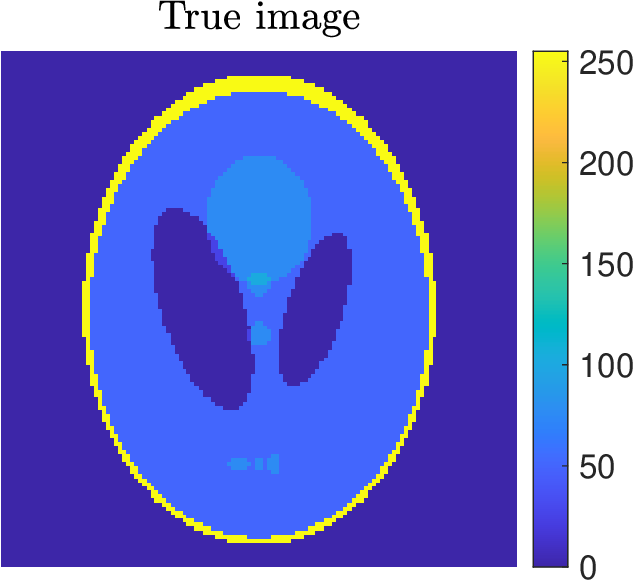}}\hspace{1.5mm}
	\subfloat
	{\label{fig:4b}\includegraphics[width=0.3\textwidth]{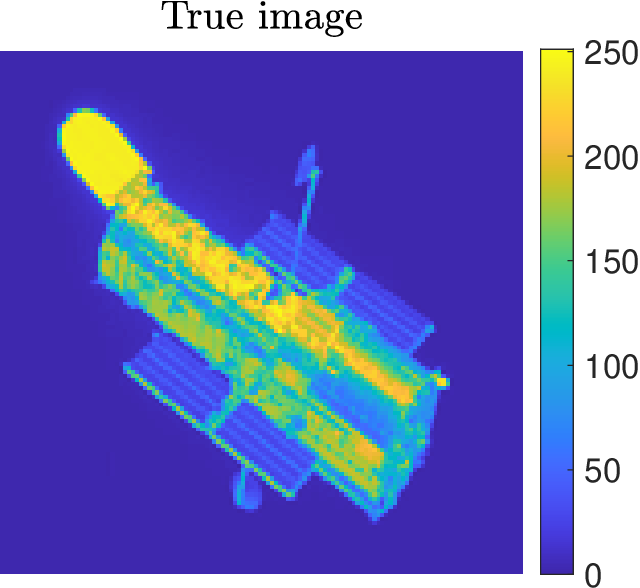}}\hspace{1.5mm}
	\subfloat
	{\label{fig:4c}\includegraphics[width=0.3\textwidth]{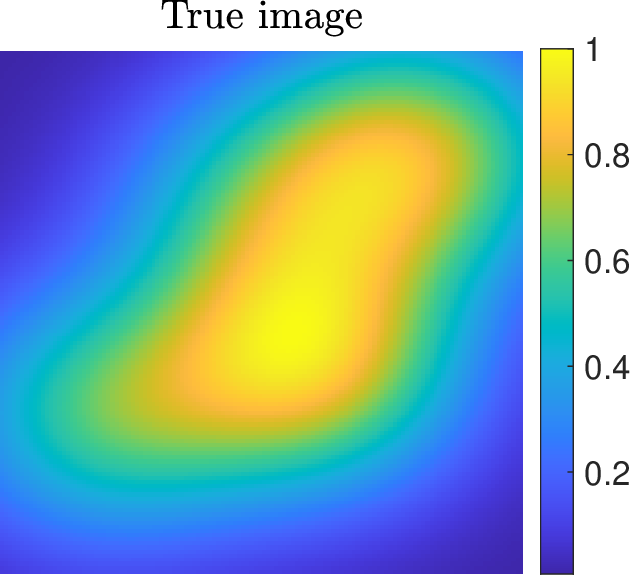}}
	\vspace{-2mm} 
	\subfloat
	{\label{fig:4d}\includegraphics[width=0.3\textwidth]{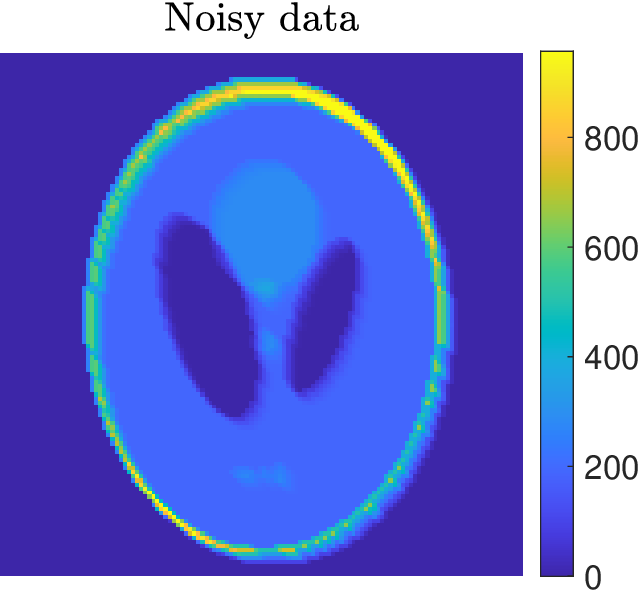}}\hspace{1mm}
	\subfloat
	{\label{fig:4e}\includegraphics[width=0.3\textwidth]{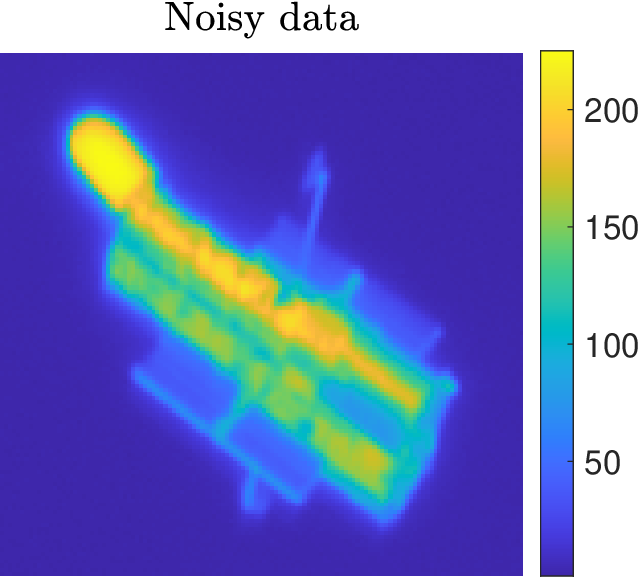}}\hspace{1.5mm}
	\subfloat
	{\label{fig:4f}\includegraphics[width=0.31\textwidth]{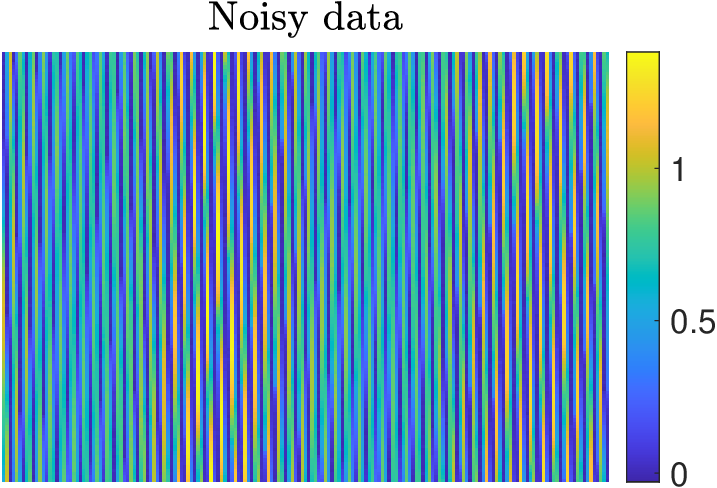}}
	\vspace{-2mm}
	\caption{True solution and noisy observed data for deblurring and tomography problems. From the leftmost column to the rightmost column are {\sf PRblurshake}, {\sf PRblurspeckle} and {\sf PRspherical}.}
	\label{fig4}
\end{figure}

For {\sf PRblurshake} and {\sf PRblurspeckle}, we construct $\bN$ using the Gaussian kernel with $l=10$ and $l=1$, respectively. For {\sf PRspherical}, we construct $\bN$ using the Mat\'{e}rn kernel
\[
	K_{M}(r):= \frac{2^{1-\nu}}{\Gamma(\nu)}\Bigg(\frac{\sqrt{2\nu}r}{l}\Bigg)^\nu B_\nu\Bigg(\frac{\sqrt{2\nu}r}{l}\Bigg),
\]
where $\Gamma(\cdot)$ is the gamma function, $B_{\nu}(\cdot)$ is the modified Bessel function of the second kind, and $l$ and $\nu$ are two positive parameters of the covariance; we set $l=100$ and $\nu=1.5$. 
For the three large-scale problems, it is almost impossible to get $(\mu_{opt}, \bx(\mu_{opt}))$ and $(\mu_{DP}, \bx(\mu_{DP}))$ by solving \cref{gen_regu}. The standard \textsf{Newton} method and the methods in \cite{cornelis2020projected1,cornelis2020projected2} can not be applied because these methods have to deal with $\bN^{-1}$. To test the performance of \textsf{PNT}, here we only compare it with \textsf{genHyb}. Additionally, we also implement \textsf{PNT-md} to demonstrate that it can save some computation compared to \textsf{PNT}.

\begin{figure}[htbp]
	\centering
	\subfloat
	{\label{fig:5a}\includegraphics[width=0.33\textwidth]{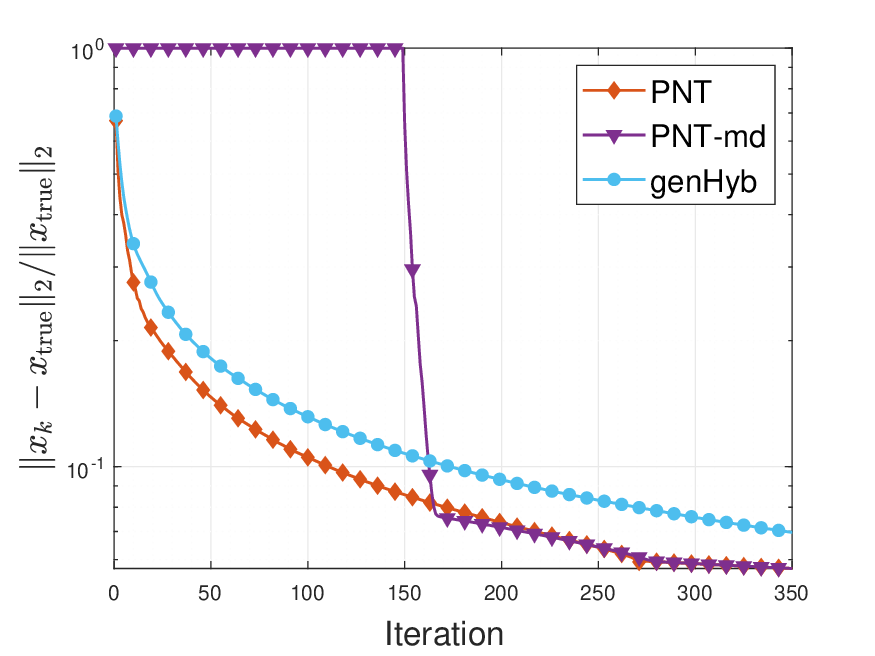}}\hspace{-2mm}
	\subfloat
	{\label{fig:5b}\includegraphics[width=0.33\textwidth]{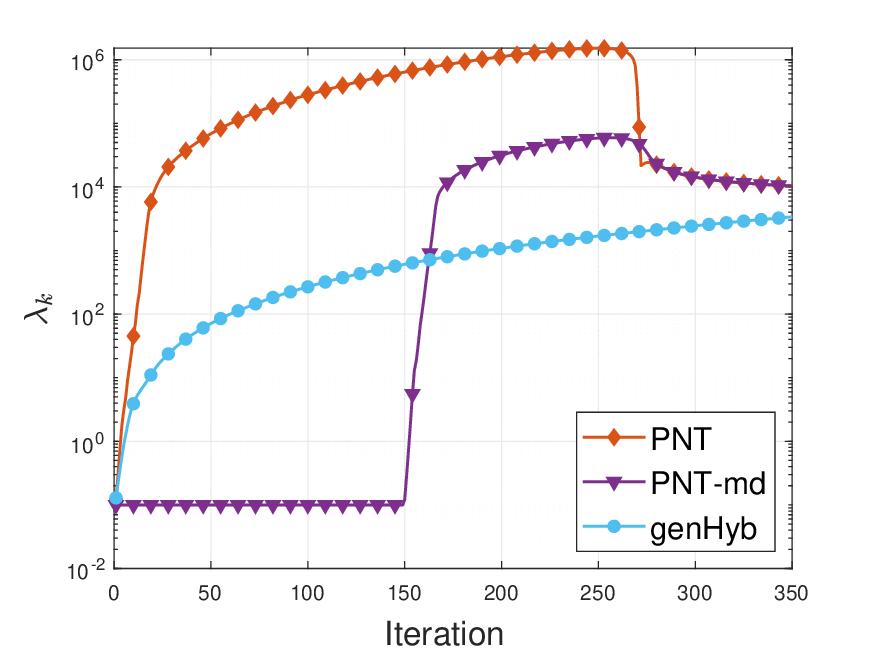}}\hspace{-2mm}
	\subfloat
	{\label{fig:5c}\includegraphics[width=0.33\textwidth]{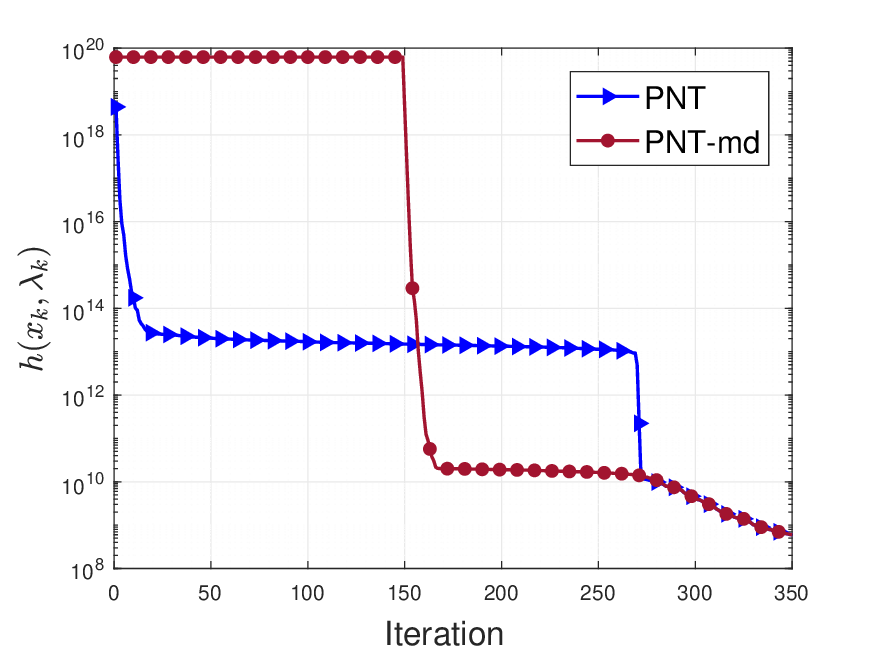}}
	\vspace{-4mm} 
	\subfloat
	{\label{fig:5d}\includegraphics[width=0.33\textwidth]{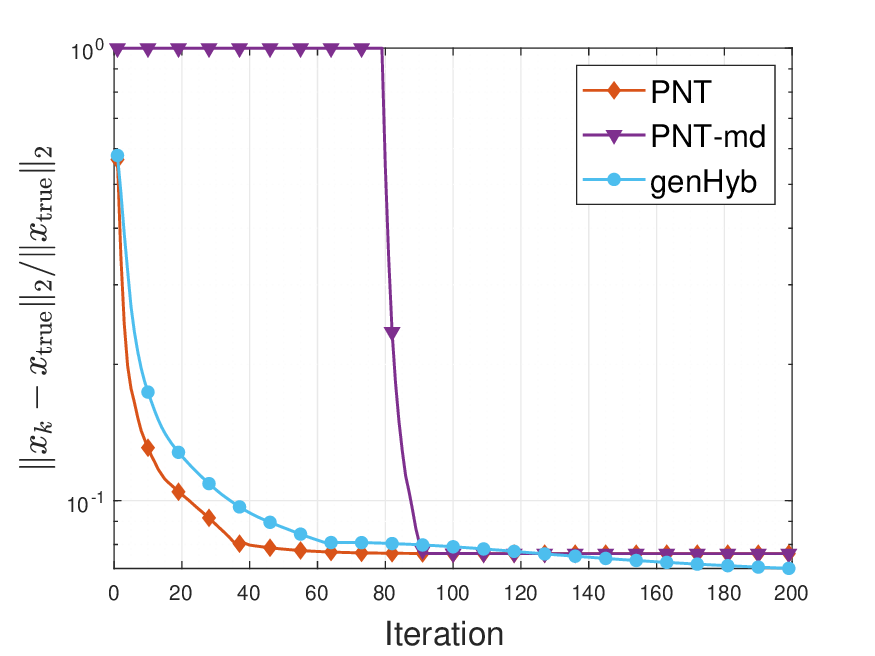}}\hspace{-2mm}
	\subfloat
	{\label{fig:5e}\includegraphics[width=0.33\textwidth]{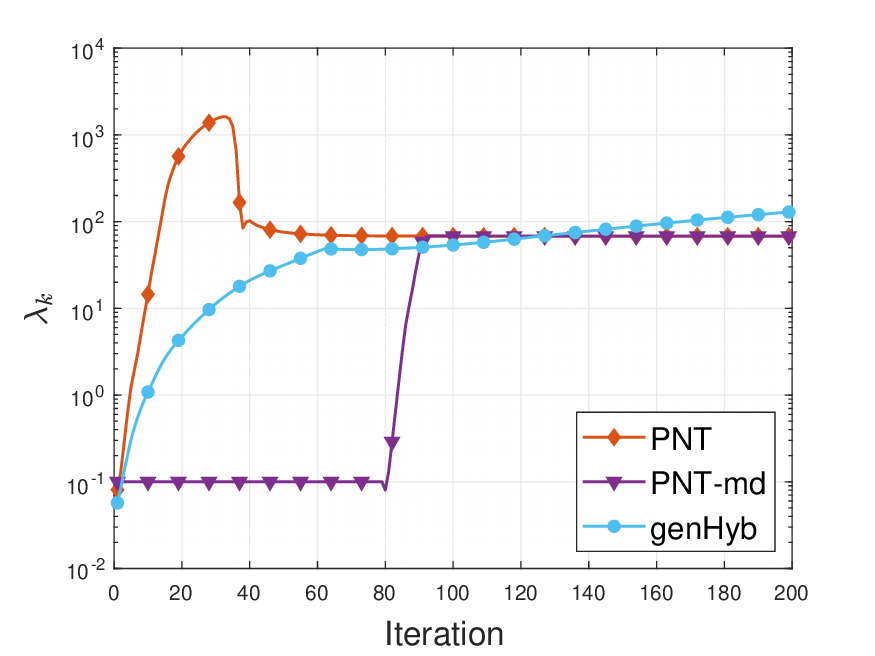}}\hspace{-2mm}
	\subfloat
	{\label{fig:5f}\includegraphics[width=0.33\textwidth]{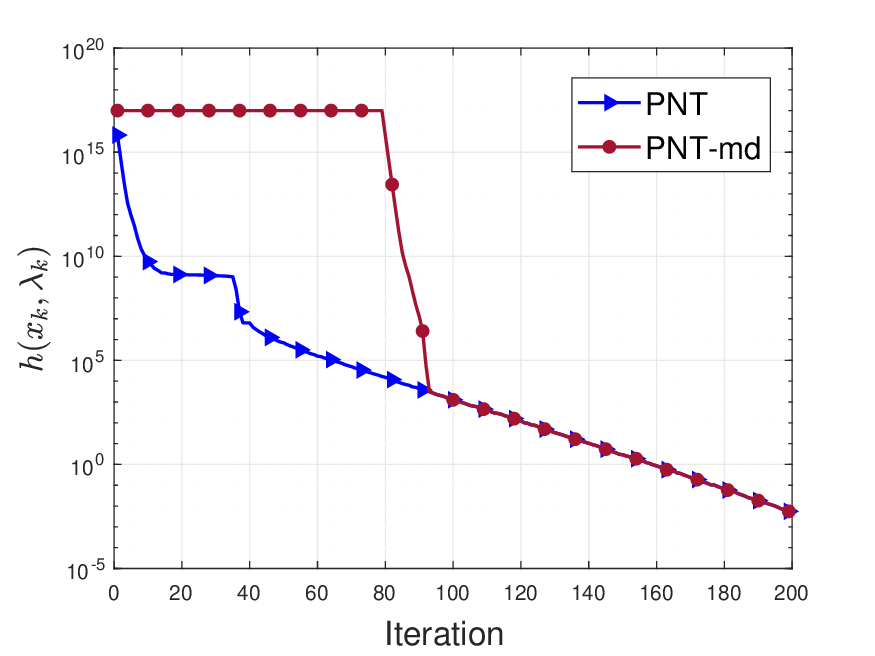}}
	\vspace{-4mm} 
	\subfloat
	{\label{fig:5g}\includegraphics[width=0.33\textwidth]{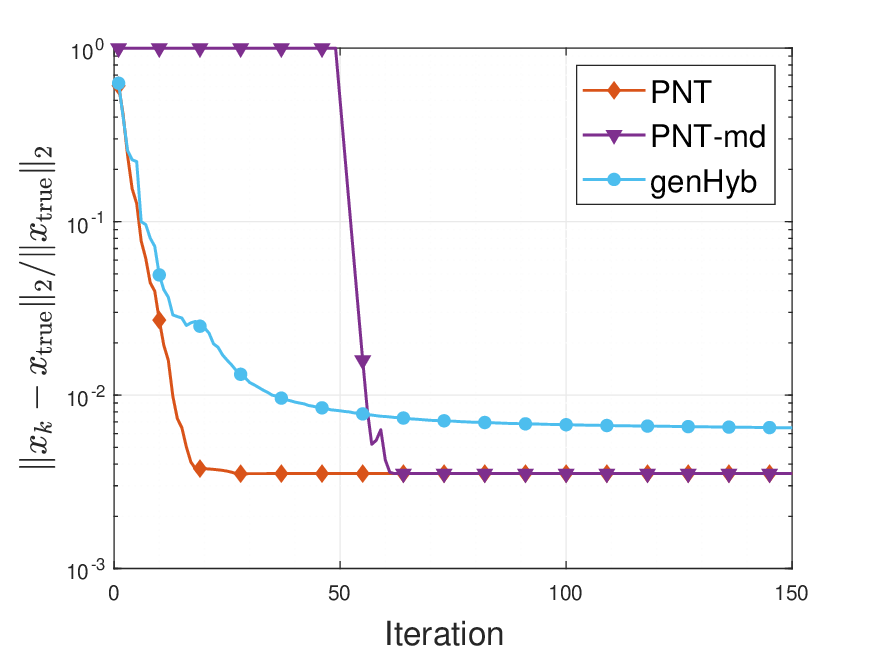}}\hspace{-2mm}
	\subfloat
	{\label{fig:5h}\includegraphics[width=0.33\textwidth]{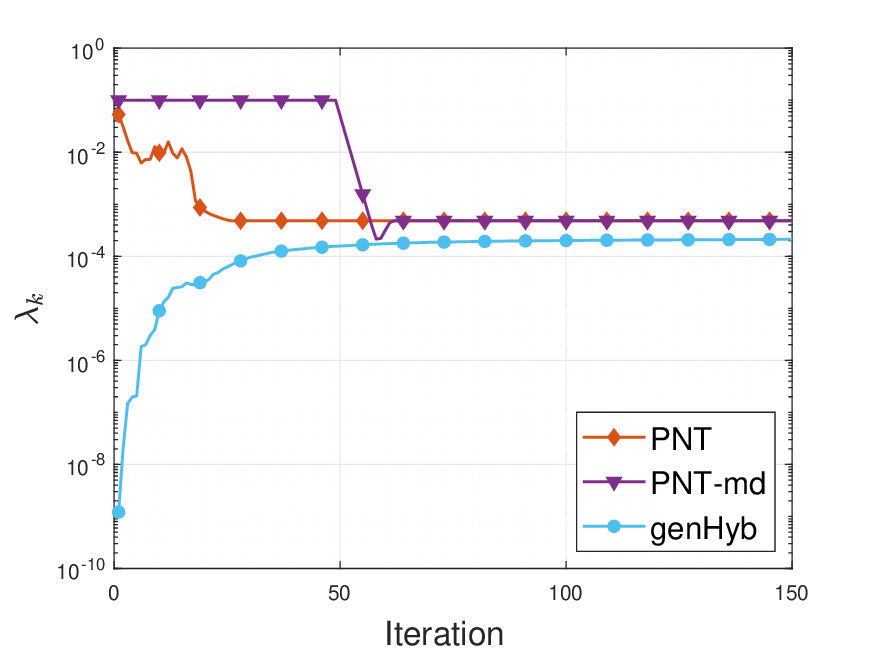}}\hspace{-2mm}
	\subfloat
	{\label{fig:5i}\includegraphics[width=0.33\textwidth]{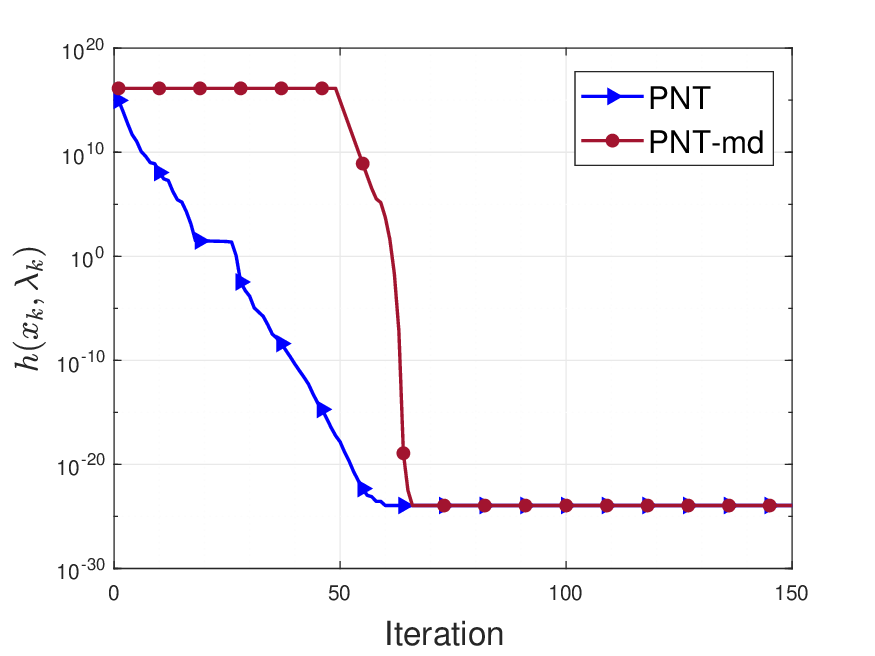}}
	\vspace{-3mm} 
	\caption{Relative errors of iterative solutions, convergence of $\lambda_k$, and convergence of merit functions. Top: {\sf PRblurshake}. Middle: {\sf PRblurspeckle}. Bottom: {\sf PRspherical}.}
	\label{fig5}
\end{figure}

\begin{figure}[htbp]
	\centering
	\subfloat
	{\label{fig:6a}\includegraphics[width=0.3\textwidth]{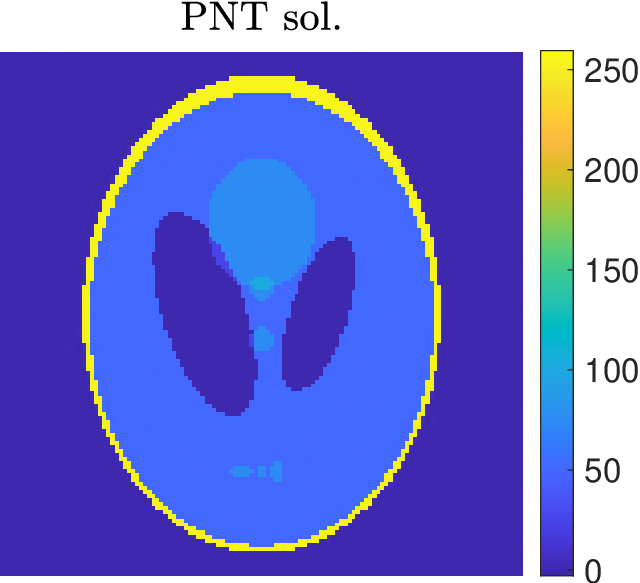}}\hspace{1.5mm}
	\subfloat
	{\label{fig:6b}\includegraphics[width=0.3\textwidth]{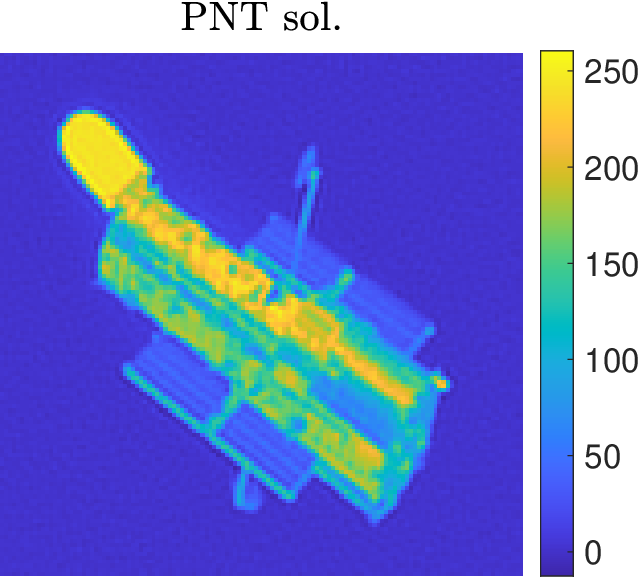}}\hspace{1.5mm}
	\subfloat
	{\label{fig:6c}\includegraphics[width=0.29\textwidth]{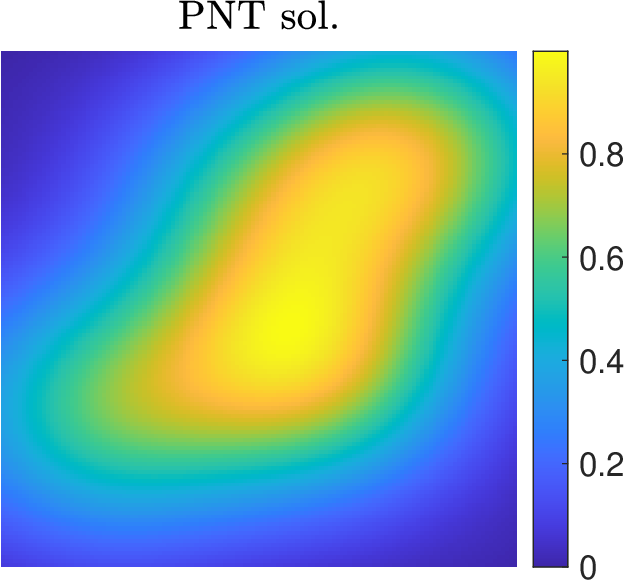}}
	\vspace{-2mm} 
	\subfloat
	{\label{fig:6d}\includegraphics[width=0.3\textwidth]{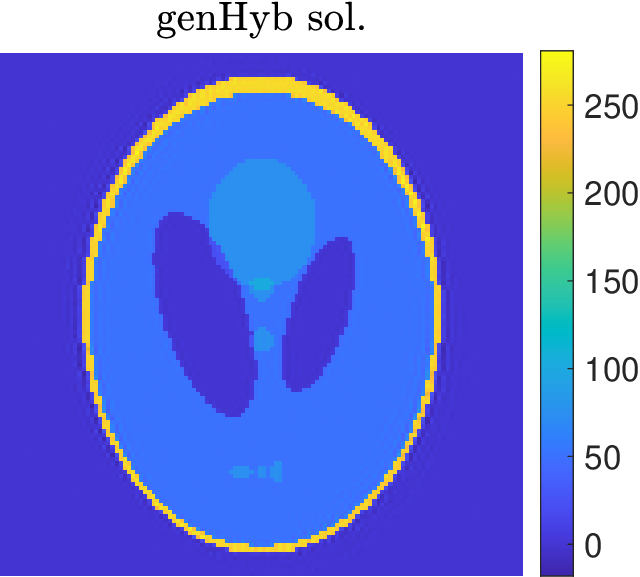}}\hspace{1.5mm}
	\subfloat
	{\label{fig:6e}\includegraphics[width=0.3\textwidth]{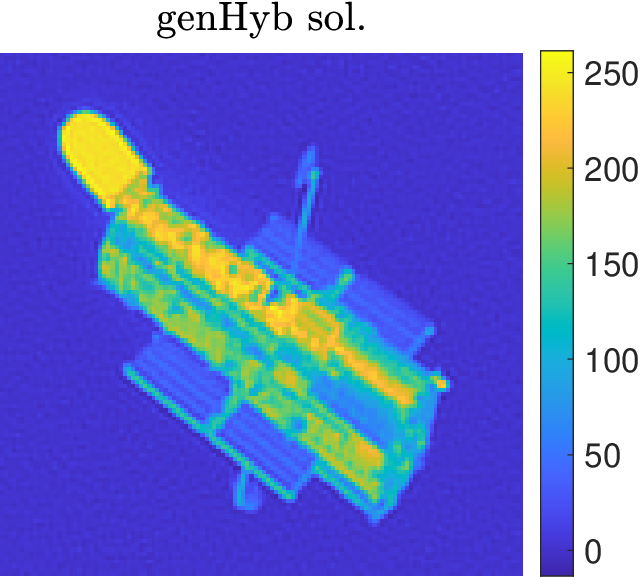}}\hspace{1.5mm}
	\subfloat
	{\label{fig:6f}\includegraphics[width=0.29\textwidth]{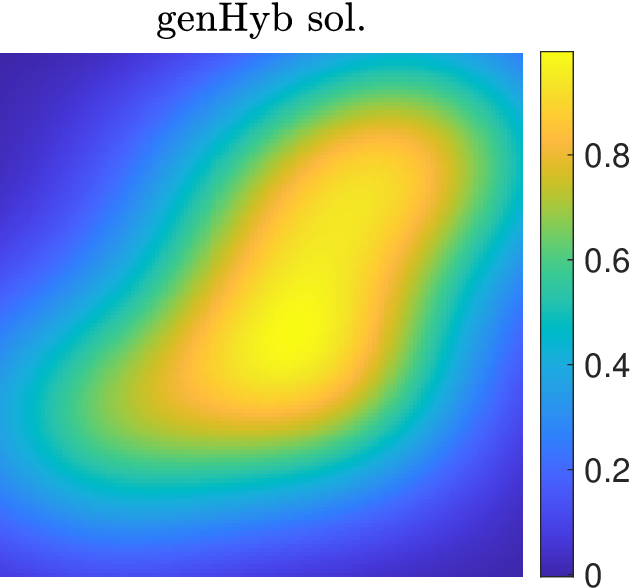}}
	\vspace{-2mm}
	\caption{Reconstructed solutions at the final iterations by \textsf{PNT} and \textsf{genHyb}. From the leftmost column to the rightmost column are {\sf PRblurshake}, {\sf PRblurspeckle} and {\sf PRspherical}.}
	\label{fig6}
\end{figure}

The relative error curves of the three methods, the convergence curves of $\lambda_k$ and $h(\bx_k,\lambda_k)$ are plotted in \Cref{fig5}. For \textsf{PNT-md}, we set $k_0=150,80,50$ for the three problems, respectively. It can be observed that \textsf{PNT} for the last two problems converges very fast: the variations of relative error and $\lambda_k$ become quickly stabilized after 50 to 150 iterations, although for the second problem $h(\bx_k,\lambda_k)$ are still decreasing significantly after 200 iterations. The \textsf{genHyb} method for {\sf PRblurshake} and {\sf PRspherical} converges slower, and it obtains two solutions with larger relative errors than that of \textsf{PNT}. This is because \textsf{genHyb} under-estimates $\lambda$ more than \textsf{PNT}. For all three problems, \textsf{PNT-md} converges very quickly from $k_0$, achieving solutions with the same accuracy as \textsf{PNT} while requiring nearly the same total number of iterations. The reconstructed images are shown in \Cref{fig6}, which reveals the effectiveness of \textsf{PNT} and \textsf{genHyb}. The variation of the condition number of $J^{(k)}(\bar{\by}_{k-1},\lambda_{k-1})$ is shown in \Cref{fig65}. The condition number does not grow very large during the iteration, allowing the small-scale linear system \cref{proj_Newton1} to be solved directly without issues.

\begin{figure}[htbp]
	\centering
	\subfloat 
	{\label{fig:65a}\includegraphics[width=0.33\textwidth]{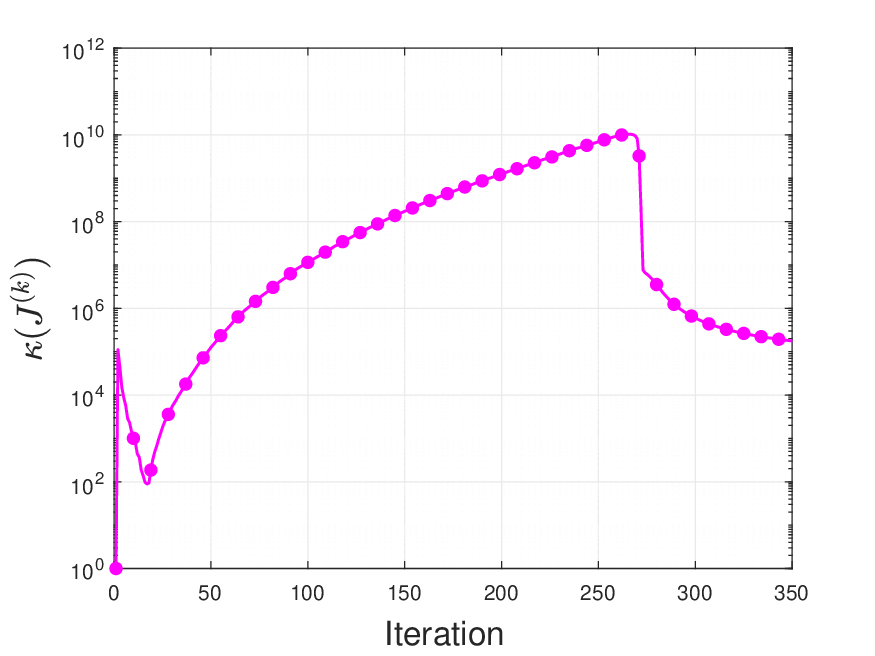}}\hspace{-3mm}
	\subfloat
	{\label{fig:65b}\includegraphics[width=0.33\textwidth]{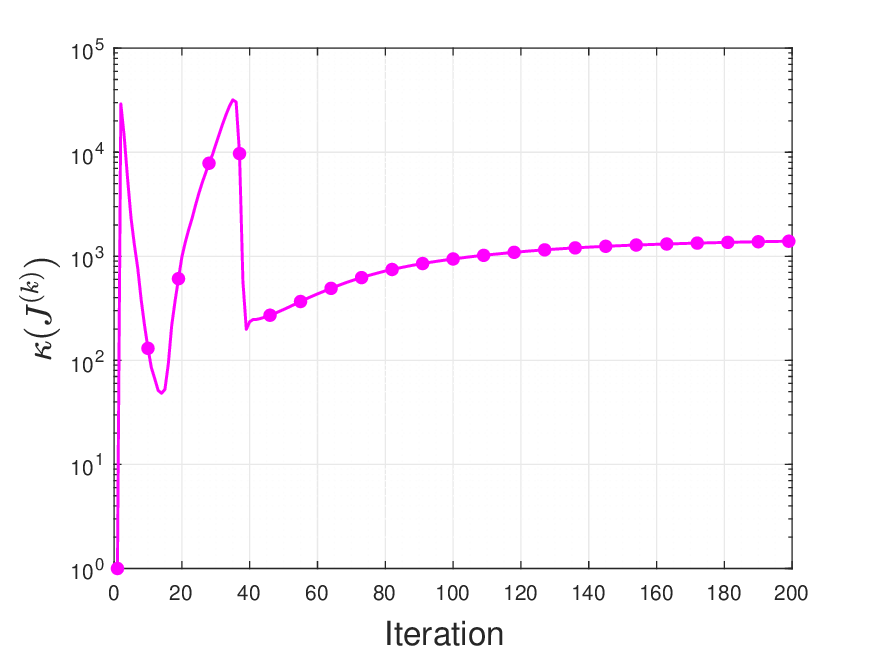}}\hspace{-3mm}
	\subfloat
	{\label{fig:65c}\includegraphics[width=0.34\textwidth]{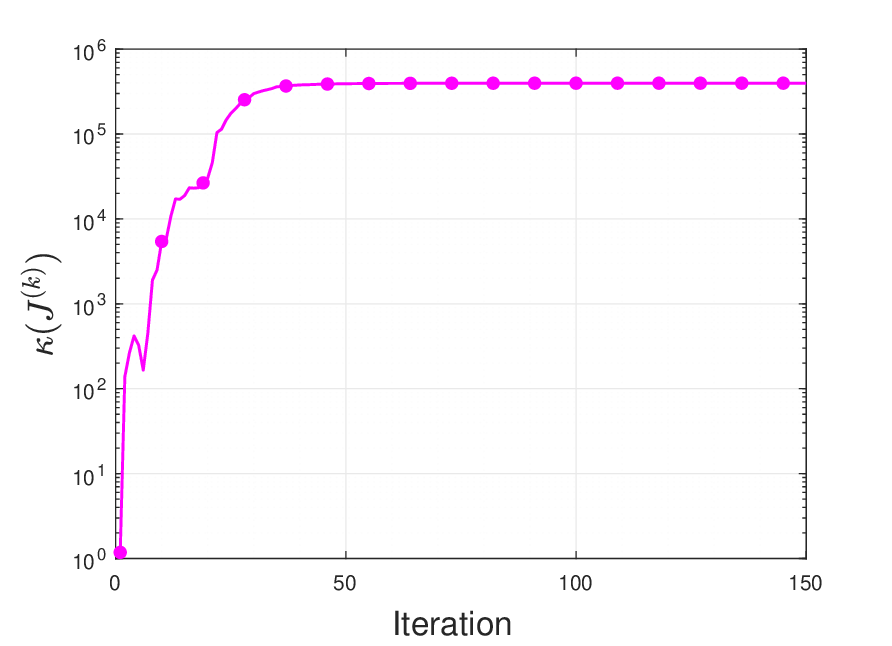}}
	\vspace{-3mm}
	\caption{Variation of the condition number of $J^{(k)}(\bar{\by}_{k-1},\lambda_{k-1})$ during the iteration of \textsf{PNT}. Left: {\sf PRblurshake}. Middle: {\sf PRblurspeckle} Right: {\sf PRspherical}.}
	\label{fig65}
\end{figure}

\begin{figure}[htbp]
	\centering
	\subfloat
	{\label{fig:7a}\includegraphics[width=0.33\textwidth]{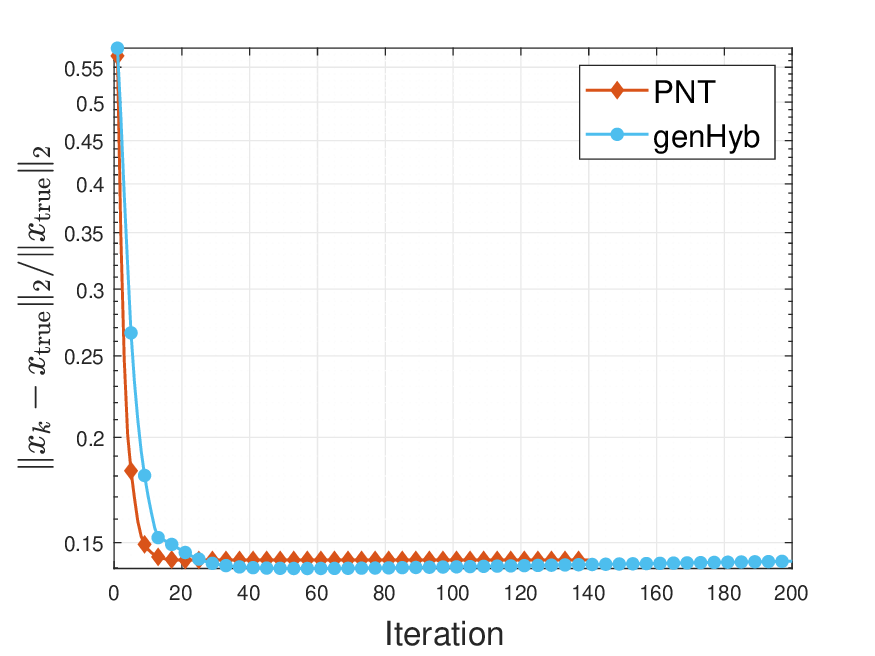}}\hspace{-4.5mm}
	\subfloat
	{\label{fig:7b}\includegraphics[width=0.33\textwidth]{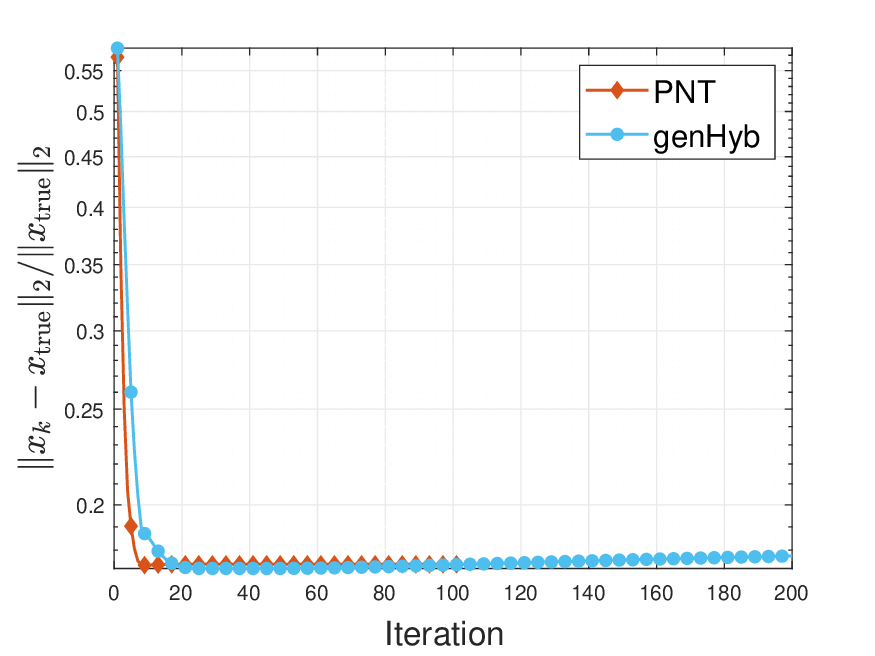}}\hspace{-4.5mm}
	\subfloat
	{\label{fig:7c}\includegraphics[width=0.33\textwidth]{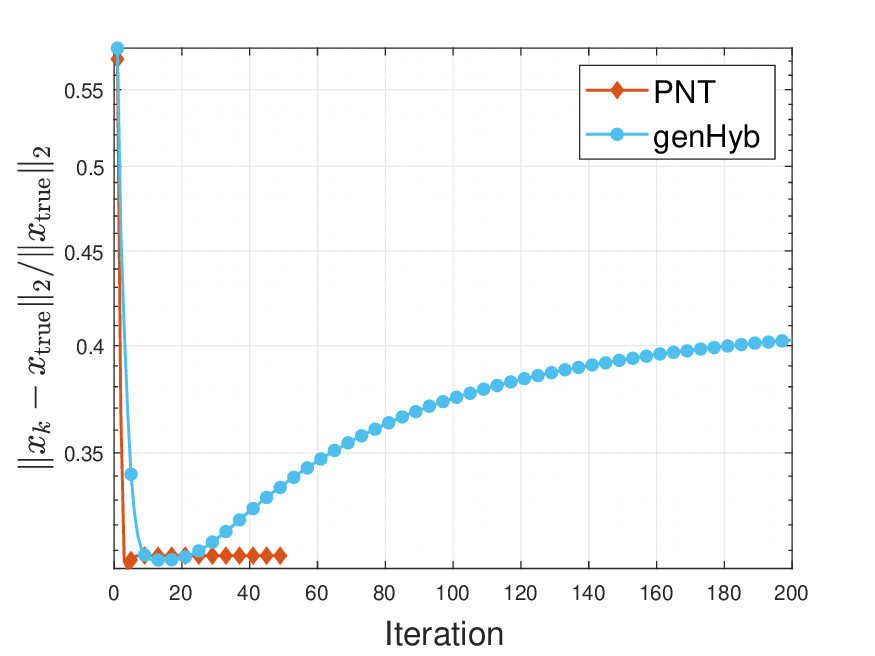}}
	\vspace{-4mm} 
	\subfloat
	{\label{fig:7d}\includegraphics[width=0.33\textwidth]{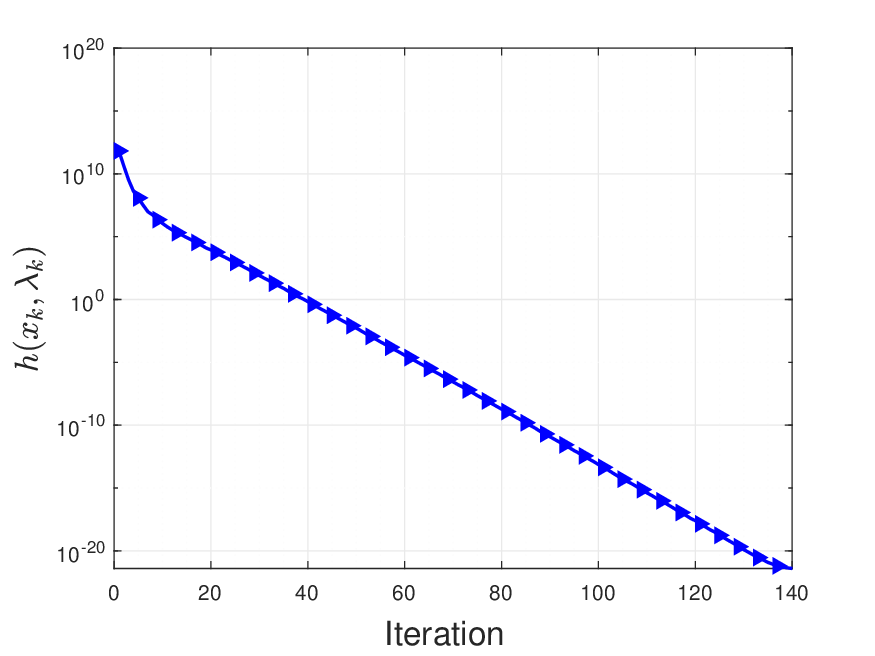}}\hspace{-4.5mm}
	\subfloat
	{\label{fig:7e}\includegraphics[width=0.33\textwidth]{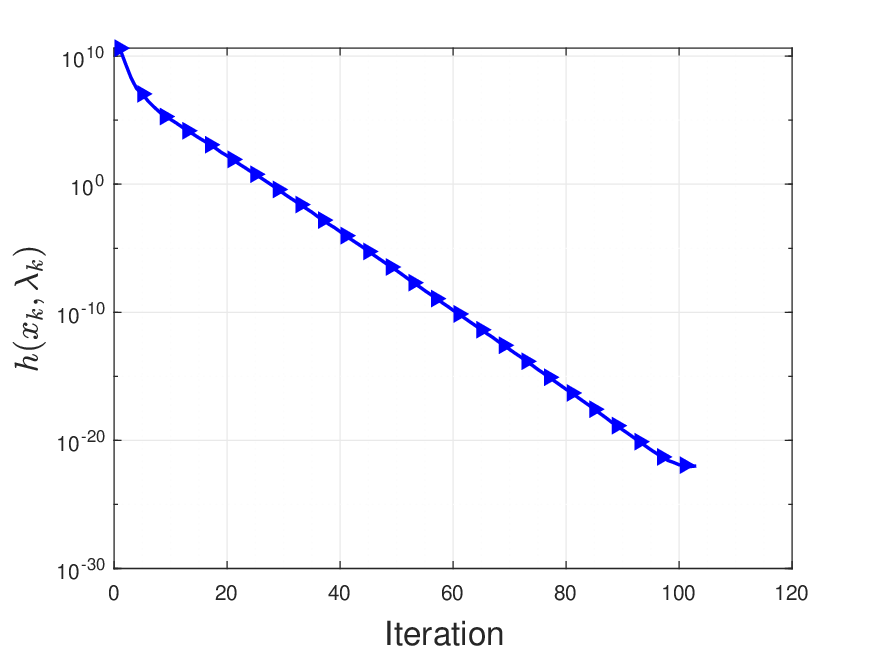}}\hspace{-4.5mm}
	\subfloat
	{\label{fig:7f}\includegraphics[width=0.33\textwidth]{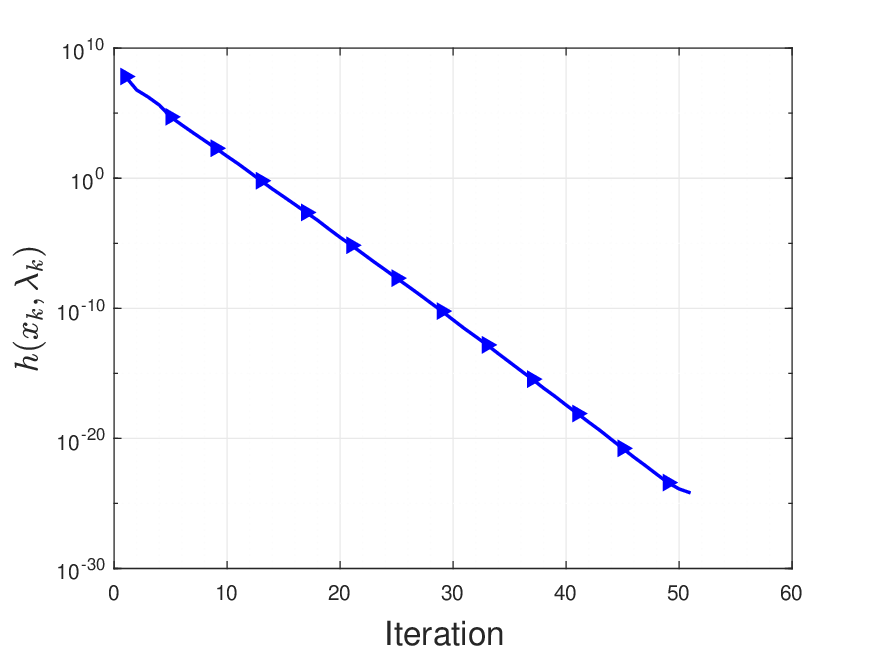}}
	\vspace{-3mm} 
	\caption{Relative errors of iterative solutions by \textsf{PNT} and \textsf{genHyb}, and the decrease of $h(\bx_k,\lambda_k)$. The test problem is {\sf PRblurspeckle}. From the leftmost column to the rightmost column, the noise levels are $\varepsilon=5\times10^{-2}, \ 10^{-1}, \ 5\times10^{-1}$.}
	\label{fig7}
\end{figure}

To further test the robustness of \textsf{PNT} and \textsf{genHyb} as the noise level gradually increases, we set the noise level of {\sf PRblurspeckle} to be $\varepsilon=5\times10^{-2}, \ 10^{-1}, \ 5\times10^{-1}$. \Cref{fig7} shows the corresponding relative error curves and the curves of $h(\bx_k,\lambda_k)$. We can find that, when the noise is not very big, both \textsf{PNT} and \textsf{genHyb} converge stably with almost the same accuracy. However, when the noise gradually increases, the situations are very different. First, we find that as the noise increases, \textsf{PNT} still converges stably, and faster. Second, $h(\bx_k,\lambda_k)$ can always decrease to an extremely small value, which is promised by \Cref{thm:conv}. The iterate of \textsf{PNT} stops when the step-length $\gamma_k$ becomes too small (less than $10^{-16}$), which happens more early if the noise is bigger. In comparison, the convergence of \textsf{genHyb} becomes unstable as the noise increases. For $\varepsilon=10^{-1}$, it can be observed that the relative error for \textsf{genHyb} increases slightly after a certain iteration, while for $\varepsilon=5\times10^{-1}$, the increase of relative error appears earlier and more clearly. This is a typical potential weakness of hybrid regularization methods, a challenge that the \textsf{PNT} method successfully addresses.

\section{Conclusion}\label{sec6}
For large-scale Bayesian linear inverse problems, we have proposed the projected Newton (\textsf{PNT}) method as a novel iterative approach for simultaneously updating both the regularization parameter and solution without any computationally expensive matrix inversions or decompositions. By reformulating the Tikhonov regularization as a corresponding constrained minimization problem and leveraging its Lagrangian function, the regularized solution and the corresponding Lagrangian multiplier can be obtained from the unconstrained Lagrangian function using a Newton-type method. To reduce the computational overhead of the Newton method, the generalized Golub-Kahan bidiagonalization is applied to project the original large-scale problem to become small-scale ones, where the projected Newton direction is obtained by solving the small-scale linear system at each iteration. We have proved that the projected Newton direction is a descent direction of a merit function, and the points generated by \textsf{PNT} eventually converge to the unique minimizer of this merit function, which is just the regularized solution and the corresponding Lagrangian multiplier.

Experimental tests on both small and large-scale Bayesian inverse problems have demonstrated the excellent convergence property, robustness and efficiency of \textsf{PNT}. The most demanding computational tasks in \textsf{PNT} are primarily matrix-vector products, making it particularly well-suited for large-scale problems.

An important remaining question is the convergence rate of \textsf{PNT}, i.e., how fast the three quantities $\|\bx_{k}-\bx_{DP}\|_2$, $|\lambda_{k}-\lambda_{DP}|$ and $h(\bx_{k},\lambda_k)$ converge to zero? The convergence rate may depend on several factors, including the ill-posedness of \cref{inverse1}, the smoothness of the true solution, the noise level, and the properties of $\{\bA,\bM,\bN\}$. We will conduct theoretical investigations into this issue in future work.

\section*{Acknowledgments}
The author thanks Dr. Felipe Atenas for helpful discussions. 

\bibliographystyle{siamplain}
\bibliography{references}

\end{document}